\documentclass [draft]{amsart}

\usepackage{amsrefs}

\usepackage{amsmath,amsfonts,amssymb,graphicx,latexsym,color}

\newtheorem{theorem}{Theorem}

\newtheorem{lemma}[theorem]{Lemma}

\newtheorem{proposition}{Proposition}

\theoremstyle{definition}

\newcommand{\eqdef}{\overset{\mbox{\tiny{def}}}{=}}

\newcommand{\R}{{\mathbb R}}
\newcommand{\NgE}{{N \ge 3}}
\newcommand{\ang}[1]{ \left< {#1} \right> }
\newcommand{\ext}[1]{ \underline{ {#1} } }
\newcommand{\ind}{ {\mathbf 1}}
\newcommand{\sph}{{\mathbb S}^2}
\newcommand{\nsm}{|}
\newcommand{\nl}{\left|}

\newcommand{\nr}{\right|}
\newcommand{\set}[2]{ \left\{ #1 \ \left| \ #2 \right. \right\} }
\newcommand{\ip}[2]{ \left< #1 , #2 \right>}

\begin{document}
\title[Global Solutions of the Boltzmann Equation without Cutoff]{Global Strong Solutions of the Boltzmann Equation without  Angular Cut-off}

\author[P. T. Gressman]{Philip T. Gressman}
\address{University of Pennsylvania, Department of Mathematics, David Rittenhouse Lab, 209 South 33rd Street, Philadelphia, PA 19104-6395, USA \bigskip} 
\email{gressman at math.upenn.edu}
\urladdr{http://www.math.upenn.edu/~gressman/}
\thanks{P.T.G. was partially supported by the NSF grant DMS-0850791.}

\author[R. M. Strain]{Robert M. Strain}
%\address{(RMS) University of Pennsylvania, Department of Mathematics, David RittenhouseLab, 209 South 33rd Street, Philadelphia, PA 19104-6395, USA} 
\email{strain at math.upenn.edu}
\urladdr{http://www.math.upenn.edu/~strain/}
\thanks{R.M.S. was partially supported by the NSF grant DMS-0901463.}

\dedicatory{}

%\subjclass[2000]{}
%\subjclass[2010]{35Q20, 35R11, 76P05, 82C40, 35H20, 35B65, 26A33}

%\date{\today, \Blue{REMOVE DATE for arXiv}} %, \Red{FIX 2010 MSC not 1991}}

\begin{abstract}
We prove the existence and exponential decay of global in time strong solutions to the  Boltzmann equation without any angular cut-off, i.e., for long-range interactions.
We consider perturbations of the Maxwellian equilibrium states and include  the physical cross-sections
arising from an inverse-power intermolecular potential $r^{-(p-1)}$ with $p>3$, and more generally, the full range of angular singularities $s=\nu/2 \in(0,1)$.  
These appear to be the first unique global solutions to this fundamentally important model, which grants a basic example where a range of geometric fractional derivatives occur in a physical model of the natural world. Our methods provide a new  understanding of the effects of grazing collisions in the Boltzmann theory. 
\end{abstract}
\maketitle

\thispagestyle{empty}

%\textbf{Mathematics Subject Classification (2010)}: 	35Q20, 35R11, 76P05, 82C40, 35H20, 35B65, 26A33

%\textbf{Keywords}: Kinetic Theory, Boltzmann equation, long-range interaction, non-cutoff, soft potentials, hard potentials, fractional derivatives. 

\tableofcontents  

\section{Introduction, main theorem, and historical remarks}

In 1872, Boltzmann was able to derive an equation which accurately models the dynamics of a dilute gas; it has since become a cornerstone of statistical physics \cite{MR1313028,MR1307620, MR1014927,MR0156656, MR1942465,MR1379589}.
There are many useful mathematical theories of global solutions for the Boltzmann equation,  and we will start off by  mentioning a brief few.  In 1933, Carleman \cite{MR1555365} proved existence and uniqueness of the spatially homogeneous problem with radial initial data.  For the spatially dependent theories, it was Ukai \cite{MR0363332} in 1974  who proved the existence of global classical solutions with close to equilibrium initial data.  Ten years later,
 Illner-Shinbrot 
\cite{MR760333}  found unique global mild solutions with near vacuum data.
Then in 1989, the work of DiPerna-Lions \cite{MR1014927} established global renormalized weak solutions for initial data without a size restriction.  We also mention recent methods introduced in the linearized regime by Guo \cite{MR2000470} in 2003 and Liu-Yang-Yu \cite{MR2043729} in 2004.
All of these methods and their generalizations apply to %the regime of 
hard sphere particles 
%interactions 
or  soft particle interactions in which there is a non-physical cut-off of an inherently nonintegrable angular singularity.

When the physically relevant effects of these angular singularities are not cut-off, 
the only global spatially dependent theory we are aware of is the remarkable paper by Alexandre-Villani \cite{MR1857879} from 2002, 
which proves the existence of
 DiPerna-Lions renormalized weak solutions \cite{MR1014927}
 if one can add to the equation a non-negative defect measure.  It is illustrated therein that the mass conservation they prove would imply this defect measure was zero if the solutions were sufficiently regular.  
% At now the defect measure appears difficult to characterize \cite[Appendix]{MR1857879}.
 At the moment this defect measure appears difficult to characterize \cite[Appendix]{MR1857879}.

Despite the well-known physical and mathematical importance of this problem, it is perhaps the last remaining 
 physically relevant
case in the Boltzmann theory in which, as far as we know, there is no theory of global in time strong solutions for spatially dependent initial data of any kind.  This issue is mentioned particularly in Villani \cite{MR1942465}.
The results herein prove the existence of such solutions for cross sections arising from an an inverse-power intermolecular potential $r^{-(p-1)}$ with $p>3$ and, more generally, for  the full range of angular singularities.

Let us now give a detailed explanation. 
We study  the {\em Boltzmann equation}
  \begin{equation}
  \frac{\partial F}{\partial t} + v \cdot \nabla_x F = {\mathcal Q}(F,F),
  \label{BoltzFULL}
  \end{equation}
where the unknown $F(t,x,v)$ is a nonnegative function.  For each time $t\ge 0$, $F(t, \cdot, \cdot)$ represents the density function of particles in the phase space; some may call $F$ the empirical measure. The spatial coordinates we consider are $x\in\mathbb{T}^3$, and the velocities are $v\in\mathbb{R}^3$.
The  {\em Boltzmann collision operator} ${\mathcal Q}$ 
is a bilinear operator which acts only on the velocity variables $F(v)$ and is local in $(t,x)$, as 
  \begin{equation*}
  {\mathcal Q} (F,F)(v) = 
  \int_{\mathbb{R}^3}  dv_* 
  \int_{\mathbb{S}^{2}}  d\sigma~ 
  B(|v-v_*|, \sigma) \, 
  \big[ F'_* F' - F_* F \big].
  \end{equation*} 
Here we are using the standard shorthand $F = F(v)$, $F_* = F(v_*)$, $F' = F(v')$, 
$F_*^{\prime} = F(v'_*)$. 
In this expression, $v'$, $v' _*$ and $v$, $v_*$ are 
the velocities of a pair of particles before and after collision, and are connected through the following formulas
  \begin{equation}
  v' = \frac{v+v_*}{2} + \frac{|v-v_*|}{2} \sigma, \qquad
  v'_* = \frac{v+v_*}{2} - \frac{|v-v_*|}{2} \sigma,
  \qquad \sigma \in \mathbb{S}^{2}.
  \label{sigma}
  \end{equation}
There are other ways to represent ${\mathcal Q}$ which result from alternate choices for the parameterization of the set of solutions to the physical law of elastic collisions:
\begin{equation}
\begin{array}{rcl}
v+v_* &=& v^{\prime } + v^{\prime }_*,
\\
|v|^2+|v_*|^2 &=& |v^\prime|^2+|v_*^\prime|^2.
\end{array}
\label{energy}
\end{equation}
We specifically  discuss Carleman-type representations 
 in the appendix of this paper.

 The {\em Boltzmann collision kernel} $B(|v-v_*|, \sigma)$ for a monatomic gas is, on physical grounds, a non-negative function which 
only depends on the {\em relative velocity} $|v-v_*|$ and on
the {\em deviation angle}  $\theta$ through 
$\cos \theta = \ip{ k }{ \sigma}$ where $k = (v-v_*)/|v-v_*|$ and $\langle \cdot, \cdot \rangle$ is the usual scalar product in $\mathbb{R}^3$. 
%This is equivalently   
%$$
%\cos \theta = \ip{v'_*-v'}{ v_*-v} /|v_*-v|^2.
%$$ 
Without loss of generality we may assume that  $B(|v-v_*|, \sigma)$
is supported on $\ip{ k }{ \sigma} \ge 0$, i.e. $0 \le \theta \le \frac{\pi}{2}$.
Otherwise we can  reduce to this situation %by replacing the collision kernel 
with the following ``symmetrization'': 
$$
\overline{B}(|v-v_*|, \sigma)
=
\left[ 
B(|v-v_*|, \sigma)
+
B(|v-v_*|, -\sigma)
\right]
{\bf 1}_{\ip{ k }{ \sigma} \ge 0}.
$$ 
Above and generally, ${\bf 1}_{A}$ is the usual indicator function of the set $A$.

\subsection{The Collision Kernel}

Our assumptions are as follows:  
 \begin{itemize}
 \item 
We suppose that $B(|v-v_*|, \sigma)$ takes product form in its arguments as
\begin{equation}
B(|v-v_*|, \sigma) =\Phi( |v-v_*| ) \, b(\cos \theta).
\notag
\end{equation}
% For the hard-sphere kernel $\gamma = b = 1$.
In general both $b$ and $\Phi$ are non-negative functions. 
 
  \item The angular function $t \mapsto b(t)$ is non-locally integrable; for $c_b >0$ it satisfies
\begin{equation}
\frac{c_b}{\theta^{1+2s}} 
\le 
\sin \theta  b(\cos \theta) 
\le 
\frac{1}{c_b\theta^{1+2s}},
\quad
s \in (0,1),
\quad
   \forall \, \theta \in \left(0,\frac{\pi}{2} \right].
   \label{kernelQ}
\end{equation}
 
  \item The kinetic factor $z \mapsto \Phi(|z|)$ satisfies for some $C_\Phi >0$
\begin{equation}
%\frac{1}{C_\Phi}  |v-v_*|^\gamma \le
\Phi( |v-v_*| ) =  C_\Phi  |v-v_*|^\gamma , \quad \gamma > -\min \{ 2s, 3/2\}.
\label{kernelP}
\end{equation}
Notice that $\gamma + 2s \ge 0$ ensures a spectral gap from \cite{MR2322149}.
 \end{itemize}

Our main physical motivation is derived from of particles interacting according to a spherical intermolecular repulsive potential of the form
$$
\phi(r)=r^{-(p-1)} , \quad p \in (2,+\infty).
$$
For these potentials, Maxwell in 1867 showed that 
the kernel $B$ can be computed.  It satisfies the conditions above with $\gamma = (p-5)/(p-1)$ 
and $s = 1/(p-1)$; 
see for instance \cite{MR1313028,MR1307620,MR1942465}.
In this situation 
$$
\gamma + 2s = \frac{p-3}{p-1} > 0, 
\quad
\text{or}
\quad
p >  3.
$$
Thus all of the conditions in  \eqref{kernelQ} and \eqref{kernelP} hold
with 
$-1 < \gamma < 1$ and 
$0 < s < 1/2$.  Some authors use the notation 
$\nu/2 = s \in (0,1)$ which is equivalent to our own.

Many research results on the non cut-off Boltzmann equation consider regularized kinetic factors, which means that 
$
\Phi( |v-v_*| ) = C_\Phi  \langle v-v_*\rangle^\gamma. 
$
In this situation our results apply easily to 
any  $s\in (0,1)$
and any 
$
\gamma   >  - 2s;
$
 we can also handle these exponents in the physical kinetic factor from \eqref{kernelP}.
We do elect not to record these details herein since it would add unnecessary technical complexity that does not directly relate to the main goal of illustrating our new methods.  
We can further prove our main result for the full physical range: $\gamma + 2s > -1$ and $p>2$; this will be the content of a forthcoming work \cite{sgNonCut2}.
However, in the context of our main theorem and \cite{sgNonCut2}, the essential mathematical difficulties associated with the angular singularities \eqref{kernelQ} are resolved in the present result. 

We will study the linearization of \eqref{BoltzFULL} around the Maxwellian 
equilibrium state 
\begin{equation}
F(t,x,v) = \mu(v)+\sqrt{\mu(v)} f(t,x,v),
\label{maxLIN}
\end{equation}
where without loss of generality
$$
\mu(v) = (2\pi)^{-3/2}e^{-|v|^2/2}.
$$
We will also suppose without restriction that the mass, momentum, and energy conservation laws for the perturbation $f(t,x,v)$ hold for all $t\ge0$ as
\begin{equation}
\int_{\mathbb{T}^3_x\times \mathbb{R}^3_v} ~ dx ~ dv ~ \begin{pmatrix}
      1   \\      v  \\ |v|^2
\end{pmatrix}
~ \sqrt{\mu(v)} ~ f(t,x,v) 
=
0.
\label{conservation}
\end{equation}
This condition should be satisfied initially, and then will continue to be satisfied for a suitably strong solution.  Our main interest is in global strong solutions to the Boltzmann equation \eqref{BoltzFULL} which are perturbations of the Maxwellian equilibrium states \eqref{maxLIN} for the long-range collision kernels \eqref{kernelP} and \eqref{kernelQ}.  

As will be seen, our solution to this problem rests heavily on our introduction of the 
following new weighted geometric fractional  Sobolev space:
$$
N^{s,\gamma}
\eqdef 
\left\{ f \in L^2(\mathbb{R}^3_v): 
 \nsm f\nsm_{N^{s,\gamma}} <\infty
\right\},
$$
where for $v \in \R^3$ we define $\ext{v} \eqdef (v, \frac{1}{2} |v|^2) \in \R^4$, and then specify the  norm 
%$| \cdot |_{N^{s,\gamma}}$ 
by
\begin{equation} 
\nsm f\nsm_{N^{s,\gamma}}^2 
\eqdef 
\nsm f\nsm_{L^2_{\gamma+2s}}^2 + \int_{\R^3} dv ~ \int_{\R^3} dv' ~ 
(\ang{v}\ang{v'})^{\frac{\gamma+2s+1}{2}}
~
 \frac{(f(v') - f(v))^2}{|\ext{v} - \ext{v'}|^{3+2s}} 
\ind_{|\ext{v} - \ext{v'}| \leq 1}. \label{normdef} 
\end{equation}
This space includes the weighted $L^2$ space given by
$$
\nsm f\nsm_{L^2_{\gamma+2s}}^2 
\eqdef
\int_{\R^3} dv ~ 
\ang{v}^{\gamma+2s}
~
|f(v)|^2.
$$
Note that the weight is as usual: 
$
\ang{v}
= \sqrt{1+|v|^2}.
$
The inclusion of the quadratic difference $|v|^2 - |v'|^2$ in the fractional kernel will be of great importance; it is not a lower order term. 
The rest of our notation is defined below.

The sharp space $N^{s,\gamma}$ is equivalent to a weighted, nonisotropic Sobolev norm, a feature which was conjectured in \cite{MR2322149}.  
Precisely, if $\R^3$ is identified with a paraboloid in $\R^4$ by means of the mapping $v \mapsto (v, \frac{1}{2}|v|^2)$ and $\Delta_{P}$ is defined to be the Laplacian on the paraboloid induced by the Euclidean metric on $\R^4$ then
\[ 
\nsm f \nsm_{N^{s,\gamma}}^2 \approx \int_{\R^3} dv \ang{v}^{\gamma+2s} \left| (I- \Delta_{P})^{\frac{s}{2}} f (v) \right|^2. 
\]
We will however omit the proof of this characterization as it has no direct role in establishing our results.  
With this, we may now to state our main result as follows:

\begin{theorem}
\label{mainGLOBAL}  
(Main Theorem).
Fix $\NgE$, the total number of spatial derivatives.
Choose $f_0 = f_0(x,v) \in L^2(\mathbb{R}^3_v; H^N(\mathbb{T}^3_x))$ in \eqref{maxLIN} which satisfies \eqref{conservation} initially.  
There is an $\eta_0>0$ such that if $\| f_0 \|_{L^2(\mathbb{R}^3_v; H^N(\mathbb{T}^3_x))} \le \eta_0$, then there exists a unique global strong solution to the Boltzmann equation \eqref{BoltzFULL}, in the form 
\eqref{maxLIN}, which satisfies
$$
f(t,x,v) \in 
L^\infty_t( [0,\infty); L^2_v H^N_x (\mathbb{T}^3_x\times \mathbb{R}^3_v ))
\cap
L^2_t( [0,\infty); N^{s,\gamma}_vH^N_x(\mathbb{T}^3_x\times \mathbb{R}^3_v )).
$$
Moreover, we have exponential decay to equilibrium.  For some fixed $\lambda >0$,
$$
\| f \|_{L^2(\mathbb{R}^3_v; H^N(\mathbb{T}^3_x))}(t) \lesssim e^{-\lambda t} 
\| f_0 \|_{L^2(\mathbb{R}^3_v; H^N(\mathbb{T}^3_x))}.
$$
 We also have positivity, i.e. $F= \mu + \sqrt{\mu} f \ge 0$ if $F_0= \mu + \sqrt{\mu} f_0 \ge 0$.
\end{theorem}

Grad proposed \cite{MR0156656} in 1963 the angular cut-off which requires that $b(\cos \theta)$ be bounded. 
%and approach zero linearly as $\theta \to 0$.
Grad also pointed out that many cut-offs are possible.  
In particular,  the following less stringent $L^1(\mathbb{S}^2)$ cut-off
 has become fashionable\footnote{
Note that this $L^1(\mathbb{S}^2)$ cut-off was already implicitly used in 1954 by Morgenstern \cite{MR0063956}.} 
$$
\int_{\mathbb{S}^{2}}d\sigma~b(\ang{k, \sigma})
%=
%2\pi \int_0^{\pi/2} d\theta~ \sin\theta ~b(\cos\theta) 
< \infty.
$$
These types of truncations have been widely accepted, and have now influenced several
decades of mathematical progress on the Boltzmann equation. 
For the intermolecular repulsive potentials previously discussed, the cut-off theory only applies physically in the limit when $p\to \infty$, which represents Hard-Sphere particles.

These cut-off assumptions were originally believed to not change the essential nature of solutions to the equation. 
It has been argued by physicists, see \cite{MR1942465}, that the important properties of the Boltzmann equation are not particularly sensitive to the dependence of the collision kernel upon the deviation angle, $\theta$.  
Our theorem above shows at the mesoscopic level of Boltzmann that this physical heuristic misses a strong dependence of the solutions on the angular singularity; specifically, there is a gain of velocity regularity and velocity moments globally in time.

Results regarding these types of smoothing effects can be seen in the pioneering work of  Desvillettes \cite{MR1324404} from 1995, which applied to  simplified models such as the Kac equation.  In a very recent preprint from 2009, perhaps the first local existence theorem for large data 
and moderate angular singularities $0<s<1/2$
has been shown by Alexandre-Morimoto-Ukai-Xu-Yang \cite{arXiv:0909.1229v1} for initial data which is somewhat smoother than ours, they also show
the $C^\infty_{t,x,v}$ regularizing effect.  
By contrast, under
the angular cut-off assumption the solution is known to have the same
regularity in a Sobolev space as the initial data.  These results go back to Boudin-Desvillettes \cite{MR1798557} in 2000 for solutions near Vacuum, and this same effect has been recently  shown in the near Maxwellian regime \cite{MR2435186,MR2476678}.

As a result of the fact that the angular singularity \eqref{kernelQ} is not
integrable on a sphere, it has been conjectured in numerous works that the nonlinear
collision operator should  behave like a fractional (flat) Laplacian in the velocity variable
$v$:
$$
\mathcal{Q}(F, F) =   -C_F {(-\Delta_v )^{s} F} +\text{lower order terms}.
$$
Our precise work at the linearized level shows that this conjecture is not the whole story.   Certainly we see that there is a smoothing effect globally in time.  However for our results the most useful intuitive point of view is to think of the collision operator as a fractional Laplacian on a manifold, and this manifold depends in an essential way on the collisional geometry. 

By comparison, the Landau equation, derived in 1936, is probably the closest analog we have to the Boltzmann collision operator for long-range interactions; however the Landau operator involves regular partial derivatives rather than
 fractional derivatives and for that reason may be somewhat more understandable at first.
Landau's equation is obtained as the limiting system when $p\to 2$ in the inverse power law potential, the collision operator can be shown to satisfy \cite{MR1942465}:
$$
Q_{\mathcal{L}}(F, F) = \sum_{i,j=1}^3 \bar{a}_{ij} \partial_{v_i}\partial_{v_j } F + 8\pi F^2,
\quad
\bar{a}_{ij}
=
\left( \frac{1}{|v|}\left[\delta_{ij} - \frac{v_i v_j}{|v|^2}  \right]  \right) * F.
$$
Notice there is a metric of sorts in this case--in the $\bar{a}_{ij}$--which depends in an essential way on your unknown solution $F$.  Even in the simplest case when  your unknown is the steady state, $F = \mu(v)$, this $\bar{a}_{ij}$ weights more heavily angular derivatives \cite{MR1946444}.  In the general case $\bar{a}_{ij}$ is known to be degenerate and not comparable to $\delta_{ij}$ at infinity, see e.g.  \cite{MR2506070}. 

For  this paper, the basic new understanding which enabled our progress was to identify that the fractional differentiation effects induced by the linearized Boltzmann collision operator are taking place on a paraboloid in $\mathbb{R}^4$.  While we do not directly identify dependence of geometry on the function $F$ itself in our formulation, it seems only reasonable to suspect that such dependence may be relevant to future work.   
Before reviewing  the details of our proof, we will mention past works.

\subsection{Review of the non cut-off theory} 
As above, it makes good sense to briefly review a few results for the Landau equation, which corresponds to the grazing collision limit $s\to 1$ of the Boltzmann equation for long-range interactions (see \cite{MR2037247} and the references therein).  For the spatially  homogeneous case with hard potentials (roughly, replace $1/|v|$ above with $|v|^{\gamma + 2}$ for $\gamma \ge 0$), global existence of unique weak solutions and the instantaneous smoothing effect was shown for the first time by Desvillettes and Villani  \cite{MR1737547} for a large class of initial data in the year 2000.
Then Guo \cite{MR1946444} in 2002 established the existence of classical solutions for the spatially
dependent case with the physical Coulombian interactions ($p=2$) for smooth near Maxwellian initial data in a periodic box. 
Guo's solutions were recently shown to experience instantaneous regularization  in \cite{MR2506070}.  For further  results in these directions we refer to the references in \cite{MR2506070}.

Due to length constraints, it is unfortunately not possible  to give an exhaustive review of results in the non cut-off theory.  We will however try to mention a sample of results. In the case of Maxwell molecules, e.g. $p=5$,  it is  remarkable that the spectrum and eigenfunctions of the linearized collision kernel can be computed explicitly as was performed in a classical paper \cite{WCUh52} from 1952.  This was later simplified by Bobylev \cite{MR1128328}; in this work the Fourier transform of the Boltzmann collision operator was shown to have an elegant form, which is now called  Bobylev's identity and has found  widespread utility.
Pao \cite{MR0636407} in 1974 used the early techniques of pseudodifferential operators and Bessel functions to study the spectral properties of the linearized operator for general inverse-power intermolecular potentials with $p>3$.  
In 
1981-82, 
Arkeryd in \cite{MR630119,MR679196} 
proved the existence of weak
solutions to the spatially homogeneous Boltzmann equation when
$0<s<\frac 12$. Then  Ukai  \cite{MR839310} in 1984  obtained a local Cauchy-Kovalevskaya type theorem in a
function space which was analytic in $x$ and Gevrey in $v$; this work applied to the moderate angular singularities $0 \le s<\frac 12$, the convergence of Grad's cutoff approximation and the positivity of solutions was also established. 
 In 1998  Villani \cite{MR1650006} introduced the new spatially homogeneous weak H-Solution formulation which can handle all physically meaningful interactions with $p>2$ and more generally; we refer the reader to this paper for a fairly exhaustive list of references up to 1998.

In 1998 Lions proved a functional inequality \cite{MR1649477} which bounds below the ``entropy dissipation'' by an isotropic Sobolev norm $H^\alpha_v$ up to lower order terms, for a certain range of $\alpha$.
Then in the work of Alexandre-Desvillettes-Villani-Wennberg
\cite{MR1765272} from 2000, this entropy dissipation smoothing estimate was obtained in the isotropic space $H^s_v$ with the optimal exponent $s$.  This work further introduced elegant formulas, such as the cancellation lemma and  isotropic sub-elliptic coercivity estimates using the Fourier transform.
This was in several ways the starting point of the modern theory of grazing collisions.
Subsequent results of Alexandre-Villani \cite{MR1696193,MR1857879} developed a renormalized DiPerna-Lions theory of weak solutions with defect measure, and established the appearance of strong compactness.  
Desvillettes-Wennberg \cite{MR2038147} further demonstrated that solutions to the spatially homogeneous Boltzmann equation for regularized hard potentials enter the Schwartz space instantaneously.
And recently in
2009  Desvillettes-Mouhot \cite{MR2525118} proved the uniqueness of spatially homogeneous strong solutions  for the full range of angular singularities $s\in(0,1)$, and they have shown existence for moderate angular singularities $s\in(0,1/2)$.

Broadly speaking, the approaches outlined above use the Fourier transform to interpret the fractional differentiation effects isotropically.  Other interesting methods have been introduced to further study the fractional differentiation effects isotropically, using more involved methods from pseudodifferential operators and harmonic analysis.
In particular, some uncertainty principles in the framework of Fefferman
\cite{MR707957} 
were introduced in 2008 by  Alexandre-Morimoto-Ukai-Xu-Yang \cite{MR2462585} to study
smoothing effects of the 
Boltzmann equation  with non-cutoff cross sections. 
Important methods introduced in that paper \cite{MR2462585}, as well as Alexandre  \cite{MR2284553}, Morimoto-Ukai-Xu-Yang
\cite{MR2476686},
and the references therein, establish the hypoellipticity of the
Boltzmann operator.  They develop methods for estimating the commutators  between the Boltzmann collision
operator and  some weighted pseudodifferential operators.  And they sharpen some of the isotropic coercivity and upper bound estimates for the Boltzmann collision operator.  
 Recently these authors have  new preprint \cite{arXiv:0909.1229v1} with a large data local existence theorem as previously mentioned, note also the earlier local existence theorem  between two moving Maxwellians \cite{MR1851391} of Alexandre.

We also mention the  
linear isotropic coercivity estimates from \cite{MR2254617,MR2322149}.  In particular Mouhot-Strain \cite{MR2322149} identified the sharp weight, $\gamma+2s$, in our norm \eqref{normdef}.   \\

During the course of the proof of our main Theorem \ref{mainGLOBAL}, we develop a new point of view and a new set of tools for the long-range interactions, which we believe have implications for a variety of future results both in the perturbative regime and maybe even beyond it.  We do not use any of the major non cut-off techniques  described above, most of which are designed around the Fourier transform and isotropic estimates.  
Moreover, we do not study the Fourier transform of the collision operator at all.

From the standpoint of harmonic analysis, the estimates we make for the bilinear operator 
\eqref{gamma0} arising from our ansatz \eqref{maxLIN} fall well outside the scope of standard theorems.  The operator and its associated trilinear form may be expressed in terms of Fourier transforms as a trilinear paraproduct; such objects have been the subject of recent work of Muscalu, Pipher, Tao, and Thiele \cite{MR2320408} and are known to be very difficult to study in general.  Known results for such objects fail to apply in our case because of the loss of derivatives (meaning that at least two of the three functions $g,h,$ and $f$ must belong to some Sobolev space with a positive degree of smoothness).  Moreover, routine modifications of known results (for example, composing with fractional integration to compensate for the loss of derivatives) also fail because of the presence of a fundamentally non-Euclidean geometry, namely, the geometry on the paraboloid.  This nontrivial geometry essentially renders any technique based on the Fourier transform difficult to use herein. Instead, we base our approach on the generalized Littlewood-Paley theory developed by Stein \cite{MR0252961}.  Rather than  directly using semigroup theory, however, we opt for a more geometric approach, as was taken, for example, by Klainerman and Rodnianski \cite{MR2221254}.  Since the underlying geometry we identify is explicit, we are able to make substantial simplifications over both of these earlier works by restricting attention to the particular case of interest.

\subsection{Possibilities for the future, and extensions}\label{possible}
We believe that our general methods and point of view can be useful in making further progress on multiple fronts in the non cut-off theory.  Herein we list  some of those.

We will soon address the generalization to the very soft potentials and $\gamma + 2s\le 0$ in a subsequent paper \cite{sgNonCut2}. Furthermore, the generalization to the whole space case $\mathbb{R}^3_x$ can be handled by combining our estimates with the existing cut-off technology in the whole space.
It would also be interesting to prove the instantaneous $C^\infty$ regularizing effect for our solutions in the spirit of \cite{MR2506070}.

We are hopeful that the estimates we prove can help to resolve the existence question for the   Vlasov-Maxwell-Boltzmann system without angular cut-off; notice at the moment the theory here is limited to the hard-sphere interactions \cite{MR2000470}.  It would be interesting to study the smoothing effect, which could potentially improve dramatically the strong assumptions in the current time decay theory \cite{MR2209761}.

For spatially homogeneous solutions, our results provide additional information for the high singularities $s\ge 1/2$, in which otherwise there is currently no existence theory for strong solutions \cite{MR2476686,MR2525118}.
The methods and point of view in this paper may help to treat the high singularities  with large spatially homogeneous data.

Lastly, it would be quite important to work with the estimates herein and in \cite{sgNonCut2} to  justify rigorously the validity of Landau approximation near Maxwellian.  \\

In the next Section \ref{nrstat}, we linearize the Boltzmann equation 
\eqref{BoltzFULL} around the perturbation \eqref{maxLIN} and then explain the sharp space associated with the linearized collision operator.  We further define all the relevant notation and formulate and discuss the main velocity fractional derivative estimates.  Then we describe several key new ideas in our proof, and outline the rest of the article. 

\section{Notation, reformulation, the main estimates, and our strategy}
\label{nrstat}

Throughout this section we will define the relevant notation for the problem.  We also reformulate the problem in terms of the equation \eqref{Boltz} for the perturbation \eqref{maxLIN}.  Last and perhaps most importantly, we will state and explain our main estimates towards the end of this section.

Throughout this paper, the notation $A \lesssim B$ will mean that an a positive constant $C$ exists such that $A \leq C B$ holds uniformly over the range of parameters which are present in the inequality (and that the precise magnitude of the constant is irrelevant).  In particular, whenever either $A$ or $B$ involves a function space norm, it will be implicit that the constant is uniform over all elements of the relevant space unless explicitly stated otherwise.  The notation $B \gtrsim A$ is equivalent to $A \lesssim B$, and $A \approx B$ means that both $A \lesssim B$ and $B \lesssim A$.

\subsection{Reformulation}
We linearize the Boltzmann equation \eqref{BoltzFULL} around the perturbation \eqref{maxLIN}.  This grants an equation for the perturbation $f(t,x,v)$ as
\begin{gather}
 \partial_t f + v\cdot \nabla_x f + L (f)
=
\Gamma (f,f),
\quad
f(0, x, v) = f_0(x,v),
\label{Boltz}
\end{gather}
where the {\it linearized Boltzmann operator} $L$ is given by
\begin{align*}
 L(g)
 \eqdef & 
- \mu^{-1/2}\mathcal{Q}(\mu ,\sqrt{\mu} g)- \mu^{-1/2}\mathcal{Q}(\sqrt{\mu} g,\mu) \\
  = &
  \int_{\mathbb{R}^3}dv_{*}
  \int_{\mathbb{S}^{2}} d\sigma~
  B(|v-v_*|,\cos \theta) \, 
   \left[g_{*} M + g M_{*}- g^{\prime}_{*} M^{\prime} - g^{\prime} M^{\prime}_{*}  \right] M_{*},
\end{align*}
and the bilinear operator $\Gamma$ is given by
\begin{gather}
%\notag
\Gamma (g,h)
\eqdef
 \mu^{-1/2}\mathcal{Q}(\sqrt{\mu} g,\sqrt{\mu} h)
 %\\
 =
 \int_{\R^3} dv_* \int_{\R^3} d \sigma B M_* (g_*' h' - g_* h). 
\label{gamma0}
\end{gather}
In both definitions, we take
$$
M(v) \eqdef \sqrt{\mu(v)} = (2 \pi)^{-3/4} e^{- |v|^2/4}.
$$
When convenient, we will without loss of generality abuse notation and neglect the constant $(2 \pi)^{-3/4}$ in the definition of $M$.
Finally, we note that 
\begin{equation}
L(g) \eqdef - \Gamma(M,g) - \Gamma(g, M).
\label{LinGam} 
\end{equation}
This reformulation shows that it is fundamentally important to obtain favorable estimates for the bilinear operator $\Gamma$.

We expand the main term of the linearized Boltzmann operator as
  \begin{equation*}
\Gamma(M,g)
  =
  \int_{\mathbb{R}^3}dv_{*}
  \int_{\mathbb{S}^{2}} d\sigma~
  B(|v-v_*|,\cos \theta) \, 
   \left[M^{\prime}_{*} g^{\prime}  -  M_{*} g \right] M_{*}.
  \end{equation*}
 We will now split this in parts whilst preserving the cancellations as follows:
   \begin{gather*}
\Gamma(M,g)  =
 \int_{\mathbb{R}^3}dv_{*}
  \int_{\mathbb{S}^{2}} d\sigma~
  B \, 
   \left(g^{\prime}- g \right) M^{\prime}_{*}M_{*}
  % \\
   -
 \tilde{\nu}(v)  ~ g(v),
  \end{gather*}
where
  $$
 \tilde{\nu}(v) 
   =
  \int_{\mathbb{R}^3}dv_{*}
  \int_{\mathbb{S}^{2}} d\sigma~
  B(|v-v_*|,\cos \theta) \, 
 ( M_{*} - M^{\prime}_{*} ) M_{*}.
 $$ 
 The first piece above satisfies a favorable identity.  This
 will follow from the pre-post collisional change of variables as
   \begin{gather*}
 - \int_{\mathbb{R}^3}dv
  \int_{\mathbb{R}^3}dv_{*}
  \int_{\mathbb{S}^{2}} d\sigma~
  B(|v-v_*|,\cos \theta) \, 
 (g^{\prime}-g) h M^{\prime}_{*}  M_{*}
 \\
     =
     -
     \frac{1}{2}
  \int_{\mathbb{R}^3}dv
  \int_{\mathbb{R}^3}dv_{*}
  \int_{\mathbb{S}^{2}} d\sigma~
  B \,  
 (g^{\prime}-g) h M^{\prime}_{*}  M_{*}
 \\
 -
      \frac{1}{2}
  \int_{\mathbb{R}^3}dv
  \int_{\mathbb{R}^3}dv_{*}
  \int_{\mathbb{S}^{2}} d\sigma~
  B \, 
 (g-g^{\prime}) h^{\prime} M_{*}  M^{\prime}_{*}
  \\
     =
     \frac{1}{2}
  \int_{\mathbb{R}^3}dv
  \int_{\mathbb{R}^3}dv_{*}
  \int_{\mathbb{S}^{2}} d\sigma~
  B \, 
 (g^{\prime}-g) (h^{\prime}-h) M^{\prime}_{*}  M_{*}.
  \end{gather*}
This shows the crucial point that herein is contained a Hilbert space structure.
For the weight,   it has been shown that we can make the splitting
$$
 \tilde{\nu}(v)  =  \nu(v) + \nu_K(v),
$$
where under only   \eqref{kernelQ} and \eqref{kernelP} it can be seen that
\begin{equation}
 \nu (v)\approx
\ang{v}^{\gamma+2s},
\quad
\text{and}
\quad
\left|  \nu_K(v) \right|
\lesssim
\ang{v}^{\gamma}.
\notag
\end{equation}
These estimates
were established by
 Pao \cite[p.568 eq. (65), (66)]{MR0636407}
when studying the eigenvalue problem.
Pao reduced this question to the known asymptotic behavior of confluent hypergeometric functions, making  use of the general addition theorems. 
%We can thus choose a suitable constant $A>0$ and let 
%$
%\nu(v) \eqdef  A + \tilde{\nu}(v).
%$
% Then  
%\begin{equation}
% \nu(v)\approx
%\ang{v}^{\gamma+2s}.
% \label{pao2}
%\end{equation}  
%With this we are ready to define our norm.

We further decompose  $L=N  + K $.  Here $N$ is the ``norm part'' and $K$ will be seen as the ``compact part.''
The norm part is then written as
\begin{gather}
   \label{normpiece}
 Ng
   \eqdef
   - \Gamma(M,f) - \nu_K(v) f
   %\\
   =
  -\int_{\mathbb{R}^3}dv_{*}
  \int_{\mathbb{S}^{2}} d\sigma~
  B %(|v-v_*|,\cos \theta) \, 
 (g^{\prime}-g) M^{\prime}_{*}  M_{*}
   + \nu(v) g(v).
%\notag
\end{gather}
We will use $\langle \cdot, \cdot \rangle$ to denote the standard $L^2(\R^3_v)$ inner product.
Then, with the previous calculations, this norm piece satisfies the following identity: 
\begin{gather*}
  \ang{Ng,g} =  \frac{1}{2} \int_{\R^3} dv \int_{\R^3} dv_* \int_{\sph} d \sigma B (g'-g)^2 M_*' M_* 
  + 
  \int_{\R^3} dv ~ \nu(v) ~ |g(v)|^2. 
\end{gather*}
As a result, 
in the following we will use the anisotropic fractional semi-norm
 \begin{equation}
| g |_B^2 \eqdef 
\frac{1}{2} \int_{\R^3} dv~ \int_{\R^3} dv_*~ \int_{\sph} d \sigma~ B~ (g'-g)^2 M_*' M_*. \label{normexpr}
\end{equation}
For the second part of $\ang{Ng,g}$ we recall the norm 
$
\nsm f\nsm_{L^2_{\gamma+2s}}
$
defined below equation \eqref{normdef}. 
These two quantities will define our designer norm, it is sharp for the linearized operator.
We also record here the definition of the ``compact piece'' $K$:
 \begin{equation}
  K g \eqdef  \nu_K(v) g 
- \Gamma(g, M)
=
  \nu_K(v) g 
-
 \int_{\R^3} dv_* \int_{\R^3} d \sigma B M_* (g_*' M' - g_* M).
 \label{compactpiece}
\end{equation}
This is our main splitting of the linearized operator.

To simplify many statements below, we  will  use the following $(\eta, \delta)$-norm:
 $$
  \nsm f\nsm_{\eta,\delta}^2 \eqdef \int_{\R^3} ~ dv ~ |f(v)|^2 ( \eta \ang{v}^{\gamma+2s} + \delta^{-1} ),
  \quad \eta, \delta >0.
  $$
This norm perhaps requires some explanation.  It will serve to unify several different desirable inequalities later.  
The first useful feature of these norms is 
to observe that
\begin{equation}
\inf_{\eta,\delta > 0} |g|_{\delta,\eta} |h|_{\eta,\delta} = |g|_{L^2} |h|_{L^2_{\gamma+2s}} + |g|_{L^2_{\gamma+2s}} |h|_{L^2}.
 \label{expand}
\end{equation}
The second desirable inequality to be made is that 
\begin{equation*}
 \nsm f\nsm_{\eta,\eta}^2 \leq 2 \eta \nsm f\nsm_{L^2_{\gamma+2s}}^2 + \eta^{-1} \int_{|v| \leq \eta^{-\frac{2}{\gamma+2s}}} dv ~ |f(v)|^2, 
 \quad 
 \forall \eta >0.  %\label{compactnorm}
\end{equation*}
This holds because $\gamma+2s > 0$ so the weight $1$ is ultimately bounded by any small constant times $\ang{v}^{\gamma+2s}$ provided $|v|$ is large enough.  The right-hand side of this inequality is precisely the sort of norm needed for estimates of the ``compact'' part.

All of the above $L^2(\mathbb{R}^3_v)$-based norms  are exclusively in the velocity variables.  We will define analogous $L^2(\mathbb{T}^3_x \times \mathbb{R}^3_v)$ norms in both space and velocity variables by replacing $| \cdot |$ with $\| \cdot \|$ and keeping the same norm subscript.  We also use $(\cdot, \cdot )$ to be the usual $L^2(\mathbb{T}^3_x \times \mathbb{R}^3_v)$ inner product.  In particular
$$
\| h\|_{N^{s,\gamma}}^2
\eqdef
\left\|~ \nsm h\nsm_{N^{s,\gamma}} ~ \right\|_{L^2(\mathbb{T}^3_x)}^2,
\quad
\| h\|_{\eta, \delta}^2
\eqdef
\left\|~ \nsm h\nsm_{\eta, \delta} ~ \right\|_{L^2(\mathbb{T}^3_x)}^2.
$$
Now, for a multi-index $\alpha = (\alpha^1, \alpha^2, \alpha^3)$, we will use the spatial derivatives
$$
\partial^\alpha 
=
\partial^{\alpha^1}_{x_1} 
\partial^{\alpha^2}_{x_2} 
\partial^{\alpha^3}_{x_3}. 
$$
We will also take
$ L^2\left(\mathbb{R}^3_v ; H^N\left(\mathbb{T}^3_x \right)\right)$ with
 $\NgE$ spatial derivatives to be  
 \begin{gather*}
%\label{norm}
\|h\|^2_{L^2(\mathbb{R}^3_v; H^N(\mathbb{T}^3_x))}\eqdef \sum_{|\alpha |\le N}
\|\partial^\alpha h\|^2_{L^2(\mathbb{R}^3_v\times\mathbb{T}^3_x)}.
\end{gather*}
 For brevity, in each of these norms we sometimes use the notation
  $L^2_v$,  $L^2_v L^2_x$, $L^2_v H^N_x$, etc,
  to denote these spaces without confusion.
For example, we will sometimes write 
$
\|h\|^2_{L^2(\mathbb{R}^3_v; H^N(\mathbb{T}^3_x))}
=
\|h\|^2_{L^2_vH^N_x}
$
and use other norms such as
$
\|h\|^2_{L^2_x}
=
\|h\|^2_{L^2(\mathbb{T}^3_x)}.
$
%and
%$
%\|h\|^2_{L^2_v}
%=
%\|h\|^2_{L^2(\mathbb{R}^3_v)}.
%$

\subsection{Main Estimates}\label{mainESTsec}
We will prove all of our estimates for functions in the Schwartz space, $\mathcal{S}(\mathbb{R}^3)$, which is the well-known space of real valued $C^{\infty}(\mathbb{R}^3)$ functions all of whose derivatives decay at infinity faster than the reciprocal of any polynomial.   
Note that the Schwartz functions are dense in the nonisotropic space $N^{s,\gamma}$, and the proof of this fact is easily reduced to the analogous one for Euclidean Sobolev spaces by means of the partition of unity as constructed, for example, in section \ref{redistsec}.
Moreover, in all of our estimates, none of the constants that come up will depend on the %full Schwartz 
regularity of the functions that we are estimating.  Thus using routine density arguments, our estimates will apply to any function in $N^{s,\gamma}$ or whatever the appropriate function space happens to be for a particular estimate.  
Our first non-linear estimate is the following:

\begin{lemma}
\label{NonLinEst}
For the non-linear term \eqref{gamma0}, we have the uniform estimate
\begin{equation}
\begin{split}
 |\ang{\Gamma(g,h),f}| \lesssim & \nsm g\nsm_{L^2} \nsm h\nsm_{N^{s,\gamma}} \nsm f\nsm_{N^{s,\gamma}} 
  + \nsm g\nsm_{L^2_{\gamma+2s}} \left[ \nsm h\nsm_{L^2} \nsm f\nsm_{N^{s,\gamma}} + \nsm h\nsm_{N^{s,\gamma}} \nsm f\nsm_{L^2} \right].
\end{split} \label{nlineq}
\end{equation}
For any $|\alpha |\le N$ and $\NgE$, with the above estimate we have
\begin{gather*}
 \left| \left(\partial^\alpha \Gamma(g,h),\partial^\alpha f\right) \right| 
 \lesssim 
  \| g\|_{H^N_{x} L^2_v} \| h\|_{H^N_{x}N^{s,\gamma}} \| f\|_{H^N_{x}N^{s,\gamma}} 
  \\
  + \| g\|_{H^N_{x} L^2_{\gamma+2s}} 
  \left[ \| h\|_{H^N_{x} L^2_v} \| f\|_{H^N_{x}N^{s,\gamma}} + \| h\|_{H^N_{x}N^{s,\gamma}} \| f\|_{H^N_{x} L^2_v} \right].
\end{gather*}
\end{lemma}

This second estimate in Lemma \ref{NonLinEst} follows easily by integrating  \eqref{nlineq} over $\mathbb{T}^3_x$, using Cauchy-Schwartz and the Sobolev embedding $L^\infty(\mathbb{T}^3_x) \supset H^2(\mathbb{T}^3_x)$.  
 The next important inequality that we establish is for the linear operator:

\begin{lemma}
\label{sharpLINEAR}
Consider the linearized Boltzmann operator $L = N + K$ where
$N$ is defined in \eqref{normpiece} and $K$ is defined in \eqref{compactpiece}.
We have the uniform inequalities
\begin{align}
\left| \ang{N g, g} \right| & \lesssim |g|_{N^{s,\gamma}}^2, \label{normupper} \\
\left| \ang{K g, g} \right| & \le \eta |g|_{L^2_{\gamma+2s}}^2 + C_\eta 
|g|_{L^2}^2, \label{compactupper}
\end{align}
where $\eta>0$ is any small number and $C_\eta>0$. 
\end{lemma}

In these estimates there are several things to observe.  
First of all there are no derivatives in the ``compact estimate'' from \eqref{compactupper}, which should be contrasted with the corresponding estimate in the Landau case \cite[Lemma 5]{MR1946444} in which the upper bound  requires the inclusion of derivatives. Further  \eqref{normupper} is a  simple consequence of the main estimate \eqref{nlineq}.  This estimate tells us that the ``norm'' piece of the linear term, given by $\ang{N g,g}$, is bounded above by a uniform constant times $\nsm g\nsm_{N^{s,\gamma}}^2$.  This means that the coercive inequality  in the next Lemma \ref{estNORM3} is essentially sharp.

\begin{lemma}
\label{estNORM3}
For the spaces defined in \eqref{normdef} and \eqref{normexpr}, we have the uniform estimate
%The inequality \Red{PUT IN EQREF}   
\begin{equation*}
\nsm f \nsm_{N^{s,\gamma}}^2 \lesssim \nsm f \nsm_{L^2_{\gamma+2s}}^2 
+ 
\nsm f \nsm_{B}^2.
%\frac{1}{2} \int_{\R^3} dv \int_{\R^3} dv_* \int_{\sph} d \sigma B (f'-f)^2 M_*' M_*, 
%\label{coercelemma}
\end{equation*}
\end{lemma}

The coercive inequality in Lemma \ref{estNORM3} and \eqref{normupper} taken together demonstrate that the ``norm piece'' \eqref{normpiece} is actually comparable to our designer norm $N^{s,\gamma}$, i.e.,
\[ \ang{ N f, f} \approx |f|_{N^{s,\gamma}}^2. \]
While the upper bound \eqref{normupper} will have no major use in our arguments, it nevertheless is important because it demonstrates that the non-isotropic space $N^{s,\gamma}$ naturally arises in this near Maxwellian problem.

Finally we have a useful, intermediate inequality which is also the key to the estimate for $K$ in \eqref{compactupper}.
For the following lemma, we take $e_l$ to be a Maxwellian, or a Maxwellian times any polynomial $p$ in $\ext{v}$ of degree at most two:

\begin{lemma}
\label{CompactEst}
For all positive $\eta,\delta$, we have the following uniform estimate
\begin{equation*}
 \left| \ang{\Gamma(g,e_l),f} \right| + \left| \ang{\Gamma(g,f),e_l} \right| \lesssim \nsm g\nsm_{\eta,\delta} \nsm f \nsm_{\delta,\eta}.
\end{equation*}
\end{lemma}

Functions of the sort as $e_l$, will come up in \eqref{base} later on.  Of course, Lemma \ref{CompactEst} is expected to hold whenever $e_l$ is smooth and rapidly decaying, for example, but we restrict $e_l$ to have the specific form mentioned here because it simplifies the proof and the additional generality will be of no use herein. 

In the following we will outline several key new ideas in our proof.

\subsection{Outline of the article, and overview of our proof}
It has been known to the experts for some time that the sum total of the inequalities in Section \ref{mainESTsec} would be sufficient for global existence \cite{MR2000470}, 
although crucially the spaces in which these inequalities should be proven was unknown.
Notice that Lemma \ref{sharpLINEAR} implies
\begin{align}
\ang{L g,g}  & \gtrsim |g|_{N^{s,\gamma}}^{2}  - \mbox{lower order terms}, \label{lowerbound}
\end{align}
This coercive lower bound inequality is fundamental to global existence.
Since the the operators $\Gamma$ and $L$ are intimately connected, among other consequences, this means that if both of \eqref{lowerbound} and Lemma \ref{NonLinEst} are simultaneously true, then the Hilbert space $N^{s,\gamma}$ satisfying these inequalities is unique.   From this point of view the first major difficulty which we had to overcome was the identification of the appropriate Hilbert space (which we have already described).

\subsection*{Identification of the space $N^{s,\gamma}$}

It turned out that the candidate Hilbert space $N^{s,\gamma}$ is a weighted, anisotropic fractional Sobolev space \eqref{normdef}  and \eqref{normpiece} which corresponds to fractional differentiation on the paraboloid in $\R^4$.  
To estimate this space, we find it convenient to use a geometric Littlewood-Paley-type decomposition, inspired by the work of Stein \cite{MR0252961}.  We do not, however, take a semigroup approach to the actual construction of our Littlewood-Paley projections as Stein did.  Instead, we use the embedding of the paraboloid in $\R^4$ to our advantage.  If $d \mu$ is the Radon measure on $\R^4$ corresponding to surface measure on the paraboloid, our approach is to take a renormalized version of the {\it four-dimensional, Euclidean} Littlewood-Paley decomposition of the {\it measure} $g d \mu$ as our non-isotropic, three-dimensional, Littlewood-Paley-type decomposition for the function $g$.  Among other benefits, this approach automatically allows for a natural extension of the Littlewood-Paley projections $P_j g$ and $Q_j g$ (from Section \ref{sec:acd}) as smooth functions defined on $\R^4$ in a neighborhood of the paraboloid.  This allows us to avoid a direct discussion of the induced metric on $\R^3$ by phrasing our results in terms of the projections $P_j g$, $Q_j g$, and various Euclidean derivatives of these functions in $\R^4$ instead of $\R^3$.

\subsection*{The upper bound inequality}

The proof of the main non-linear estimate in Lemma \ref{NonLinEst} is based on a dyadic decomposition of the singularity of the collision kernel $B$ in \eqref{kernelQ} and \eqref{kernelP} as well as a Littlewood-Paley-type decomposition of the functions $h$ and $f$.  The end result is that one is led to consider a triple sum
\[ \sum_{k=-\infty}^\infty \sum_{j' = 0}^\infty \sum_{j=0}^\infty \left| \ang{\Gamma_k(g,h_{j'}, f_{j}} \right|. \]
Here $\Gamma_k$ is the non-linear operator \eqref{gamma0} summed over a special dyadic decomposition of the singularity, and $h_{j'}$, $f_{j}$ are the functions $h$, $f$ expanded in terms of the anisotropic Littlewood-Paley decomposition described just above.
Control over the sum of the pieces rests on two important observations.
First, when considering terms for which $2^{-k}$ is large relative to $2^{-j'}$ and $2^{-j}$, a favorable estimate holds simply because the support of $B_k(|v-v_*|,\cos \theta)$ is compact and bounded away from the singularity at $\theta = 0$.  This is the regime which may be thought of as being far from the singularity.
Second, when either $2^{-j'}$ or $2^{-j}$ is large relative to $2^{-k}$, i.e., near the singularity, an improvement may be made by exploiting the inherent cancellation structure of $\Gamma_k$.  The cost which must be paid in order to use this cancellation is that derivatives must fall on either $h_{j'}$ or $f_{j}$.  In this case, with the dual formulation (described next)  it is always possible to arrange for the derivatives to be placed on the function of our choice.  Placing the derivatives on the function of largest scale (that is, the function whose index is least) gives some extra decay that allows one to sum all the terms by comparison to a geometric series.  The cancellation structure is not measured in the usual way, we measure cancellations using the metric on the paraboloid.

It should be noted that our analysis allows us to essentially ignore the dependence of $\Gamma(g,h)$ on the function $g$; this is a great advantage, as it means that one may think of the trilinear form $\ang{\Gamma(g,h),f}$ as a family of bilinear forms in $h$ and $f$ parametrized by the function $g$.  This observation is essential, since the fully trilinear form falls well outside the scope of existing tools in harmonic analysis.

\subsection*{The dual formulation}
A key point of significant technical importance in the proof of the upper bound inequality is that we must be able to make estimates for $\ang{\Gamma(g,h),f}$ which exploit the intrinsic cancellations at the cost of placing derivatives on any one of the two functions $h$ or $f$ that we choose.  If we were not forced to consider fractional derivatives, a suitable tool would be integration-by-parts.  As it stands however, it is necessary to find two different yet analogous representations of the trilinear form $\ang{\Gamma(g,h),f}$ which clearly relate cancellation to smoothness of $h$ and $f$, respectively.  It turns out that placing derivatives on $f$ is fairly straightforward to do using existing Carleman-type representations for the bilinear operator $\Gamma$.  In particular, one may apply a standard pre-post change of variables on the gain term ${\mathcal Q}^+$ to obtain the representation
\begin{align*}
\ang{\Gamma(g,h),f} = \int_{\R^3} dv \int_{\R^3} dv_* \int_{\sph} d \sigma B ~ g_* h (M_*' f' - M_* f),
\end{align*}
which is justified by approximation of $B$ by a sequence of cut-off kernels.   Clearly, for each fixed $g$, there is an operator $T_g$ such that $\ang{\Gamma(g,h),f} = \ang{T_g f,h}$, and moreover, the formula above can be used to write down an explicit formula for $T_g$.
To place derivatives on $h$, on the other hand, it is necessary to find a new representation which involves only differences of $h'$ and $h$, i.e., no differences of $g$ or $f$.  To that end, there is a need to compute what we call the ``dual formulation,'' since this amounts to writing down a formula for $T_g^*$.  These computations may be found in the Appendix; the end result is that
\begin{align*}
\ang{\Gamma(g,h),f} = \int_{\R^3} dv &  \int_{\R^3} dv_* \int_{\sph} d \sigma   B g_* f' \\ & \times \left(M_*' h \vphantom{\int} \right.  -  \left. M_* h' \frac{|v'-v_*|^3 \Phi(v'-v_*)}{|v-v_*|^3 \Phi(v-v_*)} \right).
\end{align*}
An interesting consequence of this formula is that the gain term ${\mathcal Q}^+$ is unchanged and only the loss term ${\mathcal Q}^-$ differs in these two formulas.  These two formulas also demonstrate the essentially straightforward dependence on $g$ which we use to apply traditionally bilinear methods to the trilinear form.

\subsection*{The coercive inequality}
The key to proving \eqref{lowerbound}, on the other hand, is to show the equivalence between 
\eqref{normdef} and the inner product $\ang{N f,f}$ from
\eqref{normpiece}.
We prove equivalent estimates in terms of the Littlewood-Paley projections.  This analysis consists of two parts.  The first is rewriting \eqref{normexpr} with a Carleman representation as
\begin{equation}
| f |_B^2
=
\int_{\R^3} dv \int_{\R^3} dv' K(v,v') (f'-f)^2, \label{semiCARLEMAN}
\end{equation}
for an appropriate function $K(v,v')$, see \eqref{kernel} .  If we let $d(v,v')$ denote the Euclidean distance in $\R^4$ between the points $(v,\frac{1}{2} |v|^2)$ and $(v', \frac{1}{2} |v'|^2)$, a simple pointwise estimation of this function $K$ demonstrates that
\[ K(v,v') \gtrsim (\ang{v} \ang{v'})^{\frac{\gamma+2s+1}{2}} (d(v,v'))^{-3 - 2s}, \]
for a large set of pairs $(v,v')$, the exact description of which is slightly complicated.  The second part is to demonstrate that the set of pairs for which this inequality holds is large enough to conclude an integral version of this inequality, namely,
\[ \ang{N f,f} \gtrsim \int _{\R^3} dv \int_{\R^3} dv' ~ (\ang{v} \ang{v'})^{\frac{\gamma+2s+1}{2}} (f' - f)^2 (d(v,v'))^{-3-2s}  {\mathbf 1}_{d(v,v') \leq 1}. \]
This latter argument is accomplished by means of a partition of unity and Fourier analysis, the key point being that the expressions
\[ \int_{\R^3} dv \int_{\R^3} dv' (f'-f)^2 \frac{\Omega(v-v')}{|v-v'|^{3+2s}}, \]
are uniformly comparable for all suitable $f$ as $\Omega$ ranges over the family of nonnegative, homogeneous functions of degree $0$ for which
$|\Omega |_{L^1(\mathbb{S}^2)} \gtrsim 1$ and $| \Omega |_{L^\infty(\mathbb{S}^2)} \lesssim 1$. \\

 The plan of the rest of the paper is as follows.  In Section \ref{physicalDECrel} we will formulate the first major physical decomposition of the trilinear form associated with the non-linear collision operator \eqref{gamma0}.  With this we prove the main ``size and support'' estimates.  We finish this section by formulating the main cancellation inequalities using the metric on the paraboloid.

In Section \ref{sec:aniLP}, we develop the anisotropic Littlewood-Paley decomposition which is associated to the geometry of the paraboloid.  We further prove estimates connecting the Littlewood-Paley square functions with our norm \eqref{normdef}.

In Section \ref{sec:upTRI} we prove the key estimate for the trilinear form,  Lemma \ref{NonLinEst}.  This estimate will rely heavily on all of the developments in the previous two sections.  The ``compact estimate'' in Lemma \ref{CompactEst} will follow shortly from these developments, and also the sharp linear upper bounds from  Lemma \ref{sharpLINEAR}. 

The last estimates on the velocity side are contained in Section \ref{sec:mainCOER}, where it is shown crucially that the main norm \eqref{normdef} is comparable to both our anisotropic Littlewood-Paley square function and also the space which is  generated by the linearized operator: $\ang{Nf, f}$ below \eqref{normpiece}.  This involves several ideas, including  estimating a Carleman-type representation and what we call a ``Fourier redistribution'' argument.  We further develop useful functional analytic properties of $N^{s,\gamma}$.

In Section \ref{sec:deBEest}, we show that all of our singular estimates on the velocity variables from the previous sections  can be included in the current cut-off theory.
Specifically, we use the space-time  estimates and non-linear energy method that was introduced by Guo \cite{MR2000470,MR2013332,MR1946444}.
This works in particular because our new arguments for the velocity variables outlined above are morally fully decoupled from the argument to handle the space-time aspects of the equation.
We further remark that our estimates above are, in general, flexible enough to adapt to other modern cut-off methods.

Lastly, the Appendix contains Carleman-type representations and a derivation of a ``dual formulation'' for the trilinear form \eqref{3dualZ} that is used in the main text.

\section{Physical decomposition and related estimates}
\label{physicalDECrel}

In this section we introduce the first major decomposition and prove several estimates which will play a central role in establishing the main inequality for the non-linear term $\Gamma$ from \eqref{gamma0} and the norm $| \cdot |_{N^{s,\gamma}}$.  This first decomposition is a decomposition of the singularity of the collision kernel.  For various reasons, it turns out to be useful to decompose $b ( \cos \theta)$ from \eqref{kernelQ} to regions where $\theta \approx 2^{-j} |v-v_*|^{-1}$, rather than a simpler dyadic decomposition not involving $|v-v_*|$.  The principal benefit for doing so is that this extra factor makes it easier to prove estimates on the space $L^2_{\gamma+2s}(\mathbb{R}^3_v)$  because the weight $\Phi(|v-v_*|)$ from  \eqref{kernelP}  is already present in the kernel and the extra weight $|v-v_*|^{2s}$ falls out automatically from our decomposition.

The estimates to be proved fall into two main categories:  the first are various $L^2$- and weighted $L^2$-inequalities which follow directly from the size and support conditions on our decomposed pieces (such estimates are typically called ``trivial'' estimates).   The second type of estimate will assume some sort of smoothness and obtain better estimates than the ``trivial'' estimates by exploiting the cancellation structure of the non-linear term $\Gamma$ from \eqref{gamma0}.  It is already worth stating at this point that the particular smoothness assumptions we make are dictated by the problem and will  specifically be somewhat unusual; in particular, they will not correspond to the usual, Euclidean Sobolev spaces on $\R^3$.

\subsection{Dyadic decomposition of the singularity}
Let $\{ \chi_j \}_{j=-\infty}^\infty$ be a partition of unity on $(0,\infty)$ such that $\nsm \chi_j\nsm_{L^\infty} \leq 1$ and $\chi_j$ is supported on $[2^{-j-1},2^{-j}]$.  
For each $j$, let
\[B_j = B_j(v-v_*,\sigma) \eqdef \Phi(|v-v_*|) b \left( \left< \frac{v-v_*}{|v-v_*|}, \sigma \right> \right) \chi_j (|v - v'|). \]
Note that
\[ 
|v-v'|^2 = \frac{|v-v_*|^2}{2} \left( 1 - \left< \frac{v-v_*}{|v-v_*|}, \sigma \right>  \right)
= |v-v_*|^2 \sin^2 \frac{\theta}{2}.
\]
Hence, the condition $|v-v'| \approx 2^{-j}$ is equivalent to the condition that the angle between $\sigma$ and $\frac{v-v_*}{|v-v_*|}$ is comparable to $2^{-j} |v-v_*|^{-1}$.  With this partition, we define
\begin{align*}
T^j_{+}(g,h,f)  & \eqdef \int_{\R^3} dv \int_{\R^3} dv_* \int_{\sph} d \sigma ~ B_j(v-v_*, \sigma) ~ g_* h M_*' f'   \\ 
T^j_{-}(g,h,f)  & \eqdef \int_{\R^3} dv \int_{\R^3} dv_* \int_{\sph} d \sigma ~ B_j(v-v_*, \sigma) ~ g_* h M_* f. 
\end{align*}
It turns out that we will also need to express the collision operator \eqref{gamma0} using its ``dual formulation.''  With the variant of Carleman's representation coming from Proposition \ref{carlemanA} and the notation 
$
M_*' = M(v+v_* - v')
$
(
$
=
\frac{M M_*}{M'}
$
on $E_{v_*}^{v'}$),
 we have the following alternative representation for $T^j_+$ as well as the definition of a third trilinear operator $T^j_*$ (based on the calculation \eqref{3dualZ} with, recall, $v_*' = v+v_* - v'$):
\begin{align*}
T^j_{+}(g,h,f)  & 
= 4 \int_{\R^3} dv' \int_{\R^3} dv_* \int_{E_{v_*}^{v'}} d \pi_{v} ~ 
\frac{B_j(v-v_*, 2v' - v- v_*)}{|v'-v_*| ~ |v-v_*|} ~ g_*  f' M_*' h 
\\ 
T^j_{*}(g,h,f) & 
\eqdef
4 \int_{\R^3} dv' \int_{\R^3} dv_* \int_{E_{v_*}^{v'}} d \pi_{v} ~ 
B_j ~ \frac{\Phi(v'-v_*)}{ \Phi(v-v_*)}~ \frac{ |v'-v_*|^2 }{|v-v_*|^4} ~ g_* f'  M_* h',
\end{align*}
where above we make the slight abuse of notation that
$$
B_j
=
B_j(v-v_*, 2v' - v- v_*) 
=
\Phi(|v-v_*|)b \left( \ang{\frac{v-v_*}{|v-v_*|}, \frac{2v' - v - v_*}{|2v' - v - v_*|} } \right)\chi_j (|v - v'|).
$$
In these integrals above $d\pi_{v}$ is Lebesgue measure on the two-dimensional plane $E_{v_*}^{v'}$ passing through $v'$ with normal $v' - v_*$, and $v$ is the variable of integration.  
When $f, g, h \in \mathcal{S}(\mathbb{R}^3)$, the pre-post collisional change of variables, the dual representation \eqref{3dualZ} from the appendix, and the previous calculations guarantee that 
\begin{align*}
 \left< \Gamma(g,h), f \right> & = 
\sum_{j=-\infty}^\infty  \left\{ T^j_+(g,h,f) - T^j_{-}(g,h,f) \right\}
\\
& = \sum_{j=-\infty}^\infty  \left\{ T^j_+(g,h,f) - T^j_{*}(g,h,f) \right\} . 
\end{align*}
These will be the general quantities that we estimate in the following sections.
The first step is to estimate each of  $T^j_+$, $T^j_-$, and $T^j_*$ using only the known constraints on the size and support of $B_j$.

\subsection{``Trivial'' analysis of the decomposed pieces}

Under the condition that $\gamma + 2s \ge 0$, the following basic inequality holds uniformly for all $\eta,\delta  > 0$:
\begin{align}
 \ang{v-v_*}^{\gamma + 2s}  & \lesssim \left( \eta \ang{v}^{\gamma+2s} + \delta^{-1} \right) \left( \delta \ang{v_*}^{\gamma+2s} + \eta^{-1} \right). \label{angineq1} 
\end{align}
We will use this inequality in all of our estimates and refer to it as the  $(\eta,\delta)$-inequality.  We begin with the following:

\begin{proposition}
For any integer $j$ and any $\eta,\delta > 0$, we have the uniform estimates
\begin{align}
 \left| T^j_{-}(g,h,f) \right|  & \lesssim 2^{2sj} \nsm g\nsm_{\delta,\eta} \nsm h\nsm_{\eta,\delta} \nsm f\nsm_{L^2_{\gamma+2s}}
 \label{tminusgh} 
 \\
 \left| T^j_{-}(g,h,f) \right| &  \lesssim 2^{2sj} \nsm g\nsm_{\delta,\eta} \nsm h\nsm_{L^2_{\gamma+2s}} \nsm f\nsm_{\eta,\delta}.
 \label{tminusgf}
\end{align}
Moreover, uniformly for any any integer $j$ we have
\begin{equation}
\left| T^j_{-}(g,h,f) \right|  \lesssim 2^{2sj} \nsm g\nsm_{L^2} \nsm h\nsm_{L^2_{\gamma+2s}} \nsm f\nsm_{L^2_{\gamma+2s}}. \label{tminussmall}
\end{equation}
\end{proposition}

\begin{proof}
Given the size estimates for $b(\cos \theta)$ in \eqref{kernelQ} and the support of $\chi_j$, clearly
\begin{equation}
\int_{\sph} d \sigma ~ B_j \lesssim
\Phi(|v-v_*|)
\int_{2^{-j-1} |v - v_*|^{-1}}^{2^{-j} |v - v_*|^{-1}}  d \theta ~ \theta^{-1-2s} 
\lesssim 2^{2sj} \ang{v-v_*}^{\gamma+2s}.
\label{bjEST}
\end{equation}
(Note that this inequality holds true when $\Phi(v)$ equals either $|v|^{\gamma}$ or $\ang{v}^{\gamma}$ because of the assumption $\gamma + 2s \ge 0$).  Thus
\[ \left| T^j_{-}(g,h,f) \right| \lesssim 2^{2sj} \int_{\R^3} dv \int_{\R^3} dv_* M_* \ang{v-v_*}^{\gamma+2s} |g_*| |h f|. \]
Taking a geometric mean of the $(\eta,\delta)$-inequality for $\ang{v-v_*}^{\gamma+2s}$, i.e., \eqref{angineq1}, and the inequality 
$
M_* \ang{v-v_*}^{\gamma+2s}
 \lesssim \ang{v}^{\gamma+2s}
$ 
gives that
\[ M_*^{\frac{1}{2}} \ang{v-v_*}^{\gamma + 2s}  \lesssim \left( \ang{v}^{\gamma + 2s} \left( \eta \ang{v}^{\gamma + 2s} +  \delta^{-1} \right) \left( \delta \ang{v_*}^{\gamma + 2s} +  \eta^{-1} \right) \right)^{\frac{1}{2}}. \]
With Cauchy-Schwartz on the $v_*$ and $v$ integrals, for any $j$ and any $\eta > 0$,
\[ \left| T^j_{-}(g,h,f) \right| \lesssim 2^{2sj} \nsm g\nsm_{\delta,\eta} \nsm h\nsm_{\eta,\delta} \nsm f\nsm_{L^2_{\gamma+2s}}. \]
Since $T_{-}^j(g,h,f)$ is symmetric in $h$ and $f$, \eqref{tminusgh} must imply \eqref{tminusgf}.  As for \eqref{tminussmall}, note that the inequality $M_*^{1/2} \ang{v-v_*}^{\gamma+2s} \lesssim \ang{v}^{\gamma+2s}$ implies \eqref{tminussmall} by exactly the same reasoning used in the first two inequalities.  
\end{proof}

\begin{proposition}
The inequalities below are uniform for any integer $j$ and $\eta,\delta > 0$:
\begin{align}
 \left| T^j_{*}(g,h,f) \right| & \lesssim 2^{2sj} \nsm g\nsm_{\delta,\eta} \nsm h\nsm_{\eta,\delta} \nsm f\nsm_{L^2_{\gamma+2s}} \label{tstargh} \\
  \left| T^j_{*}(g,h,f) \right| & \lesssim 2^{2sj} \nsm g\nsm_{\delta,\eta} \nsm h\nsm_{L^2_{\gamma+2s}} \nsm f\nsm_{\eta,\delta}. \label{tstargf}
\end{align}
Moreover, the following inequality also holds uniformly for any integer $j$:
\begin{equation}
 \left| T^j_{*}(g,h,f) \right|  \lesssim 2^{2sj} \nsm g\nsm_{L^2} \nsm h\nsm_{L^2_{\gamma+2s}} \nsm f\nsm_{L^2_{\gamma+2s}}. \label{tstarsmall}
\end{equation}
\end{proposition}

\begin{proof}
As in the previous Proposition, the key to both inequalities is the symmetry between $h$ and $f$ coupled with two applications of Cauchy-Schwartz.  
The difference is that, this time, the Carleman representation will be used and the main integrals of will be over $v_*$ and $v'$.  In this case, the quantity of interest is
\[ 
\int_{E_{v_*}^{v'}}  d \pi_{v} ~ b \left( \frac{|v'-v_*|^2 - |v - v'|^2}{|v' - v_*|^2 + |v - v'|^2} \right) \frac{|v'-v_*|^2}{|v-v_*|^4} \chi_j(|v-v'|). 
\]
The support condition yields $|v-v'| \approx 2^{-j}$.  Moreover, since $b(\cos \theta)$ vanishes for $\theta\in [\pi/2,\pi]$, we have
$|v' - v_*| \ge |v' - v|$. 
Consequently, the condition \eqref{kernelQ} gives 
\begin{equation}
b \left( \frac{|v'-v_*|^2 - |v - v'|^2}{|v' - v_*|^2 + |v - v'|^2} \right)
\lesssim
\left(\frac{ |v - v'|^2}{|v' - v_*|^2} \right)^{-1-s}.
\label{bESTc}
\end{equation}
Thus, the integral is bounded by a uniform constant times
\[ 
\int_{E_{v_*}^{v'}}  d \pi_{v} ~\frac{|v'-v_*|^{2+2s}}{|v-v'|^{2+2s}} |v'-v_*|^{-2} \chi_j(|v-v'|) \lesssim 2^{2sj} |v'-v_*|^{2s}. 
\]
As a result
\[ \left| T^j_*(g,h,f) \right| \lesssim \int_{\R^3} dv' \int_{\R^3} dv_* M_* |g_* h' f'| 2^{2sj} \ang{v' - v_*}^{\gamma+2s}. \]
This leads directly to \eqref{tstargh} and \eqref{tstargf} in the same way that \eqref{tminusgh} and \eqref{tminusgf} were obtained.  Similarly, \eqref{tstarsmall} follows from $M_*^{1/2} \ang{v_* - v'}^{\gamma+2s} \lesssim \ang{v'}^{\gamma+2s}$. 
\end{proof}

\begin{proposition}
For any integer $j$ and $\delta, \eta > 0$,  we  uniformly have the
 estimate
\begin{equation} 
\left|  T^j_{+}(g,h,f) \right| \lesssim 2^{2sj} \nsm g\nsm_{\delta,\eta} \nsm h\nsm_{\eta,\delta} \nsm f\nsm_{L^2_{\gamma+2s}}. \label{tplusgh}
\end{equation}
Moreover, 
for any $j\ge 0$ we have the uniform estimate
\begin{equation} 
\left|  T^j_{+}(g,h,f) \right| \lesssim 2^{2sj} \nsm g\nsm_{L^2} \nsm h\nsm_{L^2_{\gamma+2s}} \nsm f\nsm_{L^2_{\gamma+2s}}. 
\label{tplussmall}
\end{equation}
\end{proposition}

\begin{proof}
The proof of \eqref{tplusgh} follows by interpolation.  First, we will assume that $f \in L^\infty$.  Noting that $M_*' \leq 1$, the integral of $B_j$ over $\sph$ may be estimated with \eqref{bjEST} to give that
\[ \left| T^j_{+}(g,h,f) \right| \lesssim 2^{2sj} \nsm f\nsm_{L^\infty} \int_{\R^3} dv \int_{\R^3} dv_* |g_*| |h| \ang{v-v_*}^{\gamma+2s}. \]
Using the $(\eta,\delta)$-inequality for $\ang{v-v_*}^{\gamma+2s}$ (inequality \eqref{angineq1}), it follows that
\begin{equation}
\begin{split}
 \left|  T^j_{+}(g,h,f) \right| \lesssim 2^{2sj} \nsm f\nsm_{L^\infty} \left( \int_{\R^3} dv_* |g_*| \right. & \left. \vphantom{\int} ( \delta \ang{v_*}^{\gamma+2s} + \eta^{-1}) \right) \\
 & \times \left( \int_{\R^3} dv |h| (\eta \ang{v}^{\gamma + 2s} + \delta^{-1} )\right).
\end{split} \label{plus1}
\end{equation}
For both $g$ and $h$, this is exactly the $L^1$-type estimate which will give the $\nsm \cdot \nsm_{\delta,\eta}$ or the $\nsm \cdot \nsm_{\eta,\delta}$ norm, respectively, after interpolation with $L^\infty$.

For the next estimate we use $g,h \in L^\infty$. In this case, the pre-post collisional change of variables yields
\[ 
T^j_{+}(g,h,f) = \int_{\R^3} dv \int_{\R^3} dv_* \int_{\sph} d \sigma B_j M_* g_*' f h'. 
\]
Then estimating the integral of $B_j$ as in \eqref{bjEST} grants
\[  \left|  T^j_{+}(g,h,f)  \right| \lesssim 2^{2sj} \nsm g\nsm_{L^\infty} \nsm h\nsm_{L^\infty} \int_{\R^3} dv \int_{\R^3} dv_* M_* \ang{v-v_*}^{\gamma+2s} |f|. \]
 After exploiting the inequality $M_* \ang{v-v_*}^{\gamma+2s} \lesssim M_*^{1/2} \ang{v}^{\gamma+2s}$, the integral over $v_*$ may be bounded above to yield
\begin{equation} \left|   T^j_{+}(g,h,f) \right| \lesssim 2^{2sj} \nsm g\nsm_{L^\infty} \nsm h\nsm_{L^\infty} \int_{\R^3} dv ~ |f| \ang{v}^{\gamma + 2s}. \label{plus2}
\end{equation}
Interpolating \eqref{plus1} and \eqref{plus2} gives \eqref{tplusgh} (the Riesz-Thorin interpolation theorem with weights found in Stein and Weiss \cite{MR0092943} suffices here; since $g$ and $h$ are in the same space, we do not even need a multilinear interpolation theorem since we can instead treat $g_* h$ as a function on $\R^3 \times \R^3$ which is either in a weighted $L^1$ space or in $L^\infty$).

Regarding \eqref{tplussmall}, the proof is essentially unchanged.  The only difference is, when $f$ is taken in $L^\infty$ and $g,h$ are taken in $L^1$, one may improve the inequality $M_*' \leq 1$ to $M_*' \lesssim M_*^{1-\epsilon}$  for any fixed $ \epsilon > 0$
 because $|v_* - v_*'|=|v-v'|  \leq 1$ when $j \ge 0$.  Setting $\eta = \delta = 1$ gives the result, since any weight $\ang{v_*}^\beta M_*^{1-\epsilon}$ is bounded.
\end{proof}

\begin{proposition}
The following inequality holds uniformly in $j$ and $\delta,\eta > 0$:
\begin{equation}
\left|  T^j_{+}(g,h,f) \right| \lesssim 2^{2sj} \nsm g\nsm_{\delta,\eta} \nsm h\nsm_{L^2_{\gamma+2s}} \nsm f\nsm_{\eta,\delta}. \label{tplusgf}
\end{equation}
\end{proposition}

\begin{proof}
The proof proceeds along the same lines as above, the difference is that this time we use the Carleman representation for $T^j_{+}$.  In this case, as in \eqref{bESTc}, we have
\begin{gather*}
 B_j\left(v-v_*, \frac{2v'-v-v_*}{|2v'-v-v_*|} \right)  
 \lesssim  \Phi(v-v_*) \chi_j(|v-v'|) \frac{|v-v_*|^{2+2s}}{|v-v'|^{2+2s}}
\\ 
 \lesssim \Phi(v'-v_*) \chi_j(|v-v'|) \frac{|v-v_*|^{2+2s}}{|v-v'|^{2+2s}}. 
\end{gather*}
The inequalities above  hold because $|v-v_*|^2 \approx |v'-v_*|^2$ and as in the analysis of \eqref{bESTc} we know that
 $|v'-v_*|^2 \ge |v-v'|^2$.  Thus
\begin{align*}
 \int_{E_{v_*}^{v'}} d \pi_{v} \frac{{ B}_j}{|v-v_*| | v'-v_*|} & \lesssim \Phi(v'-v_*) |v'-v_*|^{2s} \int_{E_{v_*}^{v'}} d  \pi_{v} ~ \frac{\chi_j(|v-v'|)}{ |v-v'|^{2+2s}} \\
& \lesssim 2^{2sj} \ang{v'-v_*}^{\gamma+2s}.
\end{align*}
As in the previous cases, this leads to the inequality
 \begin{align*}
\left|  T^j_{+}(g,h,f) \right| & \lesssim 2^{2sj} \nsm h\nsm_{L^\infty} \left( \int_{\R^3} dv_* |g_*| ( \delta \ang{v_*}^{\gamma+2s} + \eta^{-1}) \right) \\
& \hspace{100pt} \times \ \left( \int_{\R^3} dv' |f'| (\eta \ang{v'}^{\gamma + 2s} + \delta^{-1} )\right).
\end{align*}
In a similar manner, after a pre-post change of variables, we obtain
$$
\left|  T^j_{+}(g,h,f) \right|  \lesssim 2^{2sj} \nsm g\nsm_{L^\infty} \nsm f\nsm_{L^\infty} \int_{\R^3} dv' ~ |h'| \ang{v'}^{\gamma + 2s}.
$$
We finish the proof of \eqref{tplusgf} by interpolation in exactly the same manner that the corresponding inequality \eqref{tplusgh} was obtained in the previous proposition. 
\end{proof}

\subsection{Cancellations}
In this section, we seek to establish estimates for the differences $T^k_+  - T^k_{-}$ and $T^k_+ - T^k_*$. We wish the estimates to have good dependence on $k$ (in particular, we would like the norm to be a negative power of $2^k$), but this improved norm will be paid for by assuming differentiability of one of the functions $h$ or $f$.  The key obstacle to overcome in making these estimates is that the magnitude of the gradients of $h$ and $f$ must be measured in some nonisotropic way; this is a point of fundamental importance, as the scaling is imposed upon us by the structure of the ``norm piece'' $\ang{N f,f}$.

The scaling dictated by the problem is that of the paraboloid: namely, that the function $f(v)$ should be thought of as the restriction of some function $F$ of four variables to the paraboloid $(v, \frac{1}{2} |v|^2)$.  Consequently, the correct metric to use in measuring the length of vectors in $\R^3$ will be the metric on the paraboloid in $\R^4$ induced by the four-dimensional Euclidean metric.  To simplify the calculations, we will work directly with the function $F$ rather than $f$ and take its four-dimensional derivatives in the usual Euclidean metric.  This will be sufficient for our purposes since our Littlewood-Paley-type decomposition will give us a natural way to extend the projections $Q_j f$ into four dimensions while preserving the relevant differentiability properties of the three-dimensional restriction to the paraboloid.

To begin, it is necessary to find a suitable formula relating differences of $F$ at nearby points on the paraboloid to the various derivatives of $F$ as a function of four variables.  To this end, fix any two $v,v' \in \R^3$, and consider $\gamma : [0,1] \rightarrow \R^3$ and $\ext{\gamma} : [0,1] \rightarrow \R^4$ given by
\[ 
\gamma(\theta) \eqdef \theta v' + (1-\theta) v,
\quad 
\mbox{ and } 
\quad 
\ext{\gamma}(\theta) \eqdef  \left(\theta v' + (1-\theta)v, \frac{1}{2} \left| \theta v' + (1-\theta) v \right|^2 \right). \]
Note that $\ext{\gamma}$ lies in the paraboloid $\set{(v_1,\ldots,v_4) \in \R^4}{ v_4 = \frac{1}{2} (v_1^2 + \cdots v_3^2)}$, and that $\gamma(0) = v$ and $\gamma(1) = v'$.  Elementary calculations show that
\[ \frac{d \ext{\gamma}}{d \theta} = \left( v' - v, \ang{\gamma(\theta), v' - v} \right), 
\quad 
\mbox{ and }
\quad 
 \frac{d^2 \ext{\gamma}}{d \theta^2} = (0, |v'-v|^2). \]
Now we use the standard trick of writing the difference of $F$ at two different points in terms of an integral of a derivative (in this case the integral is along the path $\gamma$):
\begin{align} 
F\left(v',\frac{|v'|^2}{2} \right) - F\left(v, \frac{|v|^2}{2} \right) & = \int_0^1 d \theta ~ \frac{d}{d \theta} F(\ext{\gamma}(\theta)) \nonumber \\
& = \int_0^1 d \theta  \left( \frac{d \ext{\gamma}}{d \theta} \cdot  (\nabla_4 F) (\ext{\gamma}(\theta)) \right), \label{paraboladiff}
\end{align}
where the dot product on the right-hand side is the usual Euclidean inner-product on $\R^4$ and $\nabla_4$ is the four-dimensional gradient of $F$.  For convenience we define
\[ 
|\nabla_4|^k F(v_1,\ldots,v_4) \eqdef \sup_{|\xi| \leq 1} \left| \left(\xi \cdot \nabla_4 \right)^k F(v_1,\ldots,v_4) \right|, 
\quad
  k=0,1,2,
\]
where $\xi \in \R^4$ and $|\xi|$ is the usual Euclidean length. In particular, note that we have defined $|\nabla_4|^0 F = |F|$.

If $v$ and $v'$ are related by the collision geometry \eqref{sigma}, then $\ang{v-v',v'-v_*} = 0$, which yields that 
$$
\ang{\gamma(\theta),v'-v} = \ang{v_*,v'-v} - (1-\theta) |v-v'|^2.
$$  
Thus, whenever $|v-v'| \leq 1$, which is the case of interest, we have
$$
\left| \frac{d \ext{\gamma}}{d \theta} \right| \lesssim |v-v'| \ang{v_*}. 
$$
 Indeed, throughout this section we suppose that $|v-v'| \leq 1$ since this is the situation where our cancellation inequalities will be used.
 In particular, we have the following inequality for differences related by the collisional geometry:
\begin{align}
 \left| F\left(v',\frac{|v'|^2}{2}\right) - F\left(v, \frac{|v|^2}{2}\right)\right| 
 &
  \lesssim 
  \ang{v_*} |v-v'|  \int_0^1  d \theta ~ |\nabla_4| F (\ext{\gamma}(\theta)). \label{paradiff1} 
\end{align}
Furthermore
by subtracting the linear term from both sides of \eqref{paraboladiff} and using the integration trick iteratively on the integrand of the integral already appearing on the right-hand side we obtain that
\begin{equation}
\begin{split}
\left| F\left(v',\frac{|v'|^2}{2} \right) \right. & \left. - ~ F\left(v, \frac{|v|^2}{2}\right) -  \frac{d \ext{\gamma}}{d \theta}(0) \cdot \nabla_4 F(v) \right| 
\\
& \lesssim
\ang{v_*}^2 |v-v'|^2 \int_0^1 d \theta  ~\left[ |\nabla_4| F (\ext{\gamma}(\theta)) + | \nabla_4|^2 F (\ext{\gamma}(\theta)) \right]. 
\end{split}
\label{paradiff2}
\end{equation}
We note that, by symmetry, the same result holds when the roles of $v$ and $v'$ are reversed (which only changes the curve $\ext{\gamma}$ by reversing the parametrization: $\ext{\gamma}(\theta)$ becomes $\ext{\gamma}(1-\theta)$).  We will use these two basic cancellation inequalities to prove the following cancellation estimates for the trilinear form:

\begin{proposition}
Suppose $f$ is a Schwartz function on $\R^3$ given by the restriction of some Schwartz function $F$ 
on $\R^4$ to the paraboloid $(v, \frac{1}{2} |v|^2)$.  For each $k$, let $|\nabla_4|^k f$ be the restriction of $|\nabla_4|^k F$ to the same paraboloid.  Then, for any $j \geq 0$,
\begin{align}
 \left| (T^j_+- T_{-}^j)(g,h,f) \right| \ & \lesssim 2^{(2s-1)j} \nsm g\nsm_{L^2} \nsm h\nsm_{L^2_{\gamma+2s}} \nl \sum_{k=0}^1 |\nabla_4|^k f \nr_{L^2_{\gamma+2s}}. \label{cancelf}
\end{align}
If $s \geq \frac{1}{2}$, then it also holds uniformly for all $j \geq 0$ that
\begin{align}
 \left| (T^j_+- T_{-}^j)(g,h,f) \right| \ & \lesssim 2^{(2s-2)j} \nsm g\nsm_{L^2} \nsm h\nsm_{L^2_{\gamma+2s}} \nl \sum_{k=0}^2  |\nabla_4|^k f \nr_{L^2_{\gamma+2s}}. \label{cancelf2}
\end{align}
\end{proposition}

\begin{proof}
Write the difference $M_*' f' - M_* f = M_* M \left( (M^{-1} f)' - (M^{-1} f) \right)$.  If $f$ extends to $\R^4$, then 
$(M^{-1} f) (v_1,\ldots,v_4) = e^{v_4/2} f(v_1,\ldots,v_4)$.  With this extension %it follows that 
\begin{equation}
\label{nablaD}
|\nabla_4| (M^{-1} f) \lesssim M^{-1} |\nabla_4|^0 f + M^{-1} |\nabla_4| f.
\end{equation}
To prove \eqref{cancelf}, we first observe from \eqref{paradiff1} that for any fixed $\epsilon > 0$ we have
$$
 |(T_+^j-  T_{-}^j)(g,h,f) | 
  \lesssim 2^{-j} \int_0^1 d \theta \int_{\R^3} dv \int_{\R^3} dv_* \int_{\sph} d \sigma B_j M_*^{1-\epsilon} |g_*| |h| \tilde{f}(\theta v' + (1-\theta)v), 
$$
where $\tilde f \eqdef |\nabla_4|^0 f + |\nabla_4|f$.  
 The loss of $\epsilon$ arises from the factor of $\ang{v_*}$ in \eqref{paradiff1}, which also accounts for the factor of $2^{-j}$.  
 Additionally we have used
 the inequality 
 $M_*' M_*^{-1} = M M'^{-1} \lesssim e^{2^{-j}|v_*|}$
 which follows from
 \begin{equation}
 |v'|^2 - |v|^2 = -2 \ang{v', v - v'} - |v-v'|^2 = - 2 \ang{v_*,v-v'} - |v-v'|^2.
 \label{anglest}
\end{equation} 
   By Cauchy-Schwartz, it suffices to prove the estimates 
\begin{equation}
 \left( \int_0^1 d \theta \int_{\R^3} dv \int_{\R^3} dv_* \int_{\sph} d \sigma B_j M_*^{1-\epsilon} |g_*|^2 |h|^2 \right)^{\frac{1}{2}} 
 \lesssim 
 2^{sj} \nsm g\nsm_{L^2} \nsm h\nsm_{L^2_{\gamma+2s}}, 
 \label{easyterm}
\end{equation}
and
\begin{equation}
 \left( \int_0^1 d \theta \int_{\R^3} dv \int_{\R^3} dv_* \int_{\sph} d \sigma  B_j M_*^{1-\epsilon} |\tilde f(\theta v' + (1-\theta) v)|^2 \right)^\frac{1}{2}  \lesssim 
 2^{sj} \nsm  \tilde f\nsm_{L^2_{\gamma+2s}}. \label{covterm}
\end{equation}
The former inequality, \eqref{easyterm}, follows from \eqref{bjEST} as before. 
The latter uniform bound, \eqref{covterm}, 
 follows after the well-known change of variables $u = \theta v' + (1-\theta)v$, which changes  $v$ to $u$.  With \eqref{sigma}, we see  (with $\delta_{ij}$  the usual Kronecker delta) that 
$$
\frac{d u_i}{ dv_j} = (1-\theta) \delta_{ij} + \theta \frac{d v'_i}{ dv_j}
 = \left(1-\frac{\theta}{2}\right)  \delta_{ij} + \frac{\theta}{2} k_{j} \sigma_{i},
$$
with the unit vector $k = (v-v_*)/|v-v_*|$.  Thus the Jacobian is 
$$
\left| \frac{d u_i}{ dv_j}  \right| 
= 
\left(1-\frac{\theta}{2}\right)^2\left\{
\left(1-\frac{\theta}{2}\right) + 
\frac{\theta}{2} \ang{k, \sigma}
\right\}.
$$
Since  $b(\ang{k, \sigma}) = 0$ when $\ang{k ,\sigma} \leq 0$ from \eqref{kernelQ}, and  $\theta \in [0,1]$, 
it follows that 
the Jacobian of this change is bounded from below on the support of the integral \eqref{covterm}. 
But after this change of variable the old pole $k=(v-v_*)/|v-v_*|$ moves with the angle $\sigma$.  However it is easy to check that, when one takes $\tilde k = (u-v_*)/|u-v_*|$, 
$
1- \ang{ k,\sigma } \approx  1 - \left< \right. \! \tilde{k}, \sigma \! \left. \right>, 
$
meaning that the angle to the pole is comparable to the angle to $\tilde{k}$ (which does not vary with $\sigma$).
Thus the estimate analogous to \eqref{bjEST} will continue to hold after the change of variables, giving precisely the estimate in \eqref{covterm}.

The proof of the inequality \eqref{cancelf2} proceeds in exactly the same fashion, using \eqref{paradiff2} instead of \eqref{paradiff1}.  In this case, similar to \eqref{nablaD}, pointwise everywhere in $\R^3$ we have
$|\nabla_4|^2 (M^{-1} f) \lesssim M^{-1} |\nabla_4|^0 f + M^{-1} |\nabla_4| f + M^{-1} |\nabla_4|^2 f$. 
By subtracting off 
$$
\frac{d \ext{\gamma}}{d \theta}(0) \cdot \nabla_4 F(v),
$$
we easily obtain
\eqref{cancelf2}, at the price of still needing to estimate this term alone.
Here the extension of $M^{-1} f$ is once again inserted in the place of $F$ in \eqref{paradiff2}.  

Notice that $\frac{d \ext{\gamma}}{d \theta}(0)$
is linear in $v'-v$ and has no other dependence on $v'$; after multiplying it by $B_j$ and integrating with respect to $\sigma$, the symmetry of $B_j$ with respect to $\sigma$ around the direction $\frac{v-v_*}{|v-v_*|}$ forces all components of this integral to vanish except the component in the symmetry direction; in other words, one may replace $v-v'$ with $\frac{v-v_*}{|v-v_*|} \ang{v-v', \frac{v-v_*}{|v-v_*|}}$ in the expression for $\frac{d \ext{\gamma}}{d \theta}(0)$.  Since $\ang{v-v',v'-v_*} = 0$, the vector further reduces to $\frac{v-v_*}{|v-v_*|}\frac{|v-v'|^2}{|v-v_*|}$.  Simply observing 
\[ \left| \frac{v-v_*}{|v-v_*|}\frac{|v-v'|^2}{|v-v_*|} \right| \leq 2^{-2j} |v-v_*|^{-1}, \]
allows one to employ the same methods as above (easier in this case) to estimate this term for the projection of $\frac{d \ext{\gamma}}{d \theta}(0)$ onto the first three of the four coordinate directions.  To be precise, one must bound the following integral
\begin{equation} 
2^{-2j} \int_{\R^3} dv \int_{\R^3} dv_* \int_{S^2} d \sigma B_j M_*^{1-\epsilon} |g_*| |h| (|\nabla_4|^0 f + |\nabla_4|^1 f) 
|v-v_*|^{-1}.
\label{linear1} 
\end{equation}
Here, once again, we absorb any powers of $\ang{v_*}$  by $M_*^{-\epsilon}$.
The estimation of this integral proceeds exactly as was done for \eqref{tminussmall}, the only difference being the extra factor $|v-v_*|^{-1}$.  In particular, \eqref{linear1} is bounded above by
\[ 2^{(2s-2)j}  \int_{\R^3} dv \int_{\R^3} dv_* \Phi(v-v_*) |v-v_*|^{2s-1} M_*^{1-\epsilon} |g_*| |h| (|\nabla_4|^0 f + |\nabla_4|^1 f). \]
When $|v-v_*| \ge 1$,   $|v-v_*|^{2s-1} \le \ang{v-v_*}^{2s-1}$, and we obtain the upper bound
\[ 
2^{(2s-2)j} |g|_{L^2} |h|_{L^2_{\gamma+2s-1}} \sum_{k=0}^1 | |\nabla_4|^k f|_{L^2_{\gamma+2s-1}}, 
\]
which has a weight on $h$ and $f$ which is even better than desired.  As for the remaining piece where $|v-v_*| \leq 1$, we have the upper bound
\[ 
\int_{\R^3} dv_* M_*^{1-\epsilon} |g_*|  \Phi(v-v_*) |v-v_*|^{2s-1} {\mathbf 1}_{|v-v_*| \leq 1} \lesssim 
|M_*^{1-2 \epsilon} g_*|_{L^2},
\]
uniformly in $v$ by Cauchy-Schwartz as long as $\gamma + 2s - 1 > \frac{-3}{2}$ so that the $L^2$-norm of the weight $\Phi(v-v_*) |v-v_*|^{2s-1} {\mathbf 1}_{|v-v_*| \leq 1}$ in the variable $v_*$ is uniformly bounded above as a function of $v$.
Integrating in $v$ and applying Cauchy-Schwartz to separate $h$ and $f$ gives the desired inequality on this piece as well.

The fourth coordinate direction of $\frac{d \ext{\gamma}}{d \theta}(0)$ is given by $\ang{v,v'-v}$, which
 reduces to 
\[ \left| \ang{v, \frac{v-v_*}{|v-v_*|}\frac{|v-v'|^2}{|v-v_*|}} \right| \lesssim \frac{ |v-v'|^2 }{|v-v_*|} \ang{v_*}. \]
Therefore this term can also be handled in the same way to the previous cases.
\end{proof}

\begin{proposition}
As in the previous proposition, suppose $h$ is a Schwartz function on $\R^3$ which is given by the restriction of some Schwartz function in $\R^4$ to the paraboloid $(v, \frac{1}{2} |v|^2)$ and define $|\nabla_4|^k h$ analogously. For any $j \geq 0$, the inequality
\begin{align}
| T^j_+(g,h,f)   - T^j_*(g,h,f)|  & \lesssim 2^{(2s-1)j} \nsm g\nsm_{L^2} \nsm f\nsm_{L^2_{\gamma+2s}} \nl \sum_{k=0}^1  |\nabla_4|^k h \nr_{L^2_{\gamma+2s}}. \label{cancelh}
\end{align}
 Moreover, if $s \geq \frac{1}{2}$, then 
 uniformly for $j \geq 0$  we have
\begin{align}
| T^j_+(g,h,f)   - T^j_*(g,h,f)|  & \lesssim 2^{2sj} \nsm g\nsm_{L^2} \nsm f\nsm_{L^2_{\gamma+2s}} \nl \sum_{k=0}^2 2^{-kj - \epsilon_k j} |\nabla_4|^k h \nr_{L^2_{\gamma+2s}}, \label{cancelh2}
\end{align}
where $\epsilon_k \geq 0$ and $2s - k - \epsilon_k < 0$ for $k = 0,1,2$.
\end{proposition}

\begin{proof}
This proof follows in the same pattern that is by now well-established.  The new feature in this case is that the pointwise difference to examine is
\[  
\Delta \eqdef 4 \left\{ M h \Phi(v-v_*) |v-v_*|^3 -  M' h' \Phi(v'-v_*) |v' - v_*|^3 \right\}. 
\]
As before, we study the function of $\theta$ given by
\[ (M h)(\ext{\gamma}(\theta)) \Phi(\gamma(\theta) - v_*) |\gamma(\theta)-v_*|^3, \]
since $\Delta$ is a constant times the difference in the values of this function at $\theta = 0$ and $\theta = 1$ (note that we also employ the extension of $M h$ to $\R^4$ and make the de facto extension of the factor $\Phi(v-v_*) |v-v_*|^3$ to $\R^4$ by assuming the new function is constant in the fourth variable).
In terms of $\Delta$, our operator may be written as
\[ 
(T^j_+ - T^j_*)(g,h,f) = \int_{\R^3} dv' \int_{\R^3} dv_* \int_{E_{v_*}^{v'}} d \pi_v  
\frac{B_j M_* g_* (M')^{-1} f' \Delta}{|v'-v_*| |v-v_*|^4 ~ \Phi(v-v_*)} . 
\]
We use \eqref{paradiff1} and \eqref{paradiff2} again to estimate $\Delta$.  The additional fact required here is that
\[ |\nabla_4|^k (\Phi(v-v_*) |v-v_*|^3) \lesssim |v-v_*|^{-k} (\Phi(v-v_*) |v-v_*|^3), \]
which simply comes from differentiating with respect to $v$ (since the extension is taken to be constant in the fourth direction, the gradient $\nabla_4$ reduces to the usual three-dimensional gradient).
Applying \eqref{paradiff1} gives the estimate
\begin{align*}
 \frac{|\Delta|}{\Phi(v-v_*)|v-v_*|^3}  \lesssim \ang{v_*} & \frac{|v-v'|}{|v-v_*|} \int_0^1 d \theta \left| (Mh)(\ext{\gamma}(\theta))\right| \\
& + \ang{v_*} |v-v'| \int_0^1 d \theta ~ (M |\nabla_4|^0 h + M |\nabla_4|^1 h)(\ext{\gamma}(\theta)).  
\end{align*}
Again, $|v-v'| \lesssim |v-v_*|$ and $|v-v'| \lesssim 2^{-j}$, so as in the previous Proposition, we may  estimate the difference $T_+^j - T_*^j$ by a sum of three terms:
\begin{align*}
 I \eqdef \int_0^1 d \theta \int_{\R^3} dv' \int_{\R^3} dv_* \int_{E_{v_*}^{v'}} & d \pi_{v} ~  \frac{2^{-j}}{|v-v_*|}
   B_j(v-v_*, 2v' - v- v_*) 
   \\
& \hspace{-20pt} \times \frac{ M_*^{1-\epsilon} |g_* f'|}{|v-v_*| |v'-v_*|}  |\nabla_4|^0 h(\theta v' + (1-\theta)v)
\\
 II \eqdef 2^{-j} \int_0^1 d \theta \int_{\R^3} dv' \int_{\R^3} dv_* \int_{E_{v_*}^{v'}} & d \pi_{v} ~ 
B_j(v-v_*, 2v' - v- v_*) 
   \\
& \hspace{-20pt} \times \frac{ M_*^{1-\epsilon} |g_* f'|}{|v-v_*| |v'-v_*|}  |\nabla_4|^0 h(\theta v' + (1-\theta)v)
\\
 III \eqdef 2^{-j} \int_0^1 d \theta \int_{\R^3} dv' \int_{\R^3} dv_* \int_{E_{v_*}^{v'}} & d \pi_{v} ~ 
   B_j(v-v_*, 2v' - v- v_*) 
   \\
& \hspace{-20pt} \times \frac{ M_*^{1-\epsilon} |g_* f'| |\nabla_4| h(\gamma(\theta))}{|v-v_*| |v'-v_*|} .
\end{align*}
In each case, the extra factors of $\ang{v_*}$ and $M (M')^{-1}$ are absorbed into a single factor $M_*^{-\epsilon}$, 
possible because of \eqref{anglest} and the fact that $j \geq 0$. Each of the terms $II$ and $III$ is completely analogous to a corresponding quantity which arose in the previous proposition.
Splitting $M_*^{1-\epsilon} g_* f' h = (M_*^{1-\epsilon/2} g_* f') (M_*^{\epsilon/2} h)$ allows one to employ Cauchy-Schwartz just as was done for \eqref{easyterm} and \eqref{covterm} to separate the estimate into one integral involving only $g$ and $f$ (which may be estimated using exactly the same argument that handles \eqref{tstarsmall}) and one integral involving only $h$.  The integral involving only $h$ is handled by changing from the Carleman representation back to the sigma representation.  Since $\Phi(v' - v_*) \approx \Phi(v-v_*)$, the desired result follows directly from the same argument used to estimate \eqref{covterm}.  As for term $I$, the two relevant inequalities to establish are
\begin{align*}
|I| & \lesssim 2^{(2s-1)j} |g|_{L^2} |f \ang{v}^{\gamma+2s}|_{L^1} |h|_{L^\infty} \\
|I| & \lesssim 2^{(2s-1)j} |g|_{L^2} |f|_{L^\infty} |h \ang{v}^{\gamma+2s}|_{L^1}.
\end{align*}
The first is completely analogous to the extra estimate in the previous proposition to handle $s > \frac{1}{2}$ and follows as long as $\gamma + 2s - 1 > \frac{-3}{2}$, just as before.  The second may be transformed into exactly the same form after reverting from the Carleman representation to the sigma representation and using the same change of variables from the previous proposition as well.  Finally, these two results are interpolated to give the desired $L^2$ inequality.

Regarding the proof of \eqref{cancelh2}, once again, we split the estimate into two parts: one involving the linear correction term and one involving the integral of the second derivative.  This latter term may be estimated by a sum of expressions of the form 
\begin{align*}
\int_0^1 d \theta \int_{\R^3} dv' \int_{\R^3} dv_* \int_{E_{v_*}^{v'}} & d \pi_{v} ~  \frac{|v-v'|^2}{|v-v_*|^k}
   B_j(v-v_*, 2v' - v- v_*) 
   \\
&  \times \frac{M_*^{1-\epsilon} |g_* f'|}{|v-v_*| |v'-v_*|}  |\nabla_4|^\ell h(\theta v' + (1-\theta)v),
\end{align*}
where $k + \ell \leq 2$. Now
\[ \frac{|v-v'|^2}{|v-v_*|^k} \lesssim 2^{-j \ell} 2^{-j(2-\ell-k)} \frac{|v-v'|^{k}}{|v-v_*|^k} \lesssim  2^{-j \ell} 2^{-j(2-\ell-k)} \frac{2^{-j \rho} }{|v-v_*|^\rho},\]
for any $\rho \in [0,k]$.  Choose $\rho$ so that $\gamma + 2s - \rho > - \frac{3}{2}$ and $2s - 2 + k - \rho < 0$; note that this is always possible when $\gamma > - \frac{3}{2}$.  
Using this estimate and breaking into regions where $|v-v_*| \leq 1$ and $|v-v_*| \geq 1$ allows one to estimate the piece of $T_+^j - T_*^j$ governed by the right-hand side of \eqref{paradiff2} by a sum
\[ 
\sum_{\ell=0}^2 2^{2sj} |g|_{L^2} |f|_{L^2_{\gamma+2s}} 2^{-\ell j - \epsilon_\ell j} ||\nabla_4|^\ell h|_{L^2_{\gamma+2s}}, 
\]
where $\epsilon_{\ell}$ is some positive number (namely, $2 - \ell - k + \rho$) such that $2s - \ell - \epsilon_l < 0$.  We remark that improvements can be made to this estimate by instead taking the $L^\infty_v$ norm of $h$ when $\ell =0$ and $k=2$ and, after a Sobolev embedding, 
estimating isotropic derivatives of our anisotropic Littlewood-Paley decomposition.  With that, this %(our most restrictive) 
estimate can be made to work for all $\gamma + 2s >-1$ as well.

Last, but not least, is the analysis of the linear correction on the left-hand side of \eqref{paradiff2}.  Without loss of generality, we may instead choose 
to use the linear term
\[  
\frac{d \ext{\gamma}}{d \theta}(1) \cdot \nabla_4 F(\ext{\gamma}(1)), 
\]
in \eqref{paradiff2} instead of evaluating at $\theta = 0$.  Again, we exploit the symmetry of the kernel $B_j$ in the plane $E_{v_*}^{v'}$ around the point $v'$.
To simplify matters, the equality
\[  
\frac{d \ext{\gamma}}{d \theta}(1) \cdot \nabla_4 F(\ext{\gamma}(1)) = M' \Phi(v' - v_*) |v'-v_*|^3 \frac{d \ext{\gamma}}{d \theta}(1) \cdot \left( \nabla_4 h' - \frac{1}{4} e_4 h' \right),  
\]
does not involve derivatives of $\Phi(|v'-v_*|) |v'-v_*|^3$ by virtue of the fact that
\[ 
\frac{d}{d \theta} |\gamma(\theta) - v_*|^2 = 2 \ang{v' - v, \gamma(\theta) - v_*} = 0, 
\]
when $\theta = 1$ because $\ang{v-v',v' - v_*} = 0$ (here $e_4$ is the unit vector pointing in the fourth direction in $\R^4$).  We are therefore left to estimate the integral
\begin{align*} \int_{\R^3} dv' \int_{\R^3} dv_* \ & \int_{E_{v_*}^{v'}} d \pi_v  \  
B_j(v-v_*, 2v' - v- v_*) \frac{ M_* g_*}{|v-v_*||v'-v_*|}
\\
& \times  \frac{\Phi(v'-v_*) |v'-v_*|^3}{\Phi(v-v_*) |v - v_*|^3} f' \frac{d \ext{\gamma}}{d \theta}(1) \cdot \left( \nabla_4 h' - \frac{1}{4} e_4 h' \right).
\end{align*}
This integrand still has the property that as $v$ varies on circles of constant distance to $v'$, the entire integrand is constant except for $\frac{d \ext{\gamma}}{d \theta}(1)$.  If we write $\frac{d \ext{\gamma}}{d \theta}(1)$ as a sum of two vectors, one lying in the span of the first three directions and the second pointing in the fourth direction, it follows that we may replace the former vector by its projection onto the direction determined by $v' - v_*$.  But since the original vector points in the direction $v-v'$, the projection vanishes.  In other words, only the projection of $\frac{d \ext{\gamma}}{d \theta}(1)$ pointing in the fourth direction remains.  Its magnitude in this direction is exactly $\ang{v',v'-v}$, the corresponding integral of which over $v$ also vanishes by symmetry.  Thus, the linear term integrates to zero in this case.
\end{proof}

Finally, let us observe that for $j \geq 0$, we have the following uniform inequalities
\begin{align}
| T^j_+(g,h,f)  - T^j_-(g,h,f)|  
&  \lesssim 2^{(2s-2)j} \nsm g\nsm_{\delta,\eta} \nsm h \nsm_{\eta,\delta}  \nl \sum_{k=0}^2 |\nabla_4|^k f \nr_{L^2_{\gamma+2s}} \label{othercptineq} \\
 | T^j_+(g,h,f)  - T^j_*(g,h,f)|  
 &  \lesssim 2^{2sj} \nsm g\nsm_{\delta,\eta} \nsm f\nsm_{\eta,\delta}  \nl \sum_{k=0}^2 2^{-kj - \epsilon_k j} |\nabla_4|^k h \nr_{L^2_{\gamma+2s}}. \label{compactineq} \end{align}
These estimates follow immediately from the work above after taking into account
\eqref{expand}.
{This observation will be necessary to establish favorable estimates for the ``compact piece'' given by $\ang{\Gamma(g,M),g}$ as well as for the term 
$\ang{\Gamma(g,g),M}$. }

\section{The anisotropic Littlewood-Paley decomposition}
\label{sec:aniLP}

In this section we introduce the Littlewood-Paley decomposition.  Rather than the standard Littlewood-Paley decomposition on $\R^3$, we will instead use a decomposition which is implicitly adapted to the induced Laplacian on the paraboloid $(v, \frac{1}{2} |v|^2) \subset \R^4$.  The main reason for doing so is that an analysis of the norm piece 
$\ang{Ng,g}$ from \eqref{normpiece}
shows an inherent anisotropy in the directions of differentiation.  
(The analogous problem for the Landau equation is anisotropic \cite{MR1946444}, and the sharp norm for the non cut-off problem was conjectured to be anisotropic in \cite{MR2322149}.)  Rather than work directly on the paraboloid, though, it turns out to be somewhat simpler to think of the Littlewood-Paley decomposition we use as being a 4-dimensional Euclidean decomposition restricted to the paraboloid. 
 In the process, we will effectively construct a corresponding extension operator for bandlimited functions (i.e., functions with dyadically localized frequency support) on the paraboloid into bandlimited functions on $\R^4$ in the neighborhood of the paraboloid (and the extension operator satisfies favorable Sobolev-type inequalities).  This trick will allow us to use the estimates from the previous section without any serious concern for the deeper geometric aspects of our nonisotropic construction (which contrasts with the approach of Klainerman and Rodnianski \cite{MR2221254}).

\subsection{Associated calculations and definitions}\label{sec:acd}

Throughout the remainder of this section, we will use the variables $v$ and $v'$ to refer to independent points in $\R^3$, meaning that we will not assume in this section that they are related by the collision geometry.  The reason we choose to use these variable names is that they give a hint about where the Littlewood-Paley projections will be later applied in situations which involve the collision geometry explicitly.

First we develop some tools for calculus on this paraboloid.  For $v \in \R^3$, let $\ext{v} \eqdef (v,\frac{1}{2} |v|^2) \in \R^4$.  Perhaps the most useful such tool is the following:  for any $v \in \R^3$, let $\tau_v : \R^3 \rightarrow \R^3$  and $\ext{\tau}_v : \R^3 \rightarrow \R^4$ be given by
\[ \tau_v (u) \eqdef u - (1- \ang{v}^{-1}) \ang{v,u} |v|^{-2} v,
\quad \mbox{ and } 
\quad \ext{\tau}_v u \eqdef (\tau_v u,  \ang{v}^{-1} \ang{v,u}). \]
These mappings should be thought of as sending the hyperplane $v_4 = 0$ 
to the hyperplane tangent to the paraboloid $(v, \frac{1}{2} |v|^2)$ at the point $\ext{v}$.  It's routine to check that $\ang{v, \tau_v u} = \ang{v}^{-1} \ang{v,u}$ and $|\tau_v u|^2 = |u|^2 - \ang{v}^{-2} \ang{v,u}^2$, which implies
$$
|\ext{\tau}_v u|^2
=
|\tau_v u|^2 + \ang{v, \tau_v u}^2 = |u|^2,
$$ 
meaning that $\ext{\tau}_v$ is an isometry from one hyperplane to the other. 
Moreover, it is easy to check that
\[ \ext{v + \tau_v u} = \ext{v} + \ext{\tau}_v u + \frac{1}{2} |u|^2 e_4, \]
where $e_4 \eqdef (0,0,0,1)$.
The last term will be thought of as a perturbation, and is the basis for all the analysis that follows:
\begin{proposition}
Suppose $\varphi$ is any fixed, smooth function supported on the unit ball in $\R^4$.  For any $j \geq 0$, the expansion
\begin{equation} \varphi \left(2^{j} \left( \ext{v + 2^{-j} \tau_v u} - \ext{v} \right) \right) = \varphi( \ext{\tau}_v u ) + 2^{-j-1} \varphi_4 (\ext{\tau}_v u) + E_j(u,v), \label{asympt}
\end{equation}
holds, where $\varphi_4 \eqdef \frac{\partial \varphi}{\partial v_4}$ and $|E_j(u,v)| \lesssim 2^{-2j}$ and is supported on a set $|u| \leq 2$.
Moreover, if $|u| \lesssim 1$, then 
\[ \left| \ang{v + 2^{-j} \tau_v u}^\beta - \ang{v}^\beta - 2^{-j} \beta \ang{v}^{-1} \ang{v,u} \ang{v}^{\beta-2} \right| \lesssim 2^{-2j} \ang{v}^{\beta-2}, \]
uniformly in $v,u$, and $j \geq 0$ for any fixed $\beta \in \R$.
\end{proposition}

\begin{proof}
Note that $\ang{v}^{\beta} = (1 + 2 (\ext{v})_4)^{\beta/2}$.  With this observation and the equality 
$$
\ext{v+2^{-j} \tau_v u} = \ext{v} + 2^{-j} \ext{\tau}_v u + 2^{-2j-1} |u|^2 e_4,
$$ 
both inequalities  immediately follow from Taylor's theorem with remainder.
\end{proof}

These expansions may be utilized to estimate a variety of integrals of the form
\[ 
\int_{\R^3} dv' ~ 2^{3j} \varphi(2^{j} (\ext{v'} - \ext{v})) \ang{v'}^\beta  
= 
\ang{v}^{-1} \! \int_{\R^3} du ~ \varphi(2^{j} (\ext{v + 2^{-j} \tau_v u} - \ext{v})) \ang{v + 2^{-j} \tau_v u}^\beta. 
\]
The right-hand side follows by the change of variables $v' \mapsto v + 2^{-j} \tau_v u$ as suggested by the proposition above.  If it is assumed that $\varphi$ is radial, then $\varphi_4$ is odd, meaning that $\varphi_4(\ext{\tau}_v u)$ is an odd function of $u$ (and, by assumption, $\varphi(\ext{\tau}_v u)$ is an even function of $u$).  Likewise $\ang{v}^\beta$ is (trivially) an even function of $u$, while $\ang{v,u} \ang{v}^{\beta-3}$ is odd.  Thus if both factors in the right-hand side of the above equality are expanded as a function of $u$ by means of the previous proposition, there are a number of terms which automatically cancel (namely, the product of the zeroth-order term for one and the first-order term for the other).  Thus the result is that
\begin{align}
 \left| \int_{\R^3} dv' 2^{3j} \varphi(2^{j} (\ext{v} - \ext{v'})) \ang{v'}^\beta  - \ang{v}^{\beta-1} \int_{\R^3} du  \varphi(\ext{\tau}_v u)  \right| \lesssim 2^{-2j} \ang{v}^{\beta-1}. \label{intest}
\end{align}
Moreover, if $\varphi$ is assumed to be radial, it is also true that $\int du \varphi(\ext{\tau}_v u) $ is a constant 
(since $\ext{\tau}_v$ is an isometry of $\R^3$ with some hyperplane in $\R^4$ passing through the origin).   It will be assumed that $\varphi$ is chosen so that this constant equals $1$.

We now define the following anisotropic Littlewood-Paley projections: 
\begin{align*} 
P_j f(v) & \eqdef \int_{\R^3} f(v') 2^{3j} \varphi(2^j (\ext{v} - \ext{v'})) \ang{v'} dv', 
\quad j\ge 0,
\\
Q_j f(v) & \eqdef P_j f(v) - P_{j-1} f(v), \quad
j\ge 1, 
\quad 
Q_0 \eqdef P_0. 
\end{align*}
Given the calculations above, if $f$ is a Schwartz function, then clearly $P_j f(v) \rightarrow f(v)$ as $j \rightarrow \infty$ by the normalization condition on $\varphi$ and the estimate \eqref{intest}.  That same estimate proves the fundamental fact that
\begin{equation} \left( \int_{\R^3} dv (P_j f(v))^p \ang{v}^\beta \right)^{\frac{1}{p}}  \lesssim \left( \int_{\R^3} dv (f(v))^2 \ang{v}^{\beta} \right)^{\frac{1}{p}}, \label{boundedlp}
\end{equation}
uniformly in $j \geq 0$ for any fixed $\beta \in \mathbb{R}$ and any $p \in [1,\infty)$ (as a consequence of Schur's test for integral operators).

\subsection{Square function estimates related to the norm \eqref{normdef}}
The next step is to establish a favorable inequality relating the associated square functions, i.e., 
$\{\sum_j 2^{2sj} (Q_jf)^2(v)\}^{1/2}$,
to an integral involving a squared difference of the form found in \eqref{normdef}.  To this end, notice that
\begin{align*}
 \frac{1}{2} \int_{\R^3}  \! \! \! dv \! \int_{\R^3}  \! \! \! dv' \! \int_{\R^3} \! \! \! dz  & (f(v) - f(v'))^2 q_j(\ext{z},\ext{v}) q_j(\ext{z}, \ext{v'}) \ang{v} \ang{v'} \ang{z}^\beta  \\ 
 & =   - \int_{\R^3} \! \! dv  (Q_j f (v))^2 \ang{v}^\beta + \int_{\R^3} \! \! dv Q_j (f^2)(v) Q_j(1)(v) \ang{v}^\beta,
\end{align*}
with 
$
q_j(\ext{z},\ext{v})=\left[ 2^{3j} \varphi(2^j(\ext{v} - \ext{z})) - 2^{3j-3} \varphi(2^{j-1}(\ext{v} - \ext{z})) \right]$.  
To see this,  expand the square $(f(v) - f(v'))^2$ and exploit the symmetry of the integral in $v$ and $v'$.
Next, recall that $|Q_j(1)(v)| \lesssim 2^{-2j}$ by virtue of \eqref{intest} and the triangle inequality.  
But 
\begin{align*}
\int_{\R^3} dv |Q_j(1)(v)| |Q_j(f^2)(v)| \ang{v}^{\beta} \lesssim 2^{-2j} \int_{\R^3} dv' (f(v'))^2 \ang{v'}^{\beta},
\end{align*}
where the last inequality follows from the pointwise estimate for $Q_j(1)$ and the boundedness of $Q_j$ on $L^1$ (i.e., the inequality \eqref{boundedlp} for $p=1$).
Consequently, for any $j > 0$, we have %for the square of the norm of $Q_j f$
 the following estimate:
\begin{align*}
 \left| \frac{1}{2} \int_{\R^3}  \! \! \! dv \! \int_{\R^3}  \! \! \! dv' \! \int_{\R^3} \! \! \! dz \right. & (f(v) - f(v'))^2 q_j(\ext{z},\ext{v}) q_j(\ext{z}, \ext{v'}) \ang{v} \ang{v'} \ang{z}^\beta  \\ 
 + & \left.  \int_{\R^3} \! \! dv  (Q_j f (v))^2 \ang{v}^\beta \right| \lesssim 2^{-2j} \int dv |f(v)|^2 \ang{v}^{\beta}.
\end{align*}
This inequality allows us to compare the weighted norm of $Q_j f$ to an integral involving the squared difference $(f(v)-f(v'))^2$, the right-hand side of the inequality above being the lower order error in this comparison.

The next step is to estimate the integral of $\ang{z}^\beta q_j(\ext{v} - \ext{z}) q_j(\ext{v}-\ext{z})$ with respect to $z$.
By the same reasoning that lead to \eqref{intest} (namely, by making the change of variables $z \mapsto v+ 2^{-j} \tau_v w$), it must be the case that the inequality
\[ \int_{\R^3} dz \ang{z}^\beta |q_j(\ext{z} - \ext{v})| \lesssim \ang{v}^{\beta-1}, \] 
holds uniformly for $j \geq 0$ (in fact, this follows directly from \eqref{intest} and the triangle inequality).  Moreover, clearly $|q_j(\ext{z} - \ext{v})| \lesssim 2^{3j}$ as well.  Thus, for any $j \geq 0$, 
\[ \left| \int_{\R^3} dz \ang{z}^\beta q_j(\ext{v} - \ext{z}) q_j(\ext{z} - \ext{v'}) \right| \lesssim 2^{3j} \ang{v}^{\beta-1}, \]
and the integral is supported on some set $| \ext{v} - \ext{v'} | \lesssim 2^{-j}$.
Consequently
\begin{align*}
 \sum_{j=0}^\infty 2^{2sj} & \int_{\R^3} dv (Q_jf(v))^2 \ang{v}^\beta \\ & 
\lesssim \nsm f\nsm_{L^2_\beta}^2 + \int_{\R^3} dv \int_{\R^3} dv' \frac{(f(v) - f(v'))^2}{| \ext{v} - \ext{v'} |^{3+2s}} \chi( \ext{v} - \ext{v'}) \ang{v}^{\beta} \ang{v'}, 
\end{align*}
where $\chi(\ext{v} - \ext{v'})$ is some bounded, nonnegative function supported on a ball in $\R^4$ given by $|\ext{v} - \ext{v'}| \lesssim 1$. Note that, if the scaling function $\varphi$ is suitably rescaled on $\R^4$ (i.e., $\varphi(v)$ is replaced with $\varphi(cv)$ where $c$ is the implicit constant in $|\ext{v} - \ext{v'}| \lesssim 1$), then $\chi$ may be taken to be the characteristic function of the unit ball in $\R^4$.  Furthermore,  $|\ext{v} - \ext{v'}| \lesssim 1$ implies $\ang{v} \approx \ang{v'}$ so the powers of $\ext{v}$ and $\ext{v'}$ may be redistributed at will.  In particular, then, it must hold that the exponentially weighted, squared sum of Littlewood-Paley $Q_j$'s is dominated by the square of our fundamental norm \eqref{normdef}:
\begin{equation*}
 \sum_{j=0}^\infty 2^{2sj}  \int_{\R^3} dv (Q_jf(v))^2 \ang{v}^{\gamma+2s} \lesssim \nsm f\nsm_{N^{s,\gamma}}^2.
\end{equation*}
Likewise, since each $Q_j$ has a natural extension to $\R^4$ (obtained by replacing $\ext{v}$ by an arbitrary 4-vector in the definitions of $P_j$ and $Q_j$), it is natural to ask a similar question about the 4-derivatives $|\nabla_4|^k Q_j$. In this case, it is easy to see that the same estimates must hold (as the derivatives simply fall on $\varphi$), except that a factor of $2^j$ is introduced for every derivative.  In particular, then, we must also have
\begin{equation}
 \sum_{j=0}^\infty 2^{2(s-k)j}  \int_{\R^3} dv (|\nabla_4|^k Q_jf)^2(v) \ang{v}^{\gamma+2s} \lesssim \nsm f\nsm_{N^{s,\gamma}}^2, \qquad k = 0,1,2, \label{lpbynorm}
\end{equation}
uniformly in $j \geq 0$ as well.  These inequalities will be fundamentally important in the next section, where all the estimates so far are combined to yield the stated upper bound of the trilinear form $\ang{\Gamma(g,h),f}$.

\section{Upper bounds for the trilinear form}
\label{sec:upTRI}

In this section, we establish Lemma \ref{NonLinEst} for the nonlinear term $\Gamma$ as well as 
Lemma \ref{sharpLINEAR}
%the upper bounds \eqref{normupper} and \eqref{compactupper} 
which provide estimates related to the splitting $L = N + K$ of the linear term.

\subsection{The main upper bound inequality}

We'll write 
$$
f = P_0 f + \sum_{j=1}^\infty Q_j f \eqdef \sum_{j=0}^\infty f_j,
$$ 
and likewise for $h$, then expand the nonlinear term: 
\begin{align}
 & \ang{\Gamma(g,h),f}  = \sum_{j,j' = 0}^\infty \ang{\Gamma(g,h_{j'}),f_j} \nonumber \\
& \hspace{30pt} = \sum_{j=0}^\infty \ang{\Gamma(g,h_j),f_j} + \sum_{l=1}^\infty \sum_{j=0}^\infty  \left\{ \ang{\Gamma(g,h_{j+l}),f_j} + \ang{\Gamma(g,h_j),f_{j+l}} \right\}. \label{mainexpand}
\end{align}
Consider the sum over $l$ of the terms $\ang{\Gamma(g,h_{j+l}),f_j}$ for fixed $j$.  We expand $\Gamma$ as a series by introducing the cutoff around the singularity of $b$ in terms of $T^k_+$ and $T^k_{-}$:
\begin{align}
 \sum_{l=1}^\infty \ang{\Gamma(g,h_{j+l}),f_j} 
 & = 
 \sum_{k=-\infty}^\infty  \sum_{l=1}^\infty 
\left\{ T^k_{+}(g,h_{j+l},f_j) - T^k_{-}(g,h_{j+l},f_j) \right\} \nonumber \\
& = 
\sum_{k=-\infty}^0 \left\{ T^k_{+}(g, h - P_j h, f_j) - T^k_{-}(g, h - P_j h, f_j) \right\}  \label{farsing} \\
& \hspace{30pt} 
+ 
\sum_{l=1}^\infty \sum_{k=1}^\infty \left\{ T^k_{+}(g,h_{j+l},f_j) - T^k_{-}(g,h_{j+l},f_j) \right\}. \label{nearsing}
\end{align}
Here we have used the basic telescoping property $h - P_j h =  \sum_{l=1}^\infty h_{j+l}$.
Also, throughout the manipulation, the order of summation may be rearranged with impunity since the estimates we employ below will imply that the sum is absolutely convergent when $g,h,f$ are all Schwartz functions.  Regarding the terms \eqref{farsing}, the inequalities \eqref{tplusgh} and \eqref{tminusgh} dictate that
\[ \sum_{k=-\infty}^0 |T^k_{+}(g, h - P_j h, f_j) - T^k_{-}(g, h - P_j h, f_j)| \lesssim \nsm g\nsm_{\delta, \eta} \nsm h - P_j h\nsm_{\eta,\delta} \nsm f_j\nsm_{L^2_{\gamma+2s}}. \]
Since $\nsm P_j h\nsm_{\eta,\delta} \lesssim \nsm h\nsm_{\eta,\delta}$ (a consequence of \eqref{boundedlp}), one may conclude that
\begin{gather*}
 \sum_{j=0}^\infty \sum_{k=-\infty}^0  
 |T^k_{+}(g, h - P_j h, f_j)  - T^k_{-}(g, h - P_j h, f_j)|  
 \lesssim \nsm g\nsm_{\delta, \eta} \nsm h\nsm_{\eta,\delta} \sum_{j=0}^\infty \nsm f_j\nsm_{L^2_{\gamma+2s}}
 \\
 \lesssim \nsm g\nsm_{\delta, \eta} \nsm h\nsm_{\eta,\delta} \left| \sum_{j=0}^\infty 2^{2sj} \nsm f_j\nsm ^2_{L^2_{\gamma+2s}} \right|^\frac{1}{2}
%\\
\lesssim 
 \left(
 \nsm g\nsm_{L^2} \nsm h\nsm_{L^2_{\gamma+2s}}
 + 
 \nsm g\nsm_{L^2_{\gamma+2s}} \nsm h\nsm_{L^2} 
\right) \nsm f\nsm_{N^{s,\gamma}}.
\end{gather*}
This is just Cauchy-Schwartz.
The expansion of the product $\nsm g\nsm_{\delta, \eta} \nsm h\nsm_{\eta,\delta}$ proceeds by \eqref{expand} and 
the favorable comparison of the square-function norm of $f$ to the norm $\nsm  f \nsm_{N^{s,\gamma}}$ is provided by \eqref{lpbynorm}.
As for the terms \eqref{nearsing}, when $k \leq j$ a similar approach holds; namely, \eqref{tplussmall} and \eqref{tminussmall} guarantee that
\[ \sum_{k=1}^{j} \left| T^k_{+}(g,h_{j+l},f_j) - T^k_{-}(g,h_{j+l},f_j) \right| \lesssim 2^{2sj} \nsm g\nsm_{L^2} \nsm h_{j+l}\nsm_{L^2_{\gamma+2s}} \nsm f_j\nsm_{L^2_{\gamma+2s}}, \]
(we have used the trivial facts that $2^{2sk} = 2^{2sj} 2^{2s(k-j)}$ and 
$\sum_{k=1}^{j} 2^{2s(k-j)}\lesssim 1$). 
In particular, this inequality may be summed over $j$; another application of Cauchy-Schwartz gives that
\begin{align*}
 \sum_{j=0}^\infty \sum_{k=1}^{j} & \left| T^k_{+}(g,h_{j+l},f_j) - T^k_{-}(g,h_{j+l},f_j) \right|
 \\ 
 & \lesssim 2^{-sl} \nsm g\nsm_{L^2} \left| \sum_{j=0}^\infty 2^{2s(j+l)} \nsm h_{j+l}\nsm ^2_{L^2_{\gamma+2s}} \right|^{\frac{1}{2}} \left| \sum_{j=0}^\infty 2^{2sj} \nsm f_j\nsm ^2_{L^2_{\gamma+2s}} \right|^{\frac{1}{2}}
 \\
& \lesssim 2^{-2sl} \nsm g\nsm_{L^2} \nsm h\nsm_{N^{s,\gamma}} \nsm f\nsm_{N^{s,\gamma}}.
\end{align*}
This estimate may clearly also be summed over $l \geq 0$.

A completely analogous argument may be used to expand $\Gamma$ for the terms in \eqref{mainexpand} of the form $\ang{\Gamma(g,h_j),f_{j+l}}$ in terms of $T^k_+ - T^k_*$; 
\begin{align}
\sum_{l=1}^\infty \ang{\Gamma(g,h_j),f_{j+l}} & = \sum_{k=-\infty}^\infty \sum_{l=1}^\infty (T^k_+ - T^k_*)(g,h_j,f_{j+l}) \nonumber \\
& = \sum_{k=-\infty}^0 (T^k_+ - T^k_*)(g,h_j,f - P_j f) \label{sum1} \\
& \hspace{30pt} + \sum_{l=1}^\infty \sum_{k=1}^\infty (T^k_+ - T_*^k)(g,h_j,f_{j+l}) \label{sum2}.
\end{align}
In this case the estimates \eqref{tplusgh} and \eqref{tstargh} are used to handle the terms \eqref{sum1} just as the corresponding terms \eqref{farsing} were handled.  Regarding the sum \eqref{sum2}, now \eqref{tplussmall} and \eqref{tstarsmall} are used to estimate the sum analogously to the estimates of \eqref{nearsing}.  The only difference is that the roles of $h$ and $f$ are now reversed.

Recalling the original expansion of $\ang{\Gamma(g,h),f}$ it is clear that the only terms that remain to be considered are the following:
\begin{align} 
\tilde{\Gamma}(g,h,f)  \eqdef 
&  \sum_{l=0}^\infty \sum_{j=0}^\infty \sum_{k=j+1}^\infty \left\{T^k_{+}(g,h_{j+l},f_j) - T^k_{-}(g,h_{j+l},f_j)\right\} 
\label{maincancelf} \\
 & + \sum_{l=1}^\infty \sum_{j=0}^\infty \sum_{k=j+1}^\infty \left\{ T^k_{+}(g,h_{j},f_{j+l}) - T^k_{*}(g,h_{j},f_{j+l}) \right\} . \label{maincancelh}
\end{align}
In other words, we have already established the inequality
\begin{align*}
 \left| \ang{\Gamma(g,h),f} - \tilde{\Gamma}(g,h,f) \right| \lesssim & \ \nsm g\nsm_{L^2} \nsm h\nsm_{N^{s,\gamma}} \nsm f\nsm_{N^{s,\gamma}} \\
 & + \nsm g\nsm_{L^2_{\gamma+2s}} \left[ \nsm h\nsm_{L^2} \nsm f\nsm_{N^{s,\gamma}} + \nsm h\nsm_{N^{s,\gamma}} \nsm f\nsm_{L^2} \right].
\end{align*}
Then the terms on the right-hand side of \eqref{maincancelf} and the terms \eqref{maincancelh} are both treated by the cancellation inequalities.  The terms \eqref{maincancelf}, for example, are handled by \eqref{cancelf} and \eqref{cancelf2} (when $s \geq \frac{1}{2}$).  For any fixed $l,j,k$, we have 
\begin{align*}
 \left| T^k_{+}(g,h_{j+l},f_j) \right. & \left. - T^k_{-}(g,h_{j+l},f_j) \right|
\\ &  \lesssim 2^{2sk} \nsm g\nsm_{L^2} \nsm h_{j+l}\nsm_{L^2_{\gamma+2s}} \nl \sum_{i=0}^2 2^{-ik-\epsilon_ik} |\nabla_4|^i f_j \nr_{L^2_{\gamma+2s}},
\end{align*}
for appropriate nonnegative $\epsilon_i$'s satisfying $2s - i - \epsilon_i < 0$ (note that both \eqref{cancelf} and \eqref{cancelf2} may be written in this form).  In either case, there is decay of the norm as $k \rightarrow \infty$ since $2s - i - \epsilon_i < 0$, so 
\begin{align*}
 \sum_{k=j+1}^\infty  \left| T^k_{+}(g,h_{j+l},f_j) \right. & \left. - T^k_{-}(g,h_{j+l},f_j) \right| 
\\ & \lesssim 2^{2sj} \nsm g\nsm_{L^2} \nsm h_{j+l}\nsm_{L^2_{\gamma+2s}} \nl \sum_{i=0}^2 2^{-ji} |\nabla_4|^i f_j \nr_{L^2_{\gamma+2s}}.
\end{align*}
Just as before, Cauchy-Schwartz is applied to the sum over $j$.  In this case $2^{(2s-i)j}$ is written as $2^{(s-i)j} 2^{s(j+l)} 2^{-sl}$; the first factor goes with $f$, the second with $h$, and the third remains for the sum over $l$.  Once again \eqref{lpbynorm} is employed.  Finally, the factor of $2^{-sl}$ allows one to finish the sum over $l$.

The desired bound for the nonlinear term is completed by performing summation of the terms \eqref{maincancelh}.  The pattern of inequalities is exactly the same as the one just described, this time using \eqref{cancelh} and \eqref{cancelh2}.  In particular, one has that
\begin{align*}
 \left| T^k_{+}(g,h_{j},f_{j+l}) \right. & \left. - T^k_{*}(g,h_{j},f_{j+l}) \right|
\\ &  \lesssim 2^{2sk} \nsm g\nsm_{L^2} \nsm f_{j+l}\nsm_{L^2_{\gamma+2s}} \nl \sum_{i=0}^2 2^{-ik - \epsilon_i k} |\nabla_4|^i h_j \nr_{L^2_{\gamma+2s}}.
\end{align*}
Again, both \eqref{cancelh} and \eqref{cancelh2} may be written in this form with  
$\epsilon_i>0$ satisfying $2s - i - \epsilon_i < 0$, leading to the corresponding inequality for the sum over $k$:
\begin{align*}
 \sum_{k=j+1}^\infty  \left| T^k_{+}(g,h_{j+l},f_j) \right. & \left. - T^k_{-}(g,h_{j+l},f_j) \right| 
\\ & \lesssim 2^{2sj} \nsm g\nsm_{L^2} \nsm h_{j+l}\nsm_{L^2_{\gamma+2s}} \nl \sum_{i=0}^2 2^{-ij} |\nabla_4|^i f_j \nr_{L^2_{\gamma+2s}}.
\end{align*}
The same Cauchy-Schwartz estimate is used for the sum over $j$; there  is exponential decay allowing the sum over $l$ to be estimated.  The end result is precisely \eqref{nlineq}.

\subsection{Upper bounds related to the linear operator}
In this section, we prove Lemma \ref{sharpLINEAR} and Lemma \ref{CompactEst}.
The inequality \eqref{normupper} follows by estimating $\ang{\Gamma(M,f),f}$ using \eqref{nlineq} (plus the basic observation that $|f|_{L^2} \lesssim |f|_{N^{s,\gamma}}$).  Regarding the ``compact'' piece of the linear term (i.e., the operator \eqref{compactpiece} and associated inequality \eqref{compactupper}), one would like to establish the inequality
in Lemma \ref{CompactEst}.

To achieve Lemma \ref{CompactEst}, we essentially reprise the arguments of the previous section, in a slightly simpler form.  %As before, we define $M_0 \eqdef P_0 M$ and $M_j \eqdef Q_j M$ for $j \geq 1$.  
Now $M$ coincides with $e^{-v_4/2}$ restricted to the paraboloid $(v,\frac{1}{2} |v|^2)$ as has been already noted.  In particular, then
\[ M^{-1} P_j e_l (v) = \int_{\R^3} dv' \ang{v'} 2^{3j} \tilde{\varphi}(2^j(\ext{v} - \ext{v'})) p(\ext{v'}), \]
where $\tilde \varphi(v) = e^{2^{-j-1} v_4} \varphi(v)$.  
Recall $e_l(v)$ is defined above Lemma \ref{CompactEst}.
Thus, the asymptotic expansion \eqref{asympt} applied to $\tilde \varphi$ yields that 
\[ \tilde{\varphi}(2^j(\ext{v+2^{j} \tau_v u} - \ext{v})) = \varphi (\ext{\tau}_v u) + 2^{-j} \varphi^{(1)}(\ext{\tau}_v u) + E_j(u,v), \]
for some odd function $\varphi^{(1)}$.  Thus, just as was done for \eqref{intest} (namely, expanding $\tilde{\varphi}$, $p$ and $\ang{v'}$ and noting that all of the terms of order $1$ and $2^{-j}$ vanish by symmetry leaving only terms of order $2^{-2j}$ and higher), we may deduce that
\[ |M^{-1} Q_j e_l (v)| \lesssim 2^{-2j} \ang{v}^4, \]
and, more generally, 
\[ |M^{-1} |\nabla_4|^m Q_j e_l (v)| \lesssim 2^{(-2+m)j} \ang{v}^4, \]
for $m = 0,1,2$ (note that the power $4$ arises from the fact that $\ang{\ext{v}} \lesssim \ang{v}^2$ and the polynomial is quadratic). 
 Informally, these inequalities guarantee the intuitively obvious point that the function $e_l$ is, in fact, very smooth as measured by our Littlewood-Paley projections (in fact, the decay rate $2^{(-2+m)j}$ is limited only by the fact that we have not required that the scaling function $\varphi$ have many vanishing moments, in which case the decay would be even better).
Now we expand the sum
\begin{align*}
\ang{\Gamma(g,e_l),f} = \sum_{k=-\infty}^\infty \sum_{j=0}^\infty  \left\{ T^k_+(g,e_{lj},f) - T^k_*(g,e_{lj},f) \right\},
\end{align*}
where we abuse the future notation by defining $e_{l0} \eqdef P_0 e_l$ and $e_{lj} \eqdef Q_j e_l$ for $j > 0$.
To estimate the summand, we use both \eqref{tplusgf} and \eqref{tstargf} on $T^k_+$ and $T^k_{*}$ separately when $k \leq j$, and when $k > j$ we use the cancellation inequality \eqref{compactineq}. 
In both cases we also use the estimate for $M^{-1} Q_j e_{l}$ to obtain an inequality for $\nsm  |\nabla|_4^m e_{lj}\nsm_{L^2_{\gamma+2s}}$; namely, that it is bounded by a uniform constant times $2^{(-2+m)j}$.  Thus our inequality for the sum becomes (after exploiting \eqref{tplusgf}, \eqref{tstargf} and \eqref{compactineq})
\[  \left| \ang{\Gamma(g,e_l),f} \right| \lesssim \nsm g\nsm_{\eta,\delta} \nsm f \nsm_{\delta,\eta} \sum_{k=-\infty}^\infty \sum_{j=0}^\infty \sum_{m=0}^2 \min \{ 2^{2ks} 2^{-2j}, 2^{(2s-m)k} 2^{(-2+m)j} 2^{-\epsilon_m k} \}, \]
for appropriate nonnegative $\epsilon_m$'s which satisfy $2s - m - \epsilon_m < 0$.  Clearly this sum must be finite, which establishes the desired inequality for the compact piece.  Moreover, since each estimate for a trilinear form of $(g,h,f)$ has a corresponding analog with the roles of $h$ and $f$ reversed, an identical argument to the one just given (using the estimates \eqref{tplusgh}, \eqref{tminusgh} and \eqref{othercptineq}) establishes that
\[ \left| \ang{\Gamma(g,f),e_l} \right| \lesssim  \nsm g \nsm_{\eta,\delta} \nsm f \nsm_{\delta,\eta}, \]
holds uniformly for $\eta$ and $\delta$ as well.  In particular, Lemma \ref{CompactEst} follows.

\section{The main coercive inequality}
\label{sec:mainCOER}

This section is devoted to the proof of 
 Lemma \ref{estNORM3}.  Our approach involves  direct pointwise estimates of a Carleman representation in Section \ref{sec:pe}.  However this argument will not be completely sufficient, as explained below.  Thus in 
Section \ref{redistsec} we prove an estimate dubbed ``Fourier redistribution'' to finish the desired bound.  Then in Section \ref{sec:funcN} we will establish functional analytic results on the space $N^{s,\gamma}$.

\subsection{Pointwise estimates}\label{sec:pe}
For any Schwartz function $f$, consider the quadratic difference expression arising in the study of $\ang{N f,f}$ from \eqref{normexpr}.
By virtue of the Carleman-type change of variables, it is possible to express this semi-norm as
\eqref{semiCARLEMAN},
where the kernel can be computed with Proposition \ref{carlemanV2} in the Appendix to be
\begin{equation}
K(v,v')
\eqdef
2\int_{E_{v'}^{v}} d \pi_{v_*'} ~ \frac{M_* M_*'}{|v-v'| |v'-v_*'|} ~ B \left(2v - v' -v_*', \frac{v' - v_*'}{|v' - v_*'|} \right).  \label{kernel}
\end{equation}
Above $M_* = M(v'_*+ v' - v)$
and 
 the integration domain is the hyperplane
  $$
  E^{v}_{v'} \eqdef %\left\{ v'_*\in \mathbb{R}^3 : \ang{v' - v, v_*' -v} =0 \right\}.
  \set{ v'_*\in \mathbb{R}^3}{    \ang{v' - v, v_*' -v} =0 }.
  $$
Then $d\pi_{v'} $ denotes the Lebesgue measure on this hyperplane.

Our goal is to estimate this kernel $K$ pointwise from below and compare it to the corresponding kernel for the norm $\nsm  \cdot \nsm_{N^{s,\gamma}}$ from \eqref{normdef}; this, by virtue of \eqref{lpbynorm},  allows control of our anisotropic Littlewood-Paley square function by 
$
\ang{Nf,f}.
$  
We make this estimate when $|v-v'| \leq 1$ and $| |v|^2 - |v'|^2 | \leq |v-v'|$.  This constraint will require the introduction of a somewhat technical argument, but it is necessary since the required pointwise bound fails to hold uniformly outside this region.

On the hyperplane $E_{v'}^{v}$, we have $\ang{v-v',v-v_*'} = 0$ and $|2v - v' - v_*'| = |v' - v_*'|$; in particular, then
\[ \ang{\frac{2v - v' - v_*'}{|2v - v' - v_*'|}, \frac{v' - v_*'}{|v'-v_*'|}} = \frac{|v - v_*'|^2- |v - v'|^2}{|v' - v|^2 + |v - v_*'|^2}. \]
By virtue of the lower bound for $b(\cos \theta)$ in \eqref{kernelQ}, it follows that
\[ 
B \left(2v - v' -v_*', \frac{v' - v_*'}{|v' - v_*'|} \right) \gtrsim 
\Phi(|v'-v_*' |) ~
\frac{|v'-v_*'|^{2+2s}}{|v-v'|^{2+2s}} {\bf 1}_{|v-v_*'| > |v-v'|}. 
\]
The indicator function must be included because of the support condition in \eqref{kernelQ}.
%provided that $ \epsilon |v'-v_*'| \geq |v-v'|$ for some fixed $\epsilon$ (depending on the size of the support of $b$).  
Thus the kernel $K(v,v')$ from \eqref{kernel} is bounded below by
\begin{equation}
|v-v'|^{-3-2s} \int_{E_{v'}^v} d \pi_{v_*'} ~ M_* M_*' \Phi(|v'-v_*'|)  |v'-v_*'|^{1+2s} {\bf 1}_{|v-v_*'| > |v-v'|}. 
\label{lower1}
\end{equation}
Next we consider the magnitude of the projections of $v_* = v'+v_*'-v$ and $v_*'$ in the direction of $v-v'$.  The orthogonality constraint $\ang{v-v',v-v_*'} = 0$ dictates that
\begin{align*}
\ang{v_*, \frac{v-v'}{|v-v'|}} = \ang{v', \frac{v-v'}{|v-v'|}}  & = \frac{-|v-v'|^2 + |v|^2 - |v'|^2}{2 |v-v'|}  \\
\ang{v_*', \frac{v-v'}{|v-v'|}} = \ang{v , \frac{v-v'}{|v-v'|}}  & = \frac{ |v-v'|^2 + |v|^2 - |v'|^2}{2 |v-v'|}.
\end{align*}
With our assumptions $|v-v'| \leq 1$ and $||v|^2 - |v'|^2| \leq |v-v'|$, both right-hand sides are uniformly bounded by $1$ in magnitude, implying that $|v_*|^2 + |v_*'|^2 \leq 2 |w_*'|^2 + 1$, where 
%$w_*' = v_*' - \frac{v-v'}{|v-v'|}\ang{\frac{v-v'}{|v-v'|},v_*'}$.
$w_*'$ is the orthogonal projection of $v_*'$ onto the hyperplane through the origin with normal $v-v'$, e.g. 
$w_*' = v_*' - \frac{v-v'}{|v-v'|}\ang{\frac{v-v'}{|v-v'|},v_*'}$.

This implies $M_* M_*' \gtrsim e^{-|w_*'|^2/2}$ uniformly.  Let $w'$ and $w$ be the orthogonal projections of $v'$ and $v$  respectively onto this same hyperplane through the origin with normal $v-v'$.  Trivially $|v' -v_*'| \geq | w' - w_*'|$.  Further  $|w' - v'| \leq 1$ since
\[ 
\left|\ang{v',v-v'} \right| = \frac{1}{2}\left||v|^2-|v'|^2- |v-v'|^2 \right| \leq |v-v'|.
\] 
%where $w'$ is some vector satisfying $|w' - v'| \leq 1$.   
%treat $v_*'$ as a function of $w_*'$, i.e., 
Write $v_*' = w_*' + v - w$, then
we may parametrize the integral in \eqref{lower1} as an integral over $w_*'$ (with unit Jacobian) and
thereby bound \eqref{lower1} uniformly from below as
\[ 
K(v, v') \gtrsim
|v-v'|^{-3-2s} \int_{E'} d \pi_{w_*'} e^{- \frac{1}{2} |w_*'|^2}  |w' - w_*'|^{\gamma+2s+1}, 
\]
with
$
E' \eqdef \set{ w_*'}{ \ang{w_*',v-v'}=0, ~ |w' -w_*'| \geq |v-v'|}.  
$
This change of variable preserves these estimates
%(i.e., there is no need to shift $w_*'$ or $w'$) 
because the effect of parametrizing the integral with $w_*'$ is to shift the plane in the direction $v-v'$, while the inequalities for $M_* M_*'$ and $|v'-v_*'|$ are in terms of only the components orthogonal to $v-v'$.

If $|w'| \le 4$, then since $|v-v'| \leq 1$ it is not hard to see that
$$
\int_{E'} d \pi_{w_*'} ~ e^{- \frac{1}{2} |w_*'|^2}  |w' - w_*'|^{\gamma+2s+1} \ge 
\int_{E' \cap \{|w' -w_*'| \ge 1\}} d \pi_{w_*'} ~ e^{- \frac{1}{2} |w_*'|^2} \gtrsim 1.
$$
% With these estimates 
When 
$|w'| \ge 4$, 
we may restrict $w_*'$ to lie in the disk
$
\frac{1}{2} |w'| \ge  |w_*'| +1,
$
which implies in particular $|w' -w_*'| \approx |w'|$.
We thus have the following:
$$
\int_{E'} d \pi_{w_*'} e^{- \frac{1}{2} |w_*'|^2}  |w' - w_*'|^{\gamma+2s+1} \gtrsim
\ang{w'}^{\gamma+2s+1}\int_{0}^{\frac{1}{2} |w'| - 1}  d\rho ~ \rho e^{- \frac{1}{2} \rho^2} \gtrsim \ang{w'}^{\gamma+2s+1}.
$$
Since $|w' - v'| \leq 1$, 
the final, uniform estimate for \eqref{kernel} becomes:
\[ 
K(v, v') \gtrsim |v-v'|^{-3-2s} \ang{v'}^{\gamma+2s+1}
\ind_{|\ext{v} - \ext{v'}| \leq 1}  \ind_{||v|^2 - |v'|^2| \leq |v-v'|}.
\]
On this region $|\ext{v} - \ext{v'}| \lesssim |v-v'|$ and $\ang{v} \approx \ang{v'}$, so with \eqref{normexpr} 
we have uniformly 
\begin{equation} 
| f |_B^2
\gtrsim
\int_{\R^3} dv \int_{\R^3} dv' \frac{(f' - f)^2}{|\ext{v} - \ext{v'}|^{3+2s}} 
(\ang{v} \ang{v'})^{\frac{\gamma+2s+1}{2}} 
\ind_{|\ext{v} - \ext{v'}| \leq 1}  \ind_{||v|^2 - |v'|^2| \leq |v-v'|}. 
\label{coerciveptwise}
\end{equation}
 To obtain a favorable coercivity estimate from this, it would suffice to show that the expression \eqref{coerciveptwise} is bounded from below by the corresponding piece of \eqref{normdef}  (since the former expression has already been shown to be connected to our exotic Littlewood-Paley projections).    Because of the cutoff restricting 
 $
 \ind_{||v|^2 - |v'|^2| \leq |v-v'|}
 $
% $$
% ||v|^2 - |v'|^2| \leq |v-v'|,
% $$ 
 a direct pointwise comparison is not sufficient.  This is not merely a limitation of the argument leading to \eqref{coerciveptwise}; in fact, a more involved analysis of \eqref{normexpr} shows that there is exponential
  decay of $K(v,v')$ in $|v-v'|$ when $v$ and $v'$ point in the same direction.  Thus there is an intrinsic obstruction to obtaining the correct coercive inequality by means of a simple, pointwise comparison of these expressions.

\subsection{Fourier redistribution}
\label{redistsec}
To get around this obstruction, we use the following trick (dubbed here ``Fourier redistribution'').  Essentially the idea is to appeal to the Fourier transform in the situation where the pointwise bound is not available.  The key idea is already contained in the following proposition:

\begin{proposition}
Suppose $K_1$ and $K_2$ are even, nonnegative, measurable functions on $\R^3$ satisfying \label{fourierdist}
\[ \int_{\R^3} du ~ K_l(u) |u|^2 < \infty, \qquad l=1,2. \]
Suppose $\phi$ is any smooth, nonnegative function on $\R^3$ and that there is some constant $C_\phi$ such that 
%$|\nabla_4 |^k \phi(u)$ 
$|\nabla^2 \phi(u) | \le C_\phi $ 
%is bounded by $C_\phi$ 
for all $u$.
%and $k=0,1,2$
For $l=1,2$, consider the following quadratic forms (defined for arbitrary real-valued Schwartz functions $f$):
\[ \nsm f\nsm_{K_l}^2 \eqdef \int_{\R^3} dv \int_{\R^3} dv'~ \phi(v) \phi(v') K_l(v-v') (f(v) - f(v'))^2. \]
If there exists a finite, nonnegative constant $C$ such that, for all $\xi \in \R^3$
\[ \int_{\R^3} du~ K_1(u) |e^{2 \pi i \ang{\xi,u}} - 1|^2 \leq C + \int_{\R^3} du~ K_2(u) |e^{2 \pi i \ang{\xi,u}} - 1|^2, \]
then for all Schwartz functions $f$, 
\[ \nsm f\nsm_{K_1}^2 \leq \nsm f\nsm_{K_2}^2 + C' C_{\phi} \int_{\R^3} dv~ \phi(v) (f(v))^2, 
\]
where the constant $C'$ satisfies $C' \lesssim 1 + C+ \int_{\R^3} du (K_1(u) + K_2(u)) |u|^2$ uniformly in $K_1, K_2$, $\phi$ and $C$.
\end{proposition}

\begin{proof}
We begin with the following identity:
\begin{align*}
 \phi(v) \phi(v') (f(v) - f(v'))^2 = \ & (\phi(v) f(v) - \phi(v') f(v'))^2 \\
 & \ + \phi(v) (f(v))^2 (\phi(v') - \phi(v)) \\
& \ + \phi(v') (f(v'))^2 (\phi(v) - \phi(v')). 
\end{align*}
Multiply both sides by $K_l(v-v')$ and integrate with respect to $v$ and $v'$.  Exploiting symmetry, the result is:
\begin{align*}
 \int_{\R^3} dv \int_{\R^3} dv' & K_l(v-v') (f(v) - f(v'))^2 \phi(v) \phi(v') \\
& = \int_{\R^3} dv \int_{\R^3} dv' K_l(v-v') ( \phi(v) f(v) - \phi(v') f(v'))^2 \\
 & \ \ \ \ \ \  + 2 \int_{\R^3} dv \phi(v) (f(v))^2 \ {p.v.} \! \! \int_{\R^3} dv' K_l(v-v') (\phi(v') - \phi(v)).
\end{align*}
Now Taylor's theorem and the hypotheses on the second derivative of $\phi$ dictate 
\begin{align*}
 | \phi(v') - \phi(v) &  - \ang{v'-v, \nabla \phi(v) } | \leq \frac{1}{2} |v'-v|^2 C_\phi.
\end{align*}
If we define $C(K_l) \eqdef \int du K_l(u) |u|^2$, it follows that the difference
\begin{align*}
 \left| \int_{\R^3} dv \int_{\R^3} dv' \right. & K_l(v-v') (f(v) - f(v'))^2 \phi(v) \phi(v') \\
&  \left. - \int_{\R^3} dv \int_{\R^3} dv' K_l(v-v') ( \phi(v) f(v) - \phi(v') f(v'))^2 \right|, 
\end{align*}
is bounded above by $C(K_l) \int_{\R^3} dv \phi (v) (f(v))^2$.  If we cutoff $|u| > \epsilon$, then clearly the Plancherel formula can be applied to the second term inside the absolute values above, with $F(v) \eqdef \phi(v) f(v)$, to get 
\begin{align}
 \int_{\R^3} dv \int_{\R^3} du &~ K_l(u) ( F(v+u) - F(v))^2 \ind_{|u| > \epsilon} \nonumber \\
& = \int_{\R^3} d \xi \int_{\R^3} du ~ K_l(u) |e^{2 \pi i \ang{\xi,u}}-1|^2 |\widehat{ F}(\xi)|^2 \ind_{|u| > \epsilon}. \label{plancherel}
\end{align}
Clearly the limiting case $\epsilon \rightarrow 0$ will hold as well because $|e^{2 \pi i \ang{\xi,u}}-1|^2$ vanishes to second order in $u$ and $\widehat{F}$ may be assumed to have arbitrarily rapid decay in $|\xi|$.  From here, the remainder is clear.  The hypotheses on $K_1$ and $K_2$ give that the limit of the Plancherel term \eqref{plancherel} is bounded above by the Plancherel term for $K_2$ plus $C$ times the $L^2$-norm of $\hat F$.  This term plus the errors in comparing the Plancherel pieces \eqref{plancherel} to the norms $\nsm \cdot\nsm_{K_l}^2$ give rise to the constant $C'$.
\end{proof}

Next, fix functions $K_1,K_2$ on $\R^4$ given by $K_1(u) \eqdef |u|^{-3-2s} \ind_{|u| \leq 1}$ and $K_2(u) \eqdef |u|^{-3-2s} \ind_{|u| \leq 1} \ind_{|u_4| \leq \epsilon |u|}$, that is, $K_1$ is restricted to the unit ball and $K_2$ is further restricted to 
%the set 
$(1 - \epsilon^2) u_4^2 \leq \epsilon^2 (u_1^2 + u_2^2 + u_3^2)$.  %Using these functions, w
We define the semi-norm $N_0$ by
\begin{align*}
%\nsm f\nsm_{N}^2  & \eqdef \int_{\R^3} dv \int_{\R^3} dv' K_1(\ext{v} - \ext{v'}) (f' - f)^2 \ang{v}^{\frac{\beta+1}{2}} \ang{v'}^{\frac{\beta+1}{2}} \\
\nsm f\nsm_{N_0}^2 & \eqdef \int_{\R^3} dv \int_{\R^3} dv' K_2(\ext{v} - \ext{v'}) (f' - f)^2 (\ang{v'}\ang{v})^{\frac{\gamma+2s+1}{2}}.
\end{align*}
Note that, if $K_2$ is replaced by $K_1$, the resulting expression is the derivative part of our main norm \eqref{normdef}.
By a pointwise comparison of $K_1$ and $K_2$, it is trivially true that $\nsm f\nsm_{N_0} \lesssim \nsm f\nsm_{N^{s,\gamma}}$, but our goal is to prove an inequality in the reverse direction.  To that end, let $\{ \phi \}$ be a smooth partition of unity on $\R^4$ which is locally finite and satisfies uniform bounds for each $\phi$ and their first and second (Euclidean) derivatives.  Suppose furthermore that each $\phi$ is supported on a (Euclidean) ball of radius $\frac{\epsilon}{8}$.  
Recall the notation from Section \ref{sec:acd}.
We restrict these functions to the paraboloid $(v, \frac{1}{2} |v|^2)$ and insert them into the norms $\nsm \cdot\nsm_{N^{s,\gamma}}$ and $\nsm \cdot\nsm_{N_0}$:
\begin{align}
\int_{\R^3} dv \int_{\R^3} dv' K_l(\ext{v} - \ext{v'}) (f' - f)^2 \ang{v}^{\frac{\gamma+2s+1}{2}} \ang{v'}^{\frac{\gamma+2s+1}{2}} \phi(\ext{v}) \phi(\ext{v'}), \label{cutoffterms}
\end{align}
for $l=1,2$.  Suppose that $v_0 \in \R^3$ satisfies $\phi(\ext{v_0}) \neq 0$ for some fixed $\phi$.  Make the change of variables $v \mapsto v_0 + \tau_{v_0} u$ and likewise for $v'$; including the Jacobian factor $\ang{v_0}^{-1}$ for each integral, the result is an integral over $u$ and $u'$ of the integrand
\begin{align*}
 \ang{v_0}^{-2} & K_l(\ext{v_0 + \tau_{v_0} u}   -  \ext{v_0 + \tau_{v_0} u'})   (f(v_0 + \tau_{v_0} u) - f(v_0 + \tau_{v_0} u'))^2 \\
& \times \phi(\ext{v_0 + \tau_{v_0} u}) \phi(\ext{v_0 + \tau_{v_0} u'}) (\ang{v_0 + \tau_{v_0} u} \ang{v_0 + \tau_{v_0} u'})^{\frac{\gamma+2s+1}{2}}. 
\end{align*}
Now we expand.  The argument of $K_l$, for example, becomes
\begin{align*}
 \ext{v_0 + \tau_{v_0} u} - \ext{v_0 + \tau_{v_0} u'} & = \ext{\tau}_{v_0} (u-u') + \frac{1}{2} (|u|^2 - |u'|^2)e_4 \\
& = \ext{\tau}_{v_0} (u-u') + \ang{u-u',\frac{u+u'}{2}} e_4.
\end{align*}
%\begin{align*}
% \ext{v_0 + \tau_{v_0} u} - \ext{v_0 + \tau_{v_0} u'} & = \ext{\tau}_{v_0} (u-u') + \frac{1}{2} (|u - v_0|^2 - |u'-v_0|^2)e_4 \\
%& = \ext{\tau}_{v_0} (u-u') + \ang{u-u',\frac{u+u'}{2} - v_0} e_4.
%\end{align*}
Since the support of $\phi$ is in a ball of radius $\frac{\epsilon}{8}$, it follows that $| \frac{u+u'}{2}  | \leq \frac{\epsilon}{4}$, hence the magnitude of the coefficient of $e_4$ is at most $\frac{\epsilon}{4} |u-u'| = \frac{\epsilon}{4} |\ext{\tau}_{v_0} (u-u')|$, so 
$$
|(\ext{v_0 + \tau_{v_0} u }-\ext{v_0 + \tau_{v_0} u'}) - \ext{\tau}_{v_0} (u-u')| \leq \frac{\epsilon}{2}|\ext{\tau}_{v_0} (u-u')|.
$$  
Thus on this piece of the partition, $K_1$ may be bounded above by
\begin{align*}
 K_1(\ext{v_0+\tau_{v_0}u } & - \ext{v_0 + \tau_{v_0} u'}) \phi(\ext{v_0+\tau_{v_0}u}) \phi(\ext{v_0 + \tau_{v_0} u'}) \\
& \lesssim |u-u'|^{-3-2s} \ind_{|u-u'| \leq 2}   \phi(\ext{v_0 + \tau_{v_0} u})  \phi(\ext{v_0 + \tau_{v_0} u'}),
\end{align*}
which is translation-invariant in $u$ and $u'$.

Next we must make a similar bound for $K_2$ from below. On this piece of the partition, if the fourth coordinate of $\ext{\tau}_{v_0}(u-u')$ is bounded above by $\frac{\epsilon}{2} |\tau_{v_0}(u-u')|$, then the fourth coordinate of $(\ext{v_0 + \tau_{v_0} u }-\ext{v_0 + \tau_{v_0} u'})$ will be at most 
$$
\frac{\epsilon}{2} |\tau_{v_0}(u-u')| + \frac{\epsilon}{4} |\ext{\tau}_{v_0}(u-u')| \leq \frac{3 \epsilon}{4 - 2 \epsilon} |\ext{v_0 + \tau_{v_0} u }-\ext{v_0 + \tau_{v_0} u'}|.
$$
  Thus it also holds that
\begin{align*}
 |u-u'|^{-3-2s} & \ind_{|u-u'| \leq \frac{1}{2}} \ind_{\ang{v_0}^{-1} |\ang{v_0,u-u'}| \leq \frac{\epsilon}{2} |u-u'|} \phi(\ext{v_0 + \tau_{v_0} u})  \phi(\ext{v_0 + \tau_{v_0} u'}) \\
& \lesssim K_2(\ext{v_0+\tau_{v_0}u }  - \ext{v_0 + \tau_{v_0} u'}) \phi(\ext{v_0+\tau_{v_0}u}) \phi(\ext{v_0 + \tau_{v_0} u'}),
\end{align*}
which is also translation-invariant. To apply the proposition, then, it suffices to check the Fourier condition and estimate the derivatives of the cutoff functions.
Clearly zeroth-order through second-order derivatives of 
\[ \ang{v_0}^{-1} \ang{v_0 + \tau_{v_0} u}^{\frac{\gamma+2s+1}{2}} \phi (\ext{v_0 + \tau_{v_0} u}), \]
with respect to $u$ will be uniformly bounded by $\ang{v_0}^{\frac{\gamma+2s-1}{2}}$ by virtue of the corresponding estimates for $\phi$ coupled with the fact that $\tau_{v_0}$ has norm $1$ as a mapping of Euclidean vector spaces and $\ext{\tau}_{v_0}$ is an isometry. 

Modulo the verification of the Fourier condition, then, the norm $\nsm f\nsm_{N^{s,\gamma}}^2$ on any particular piece of the partition will be bounded above by $\nsm f\nsm_{N_0}^2$ on the same piece plus an error term which is given by integration $\int du (f(v_0 + \tau_{v_0} u))^2 \phi$ times a constant linear in $\phi$.  That is, the difference of the quantities given by \eqref{cutoffterms} for $l=1$ and $l=2$  is at most
\[ C' C_{\tilde \phi} \int_{\R^3} du ~ \tilde \phi(v_0 + \tau_{v_0} u) |f(v_0 + \tau_{v_0} u )|^2, \] 
as a result of the previous proposition. Here $\tilde \phi \eqdef \ang{v}^{(\gamma+2s-1)/2} \phi$ (recall that the extra factor of $\ang{v_0}^{-1}$ comes from the change-of-variables we employed).
Thus the quadratic dependence on $\tilde \phi$ gives a factor of $\ang{v_0}$ to the power $\gamma+2s-1$; however an additional factor of $\ang{v_0}$ is obtained when the change-of-variables is reversed (that is, $v_0 + \tau_{v_0} u$ reverts back to $v$).  Thus, summing over the partition will give
\begin{equation*} 
\nsm  f \nsm_{N^{s,\gamma}}^2 \lesssim \nsm f\nsm_{N_0}^2 + \int_{\R^3} dv (f(v))^2 \ang{v}^{\gamma+2s}. 
%\label{coercivestep2}
\end{equation*}
To complete the comparison, then, it suffices to make the following estimate:
\begin{proposition}
Fix any $\epsilon > 0$, and let $E_1$ and $E_2$ be the sets in $\R^3$ given by   $E_1 \eqdef \set{ u \in \R^3}{ |u| \leq 2}$ and $E_2 \eqdef \set{u \in \R^3}{ |u| \leq \frac{1}{2} \mbox{ and } |u_3| \leq \epsilon |u|}$.  Then
\begin{equation}
 \int_{E_1} du ~ |e^{2 \pi i \ang{\xi,u}} - 1|^2 |u|^{-3-2s} \lesssim 1 + \int_{E_2} du ~ |e^{2 \pi i \ang{\xi,u}} - 1|^2 |u|^{-3-2s}, \label{plancherelcheck}
\end{equation}
uniformly for all $\xi \in \R^3$.
\end{proposition}
\begin{proof}
Writing both sides in polar coordinates, we see that each side may be realized as an integral over the unit sphere $\sph$ of
\[ \int_{\tilde{E}_l} d \sigma  ~ \Psi(\ang{\xi,\sigma}), \]
where $\tilde{E}_1 = \sph$, $\tilde{E}_2$ is a small band near the equator, and $\Psi(\lambda)$ is %an integral 
of the form
\[ 
\Psi(\lambda) \eqdef \int_0^a dt ~ |e^{2 \pi i \lambda t} - 1|^2 t^{-1-2s}, 
\]
for some appropriate value of $a$ (%in this case, 
$a=2$ or $a = \frac{1}{2}$).
From the elementary inequalities
\begin{align*}
 \int_0^{(2 \lambda)^{-1}} dt |e^{2 \pi i \lambda t} - 1|^2 t^{-1-2s}  & \approx \int_0^{(2 \lambda)^{-1}} dt \lambda^2 t^2 t^{-1-2s} \approx \lambda^{2s} \\
\int_{(2 \lambda)^{-1}}^\infty dt |e^{2 \pi i \lambda t} - 1|^2 t^{-1-2s} & \lesssim \int_{(2 \lambda)^{-1}}^\infty dt \ t^{-1-2s} \approx \lambda^{2s} ,
\end{align*}
it follows that the integrands will be comparable to $|\ang{\xi,\sigma}|^{2s}$ when this quantity is bounded below by a fixed constant and less than a constant times $|\ang{\xi,\sigma}|^{2s}$ regardless of whether or not this quantity is bounded below.  For any $\xi$ with $|\xi| \geq 1$, then, at least a positive measure region of $\sph$ will have $|\ang{\xi,\sigma}| \gtrsim |\xi|$ (whether in $\tilde{E}_1$ or $\tilde{E_2}$), so both sides of \eqref{plancherelcheck} will be comparable to $|\xi|^{2s}$, which is sufficient for the inequality \eqref{plancherelcheck} to hold.
\end{proof}

The proof of the coercive inequality is now complete, for we demonstrated that
\[ 
\int_{\R^3} dv \int_{\R^3} dv_* \int_{\sph} d \sigma ~ B (f'-f)^2 M_*'M_* \gtrsim \nsm f\nsm_{N_0}^2, 
\]
by direct pointwise comparison and that $\nsm f\nsm_{N_0}^2 + \nsm f\nsm_{L^2_{\gamma+2s}}^2 \gtrsim \nsm f\nsm_{N^{s,\gamma}}$ by Fourier redistribution. The combination of these inequalities gives Lemma \ref{estNORM3}.
%\[ \int_{\R^3} dv \int_{\R^3} dv_* \int_{\sph} d \sigma B (f'-f)^2 M_*'M_* + \nsm f\nsm_{L^2_{\gamma+2s}}^2 \gtrsim \nsm f\nsm_{N^{s,\gamma}}^2, \]
%which is the assertion made by .

\subsection{Regarding the functional analysis of $N^{s,\gamma}$}
\label{sec:funcN}
  An important consequence of the analysis of the previous section is that we have an alternate characterization of the space $N^{s,\gamma}$ in terms of the usual Sobolev spaces.  In particular, let $\{ \phi_i \}$ be a partition of unity constructed as above by restricting a smooth, locally finite partition of unity on $\R^4$ (such that each $\phi_i$ has support in a ball of unit radius) to the paraboloid $(v, \frac{1}{2} |v|^2)$.  For each $\phi_i$ in the partition, let $v_i$ be some point in its support.  If we define
\[ f_i(u) \eqdef \phi_i(\ext{v_i + \tau_{v_i} u}) f(v_i + \tau_{v_i} u), \]
it follows that we have the comparison
\begin{equation} 
|f|_{N^{s,\gamma}}^2 \approx \sum_{i=1}^\infty \ang{v_i}^{\gamma+2s-1} |f_i|_{H^s}^2, \label{isosobolev}
\end{equation}
where $H^s$ is the usual (three-dimensional) $L^2(\mathbb{R}^3)$-Sobolev space.  This result is true by virtue of the fact that
\[ 
|f|_{H^s}^2 \approx |f|_{L^2}^2 + \int_{\R^3} dv \int_{\R^3} dv' \frac{(f(v')- f(v))^2}{|v-v'|^{3+2s}} {\mathbf 1}_{|v-v'| \leq 1}, 
\]
which follows itself by an application of the Plancherel theorem as in Proposition \ref{fourierdist} together with the asymptotic estimates for the integrals \eqref{plancherelcheck}.

With the aid of \eqref{isosobolev}, a number of elementary functional analysis properties of $N^{s,\gamma}$ reduce to the situation of the standard Sobolev spaces.  For example, it is a simple exercise to show that Schwartz functions are dense in $N^{s,\gamma}$ by exploiting this same fact for the space $H^s$, approximating $f_i$ individually in $H^s$, and summing over the partition (note that this requires the elements of the partition $\phi_i$ themselves to be Schwartz functions, but this additional restriction is not a problem to satisfy).

A somewhat more sophisticated result which may be obtained by similar reasoning is the fact that $N^{s,\gamma}$ embeds compactly in $L^2$.  Clearly one also has the comparability statement that
\[ \sum_i \ang{v_i}^{-1} |f_i|_{L^2}^2 \approx |f|_{L^2}^2, \]
since
\[ 
\sum_i \ang{v_i}^{-1} |f_i|_{L^2}^2 = \int_{\R^3} dv \left(\sum_i |\phi_i(v)|^2 \right) |f(v)|^2, 
\]
and the sum of the squares of the $\phi_i$'s must be uniformly bounded above and below.  Now suppose $f^n$ is any sequence of functions in $N^{s,\gamma}$ with norms uniformly bounded by $1$, having the weak limit $f^0 \in N^{s,\gamma}$. Then the compact embedding of $H^s$ into $L^2$ on compact domains ensures that $f^n_i$ converges in $L^2$ for each $i$; the limit of $f^n_i$ must equal $f^0_i$.  Now by \eqref{isosobolev} and the corresponding estimate for the $L^2$-norm, we have for any $R > 1$ that
\begin{gather*}
 | f^n - f^0|_{L^2}  \approx \sum_{i=1}^\infty \ang{v_i}^{-1} |f^n_i - f^0_i|^2_{L^2} 
\\  
\lesssim 
\sum_{i\ :\ |v_i| \leq R}  \ang{v_i}^{-1} |f^n_i - f^0_i|^2_{L^2} + R^{-\gamma-2s} \sum_{i \ :\ |v_i| \geq R} \ang{v_i}^{\gamma+2s-1} |f^n_i - f^0_i|^2_{H^s}
\\
\lesssim 
\sum_{i\ :\ |v_i| \leq R}  \ang{v_i}^{-1} |f^n_i - f^0_i|^2_{L^2} + 
R^{-\gamma-2s} 
|f^n - f^0|_{N^{s,\gamma}}.
\end{gather*}
To show that $f^n$ converges to $f^0$ in $L^2$, simply observe that the {\it finite} sum over $i$ in the last line above goes to zero by compactness for any fixed $R$, while $|f^n - f^0|_{N^{s,\gamma}} \lesssim 1$ uniformly for all $n$; thus taking $R \rightarrow \infty$ establishes the claim.

\section{De-coupled space-time estimates and global existence} 
\label{sec:deBEest}

In this last section, we show that the sharp estimates proven in the previous sections can be applied to the modern technology from the linearized cut-off Boltzmann theory to establish global existence.   This works precisely because of the specific structure of the interactions between the velocity variables and the space-time variables.  The methodology that we employ essentially de-couples the  required space-time estimates that are needed from the new fractional and anisotropic derivative estimates which are shown in the previous sections.

The method that we choose to utilize in this section goes back to Guo \cite{MR2000470}.  A key point of this approach is to derive a system of space-time ``macroscopic equations,'' see \eqref{c} - \eqref{adot} below, which have certain elliptic structure and also some hyperbolic structures. These structures can be used to prove an instantaneous coercive lower bound for the linear operator $L$, for solutions to the full non-linear equation \eqref{Boltz}, in our new precise norm \eqref{normdef}.
This original method \cite{MR2000470} made use of high order temporal derivatives, which we could also utilize.  
But as a result of advances in  \cite{MR2095473}, \cite{Jang2009VMB,MR2420519} the need for temporal derivatives was removed from the method.  The key point here is to use both the macroscopic equations \eqref{c} - \eqref{adot} and the conservation  laws \eqref{cl.0} - \eqref{cl.2} to remove the need to estimate time derivatives with an ``interaction functional'' that is comparable to the energy.  We point the readers attention to the general abstract framework of  \cite{villani-2006}  also
in this direction.

In what follows, we will show that the use of our new precise weighted geometric fractional derivative norm \eqref{normdef} and our crucial sharp estimates stated in
Lemma \ref{NonLinEst},
Lemma \ref{estNORM3},
and
Lemma \ref{CompactEst} can be combined with the above general de-coupled energy method in order to prove global existence and decay as in our main Theorem \ref{mainGLOBAL}.

We will now discuss the coercivity of the linearized collision operator, $L$, away from its null space.  
More generally, from the H-theorem  $L$ is non-negative and for every fixed $(t,x)$ the null
space of $L$ is given by the five dimensional space 
\begin{equation}
{\mathcal N}\eqdef {\rm span}\left\{
\sqrt{\mu},
v_1 \sqrt{\mu}, v_2 \sqrt{\mu}, v_3 \sqrt{\mu}, 
|v|^2 \sqrt{\mu}\right\}.
 \label{null}
\end{equation}
We define the orthogonal projection from $L^2(\mathbb{R}^3_v)$ onto the null space ${\mathcal N}$ by ${\bf P}$. 
Further expand ${\bf P} h$ as a linear combination of the basis in (\ref{null}): 
\begin{equation}
{\bf P} h
\eqdef
 \left\{a^h(t,x)+\sum_{j=1}^3 b_j^h(t,x)v_j+c^h(t,x)|v|^2\right\}\sqrt{\mu}.
\label{hydro}
\end{equation}
%
%This together with the
%form \eqref{hydro} of ${ \bf  P  } $ imply
%\begin{eqnarray*}
%% \nonumber to remove numbering (before each equation)
%  a^f(t,x) &=& 
%  \frac{1}{2}\int_{\mathbb{R}^3} ~ dv ~  (5-|v|^2)\sqrt{\mu} f,\label{def.a}\\
%  b^f(t,x) &=& 
%  \int_{\mathbb{R}^3}~ dv ~ v\sqrt{\mu} f,\label{def.b}\\
%  c^f(t,x) &=& 
%  \frac{1}{6}\int_{\mathbb{R}^3}~dv~ (|v|^2-3)\sqrt{\mu} f,
%  \label{def.c}
%\end{eqnarray*}
%where $b^f=(b^f_1,b^f_2,b^f_3)$.
We can then decompose  $f(t,x,v)$ as
\[
f={\bf P}f+\{{\bf I-P}\}f.
\]
We characterize in 
Lemma \ref{lowerN} the functional properties of the linearized collision operator that are useful for our main results.  

\begin{lemma}
\label{lowerN}  $L\ge 0$.
$Lh = 0$ if and only if $h = {\bf P} h$.  
And
$\exists \delta_0>0$ such that
\begin{equation*}
\langle L h, h \rangle
\ge
\delta_0 
| \{ {\bf I - P } \} h |_{N^{s,\gamma}}^2.
%\label{lowerL}
\end{equation*}
\end{lemma}

The proof will follow directly from Lemma \ref{CompactEst} and the splitting $L = N + K$. 
The coercive estimate in Lemma \ref{lowerN} is proven via a contradiction argument.  This is the only location in our paper where a non-constructive argument is used. We expect that our estimates can aid in a future constructive--but perhaps substantially longer--proof of this coercivity; see \cite{MR2254617,MR2322149} and the references therein.

Throughout the remainder of the paper, we will abuse notation and write $|h|_{N^{s,\gamma}}^2$ from \eqref{normdef} in place of $\ang{N h,h}$ from \eqref{normpiece}; since these quantities have already been shown to be comparable, all of the previous inequalities involving $|\cdot|_{N^{s,\gamma}}$ will remain true after this redefinition at the price of a fixed multiplicative constant. 
 \\

\noindent {\it Proof of Lemma \ref{lowerN}.}    Most of this lemma is standard, see e.g. \cite{MR1379589} for proofs of the first statements in the cut-off case.
Without cut-off  the first statements can be established with the same proofs as in the cut-off situation.  This is done via  the usual approximations of the singular kernel \eqref{kernelQ} with a non-singular kernel $b_\epsilon(\cos\theta)$ and sending $\epsilon \downarrow 0$.  
  We only prove the coercive lower bound for the linear operator.  

 Assuming that coercivity fails grants a sequence of functions $h^n$ which satisfy 
$
{\bf P} h^n = 0,
$ 
$\nsm h^n \nsm_{N^{s,\gamma}}^2 = \langle h^n, h^n \rangle_{N^{s,\gamma}}  =  \langle N h^n, h^n \rangle  = 1$ and
$$
\langle L h^n, h^n \rangle
=
\nsm  h^n \nsm_{N^{s,\gamma}}^2 
- 
\langle {K} h^n, h^n \rangle
\le \frac{1}{n}.
$$ 
Thus $ \{ h^n \} $ is weakly compact in ${N^{s,\gamma}}$ with limit point $h^0$.  
By weak lower-semi continuity $\nsm h^0 \nsm_{N^{s,\gamma}}^2  \le 1$.  Furthermore,
$$
\langle L h^n, h^n \rangle
=
1
- 
\langle {K} h^n, h^n \rangle.
$$ 
We {\it claim} that 
$$
\lim_{n\to\infty}
\langle {K} h^n, h^n \rangle
=
\langle {K} h^0, h^0 \rangle.
$$
The claim will follow from the prior Lemma \ref{CompactEst}.   The claim implies
$$
0 = 1 - \langle {K} h^0, h^0 \rangle.
$$
Or equivalently
$$
\langle L h^0, h^0 \rangle
=
\nsm h^0 \nsm_{N^{s,\gamma}}^2 
- 
1.
$$
Since $L\ge 0$, we have $\nsm h^0 \nsm_{N^{s,\gamma}}^2
=
1$ which implies $h^0 = {\bf P} h^0$.  On the other hand since
 $h^n = \{{\bf I- P}\} h^n$ the weak convergence implies $h^0 = \{{\bf I- P}\} h^0$.
This  is a contradiction to $\nsm h^0 \nsm_{N^{s,\gamma}}^2
=
1$. 

It remains to establish the claim.  We expand out
\begin{align*}
\langle {K} h^n, h^n \rangle
-
\langle {K} h^0, h^0 \rangle
= &  
\langle {K} h^0, (h^n - h^0) \rangle \\ & + \ang{K (h^n - h^0),h^0} \\ & +
\langle {K} (h^n - h^0), (h^n-h^0) \rangle.
\end{align*}
We recall the definition \eqref{compactpiece} and the general estimate in Lemma \ref{CompactEst}; setting $\eta = \delta = 1$, it is easy to see that
\begin{gather*}
\left| \langle {K} g, h \rangle \right| 
 \lesssim 
\nsm g \nsm_{1,1} 
\nsm h \nsm_{1,1}.
\end{gather*}
Since $|f|_{\eta,\delta} \lesssim |f|_{N^{s,\gamma}}$ for any fixed $\eta,\delta$, it follows that $K$ is bounded on $N^{s,\gamma}$.  In particular then
\[ \ang{K h^0, (h^n - h^0)} \rightarrow 0, \]
as $n \rightarrow \infty$ by weak convergence. 
Since $\langle {K} g, h \rangle = \langle  g,  {K}h \rangle$, the same is true of the second term.
% Likewise, the linear functional
%\[ f \mapsto \ang{K f, h^0} \]
%is bounded on $N^{s,\gamma}$, so $\ang{K(h^n - h^0),h^0} \rightarrow 0$ as well by weak convergence. 
Finally, by \eqref{compactupper}, we have uniformly for all $\eta > 0$ that 
\begin{gather*}
\left| \langle {K} (h^n - h^0), (h^n-h^0) \rangle \right| 
 \lesssim 
\eta |h^n - h^0|_{L^2_{\gamma+2s}(\mathbb{R}^3_v)}^2 + C(\eta) |h^n -h^0|_{L^2_v}^2.
\end{gather*}
Since $|h^n - h^0|_{L^2_{\gamma+2s}} \le |h^n - h^0|_{N^{s,\gamma}} \leq 2$ for all $n$ and $N^{s,\gamma}$ embeds compactly in $L^2_v$ %(|v| \le C(\eta))$ 
as in Section \ref{sec:funcN}, it follows that for any $\eta > 0$
\[ \limsup_{n \rightarrow \infty} \left| \langle {K} (h^n - h^0), (h^n-h^0) \rangle \right|  \lesssim \eta. \]
 In particular, this implies  the limit is zero and establishes the claim.
\qed \\

\subsection{Local Existence}  Local existence has been shown in the recent preprint \cite{arXiv:0909.1229v1}, with higher regularity assumptions on the initial data than we consider.  We establish the a priori bounds for local existence herein using our estimates above with initial data in the space $L^2_v H^N_x$.
Our local existence proof for \eqref{Boltz} is based on a uniform energy estimate for
an iterated sequence of approximate solutions.

Our iteration starts at $f^0(t,x,v)\eqdef f_0(x,v)$.  We   solve for $f^{n+1}(t,x,v)$ such that 
\begin{equation}
\left( \partial _t+v\cdot \nabla _x+N \right) f^{n+1}+Kf^n
=
\Gamma (f_{,}^nf^{n+1}),\;\;\;f^{n+1}(0,x,v)=f_0(x,v).  \label{approximate}
\end{equation}
For notational convenience during  the proof we 
define the ``dissipation rate'' as
\begin{gather}
\notag
\mathcal{D}(f(t))\eqdef \sum_{|\alpha |\le N}
\|\partial^\alpha f(t)\|_{N^{s,\gamma}}^2.
\end{gather}
We
will also use the following total norm
\begin{gather}
\mathcal{G}(f(t))
\eqdef
\|f(t)\|^2_{L^2(\mathbb{R}^3_v; H^N(\mathbb{T}^3_x))}+ \int_0^t d\tau ~\mathcal{D}(f(\tau)).
\label{totalG}
\end{gather} 
Our goal will be to obtain a uniform estimate for the iteration on 
a small time interval. The crucial energy estimate is as follows:

\begin{lemma}
\label{uniform}  The sequence \{$f^n(t,x,v)\}$ is well-defined. There exists a short time
$
T^{*} = T^{*}(\|f_0\|^2_{L^2_v H^N_x})>0,
$
such that for $\|f_0\|^2_{L^2_v H^N_x}$
sufficiently small, there is a uniform constant $C_0>0$ such that
\begin{equation}
\sup_{n\ge 0} ~
\sup_{0\le \tau \le T^{*}} ~ 
\mathcal{G}(f^{n}(\tau))
\le
2C_0
\|f_0\|^2_{L^2_v H^N_x}.  \label{uni}
\end{equation}
\end{lemma}

\begin{proof}
The proof proceeds with an induction over $k$. Clearly $k=0$ is true and we
assume that (\ref{uni}) is valid for $k=n$.  
For a given $f^{n}$, it is standard to show that there exists a solution
 $f^{n+1}$ to the linear equation \eqref{approximate}.
We focus here on the proof of (\ref{uni}).
Take the spatial derivatives $\partial ^\alpha $ of (\ref{approximate}) to obtain
\begin{equation}
\left( \partial _t+v\cdot \nabla _x\right)\partial ^\alpha f^{n+1}
+N\left(\partial^\alpha f^{n+1}\right)
+K\left(\partial ^\alpha f^n\right)
=
\partial ^\alpha \Gamma \left(f^n, f^{n+1}\right).  \label{xderi}
\end{equation}
Therefore, applying the non-linear estimate in Lemma \ref{NonLinEst}  yields 
\begin{gather*}
\frac 12
\frac d{dt}\|\partial^\alpha f^{n+1}\|_{L^2_vL^2_x}^2
+\| \partial^\alpha f^{n+1} \|_{L^2_x N^{s,\gamma} }^2
+(K\left(\partial ^\alpha f^n\right),\partial ^\alpha
f^{n+1})
\\
=(\partial ^\alpha \Gamma \left(f_{,}^nf^{n+1}\right),\partial ^\alpha f^{n+1})
\\
\lesssim
\|f^n\|_{H^N_{x} L^2_v  }
 \|f^{n+1}\|_{H^N_{x} N^{s,\gamma} }^2
 +
\|f^n\|_{H^N_{x} N^{s,\gamma}   }
 \|f^{n+1}\|_{H^N_{x} N^{s,\gamma} }
 \|f^{n+1}\|_{H^N_{x} L^2_v  }.
\end{gather*}
Then integrating the above over $[0,t]$ we obtain
\begin{gather}
\frac 12\|\partial ^\alpha f^{n+1}(t)\|_{L^2_vL^2_x}^2
+
\int_0^t~d\tau~ \| \partial^\alpha f^{n+1}(\tau) \|_{L^2_x N^{s,\gamma} }^2
+
\int_0^t ~d\tau~ (K\left(\partial ^\alpha f^n \right) ,\partial ^\alpha f^{n+1})  
\notag
\\
\label{xterm} 
\le \frac 12 \|\partial ^\alpha f_0\|_{L^2_vL^2_x}^2
+
C \int_0^t ~d\tau~
\|f^n\|_{H^N_{x} L^2_v  }
 \|f^{n+1}\|_{H^N_{x} N^{s,\gamma} }^2(\tau)
 \\
 +
 C \int_0^t ~d\tau~
\|f^n\|_{H^N_{x} N^{s,\gamma}   }
 \|f^{n+1}\|_{H^N_{x} N^{s,\gamma} }
 \|f^{n+1}\|_{H^N_{x} L^2_v  }(\tau). 
\nonumber
\end{gather}
We notice that from Lemma \ref{CompactEst} applied to \eqref{compactpiece}, for any $\eta
>0 $ small and $\delta = 1/2$, 
\begin{gather*}
\left| \int_0^t d\tau (K(\partial ^\alpha f^n),\partial ^\alpha f^{n+1}) 
\right|
\le \int_0^t d\tau
\left( \frac 12 \|\partial ^\alpha f^{n+1}(\tau)\|_{L^2_{\gamma + 2s}}^2
+
C_\eta\|\partial ^\alpha f^{n+1}(\tau)\|_{L^2_{x,v}}^2 \right)
\\
+\eta \int_0^t ~ d\tau ~\|\partial ^\alpha f^n (\tau)\|_{L^2_{\gamma + 2s}}^2
+C \int_0^t ~ d\tau ~ \| \partial ^\alpha f^n(\tau)\|_{L^2_{x,v}}^2.
\end{gather*}
We incorporate this inequality into \eqref{xterm} and sum over $|\alpha | \le N$ to obtain 
\begin{gather}
\mathcal{ G}(f^{n+1}(t)) 
\le
C_0\| f_0 \|_{L^2_v H^N_x}^2
+
\int_0^t ~ d\tau ~ \left\{ C\| f^{n+1}\|^2_{H^N_{x} L^2_v }(\tau)
+
C\eta \| f^n\|_{H^N_{x}L^2_{v,\gamma + 2s}}^2(\tau) \right\}
\nonumber 
\\
+
 C_\eta \int_0^t ~ d\tau ~ \| f^n\|^2_{H^N_{x} L^2_v  }(\tau)
+
C\sup_{0\le \tau\le t}\mathcal{ G}(f^{n+1}(\tau))\sup_{0\le \tau\le t}\mathcal{ G}^{1/2}(f^n(\tau))  
\notag
\\
\le 
C_0\| f_0 \|_{L^2_v H^N_x}^2+C_\eta  t\left\{ \sup_{0\le \tau\le t}\mathcal{ G}(f^{n+1}(\tau))
+\sup_{0\le \tau\le t}\mathcal{ G}(f^n(\tau))\right\} 
\nonumber 
\\
+C\eta \sup_{0\le \tau\le t}\mathcal{ G}(f^n(\tau))  
+C\sup_{0\le \tau\le t}\mathcal{ G}(f^{n+1}(\tau))\sup_{0\le \tau\le t}\mathcal{ G}^{1/2}(f^n(\tau)).  
\nonumber
\end{gather}
We are using the total norm from \eqref{totalG}.
By the induction hypothesis \eqref{uni}
$$
\sup_{0\le \tau \le t}\mathcal{ G}(f^n(\tau))\le 2C_0\| f_0 \|_{L^2_v H^N_x}^2.
$$ 
Then we collect terms in the previous inequality to obtain
\begin{gather*}
\left\{ 1-C_\eta T^{*}-C\| f_0 \|_{L^2_v H^N_x}\right\}\sup_{0\le t\le T^{*}}\mathcal{ G}%
(f^{n+1}(t))
\\
\le
 \left\{C_0+2C\eta +2C_\eta T^{*}C_0\right\}\| f_0 \|_{L^2_v H^N_x}^2. 
\end{gather*}
By choosing $\eta $ small, then choosing $T^{*}=T^{*}(\| f_0 \|_{L^2_v H^N_x})$ small, we have 
\[
\sup_{0\le t\le T^{*}}\mathcal{ G}(f^{n+1}(t))\le 2C_0\| f_0 \|_{L^2_v H^N_x}^2.
\]
We therefore conclude Lemma \ref{uniform} if $T^{*}$ and $\| f_0 \|_{L^2_v H^N_x}^2$ are sufficiently small. 
\end{proof}

With our uniform control over the iteration from \eqref{approximate} proved in Lemma \ref{uniform}, we can now prove local existence in the following Theorem.

\begin{theorem}(Local Existence)
\label{local}
For any sufficiently small $M_0>0,$ there exists a time  $T^{*} = T^{*}(M_0)>0$ and $%
M_1>0,$ such that if 
\[
\| f_0 \|_{L^2_v H^N_x}^2\le M_1, 
\]
then there is a unique  solution $f(t,x,v)$ to \eqref{Boltz} on  
$[0,T^{*})\times \mathbb{T}_x^3\times \mathbb{ R}_v^3$ such that 
\[
\sup_{0\le t\le T^{*}}\mathcal{ G}(f(t))\le M_0. 
\]
Furthermore $\mathcal{ G}(f(t))$ is continuous over $[0,T^{*}).$ Lastly, we have positivity in the sense that
 if
 $F_0(x,v)=\mu +\mu
^{1/2}f_0\ge 0,$ then 
$
F(t,x,v)=\mu +\mu ^{1/2}f(t,x,v)\ge 0. 
$
\end{theorem}

\begin{proof}By taking $n\rightarrow \infty ,$ we have shown sufficient compactness from Lemma \ref{uniform} to obtain a strong solution $f(t,x,v)$ to the Boltzmann equation \eqref{Boltz} locally in time. 

To prove the uniqueness, we suppose that there exists another solution $g$ with the same initial data satisfying 
$\sup_{0\le \tau \le T^{*}}\mathcal{ G}(g(\tau))\le M_0.$ The difference $f-g$
satisfies 
\begin{equation}
\{\partial _t+v\cdot \nabla _x\}\left(f-g\right)+L\left(f-g\right)=\Gamma \left(f-g,f\right)+\Gamma \left(g,f-g\right).
\label{difference}
\end{equation}
We apply Lemma \ref{NonLinEst} and the Sobolev embedding  $H^2(\mathbb{T}^3_x) \subset L^\infty(\mathbb{T}^3_x) $ to obtain
\begin{gather*}
\left| \left( \left\{\Gamma\left(f-g,f\right)+\Gamma\left(g,f-g\right)\right\},f-g\right) \right|
\lesssim
\left\{\| g\|_{L^2_v H^2_x}+
\|f\|_{L^2_v H^2_x}\right\}
\|f-g\|_{N^{s,\gamma}} ^2
\\
+
\left\{\| g\|_{H^2_x  N^{s,\gamma}}+
\|f\|_{H^2_x  N^{s,\gamma}}\right\}
\|f-g\|_{N^{s,\gamma}} 
\|f-g\|_{L^2_{v,x}}. 
\end{gather*}
The Cauchy-Schwartz inequality (applied in the time variable) shows us that 
\begin{gather*}
\int_0^t d\tau 
\left\{\| g\|_{H^2_x  N^{s,\gamma}}+
\|f\|_{H^2_x  N^{s,\gamma}}\right\}
\|f-g\|_{N^{s,\gamma}} 
\|f-g\|_{L^2_{v,x}}(\tau)
\\
\le
\sqrt{M_0}
\left(
\sup_{0\le \tau \le t}
\|f(\tau)-g(\tau)\|_{L^2_{v,x}}^2 
\int_0^t d\tau 
~ \|f(\tau)-g(\tau)\|_{N^{s,\gamma}} ^2
\right)^{1/2}.
\end{gather*}
We have just used the following fact, 
which follows from the local existence, that
\[
\sup_{0\le \tau \le t}
\|f(\tau)\|_{L^2_v H^2_x}
+
\int_0^t ~ d\tau ~ \| f(\tau)\|_{H^2_x N^{s,\gamma}}^2
\le M_0.
\]
And similarly for $g(t)$.
Since $L=N+K$,  we use Lemma \ref{CompactEst} applied to \eqref{compactpiece}
to obtain
\[
(L(f-g),f-g)\ge \frac 12\|f-g\|_{N^{s,\gamma}} ^2-C\|f-g\|_{L^2_{v,x}}^2.
\]
We  multiply \eqref{difference} with $f-g$ and integrate over 
$[0,t]\times \mathbb{T}_x^3\times \mathbb{ R}_v^3$ to achieve 
\begin{gather*}
\frac{1}{2}\|f(t)-g(t)\|_{L^2_{v,x}} ^2+\frac{1}{2}\int_0^t ~ d\tau~ \|f(\tau)-g(\tau)\|_{N^{s,\gamma}} ^2
\\
\lesssim 
\sqrt{M_0}
\left(
\sup_{0\le \tau \le t}
\|f(\tau)-g(\tau)\|_{L^2_{v,x}} ^2
+
\int_0^t ~ d\tau~ \|f(\tau)-g(\tau)\|_{N^{s,\gamma}} ^2
\right)
\\
+
\int_0^t~ d\tau~ \|f(\tau)-g(\tau)\|_{L^2_{v,x}}^2.
\end{gather*}
We deduce $f \equiv g$ and the uniqueness from the Gronwall inequality.

To show continuity of $\mathcal{ G}(f(t))$ in time, we  sum 
\eqref{xterm} over $|\alpha| \le N$ and integrate from $t_2$ to $t_1$ (rather than over $[0,t]$).  Then with $f^n=f^{n+1}=f$ we obtain 
\begin{gather*}
\left| \mathcal{ G}(f(t_1))-\mathcal{ G}(f(t_2)) \right|
=
\left| \frac 12\| f(t_1)\|_{L^2_v H^N_x}^2
-
\frac 12\| f(t_2)\|_{L^2_v H^N_x}^2
+
\int_{t_2}^{t_1} ~ d\tau ~\mathcal{D}(f(\tau)) \right| 
\\
\lesssim \left\{1+\sup_{{t_2}\le \tau \le {t_1}}\sqrt{\mathcal{ G}(f(\tau ))}
\right\}
\int_{t_2}^{t_1}
~ d\tau ~ 
\| f(\tau )\|_{H^N_x N^{s,\gamma}}^2
%\left\{
%\| f(\tau )\|_{H^N_x N^{s,\gamma}}^2+
%\| f(\tau )\|_{L^2_v H^N_x}^2\right\} 
\rightarrow 0,
\end{gather*}
as ${t_1}\rightarrow {t_2}$ since 
%$
%\| f(\tau )\|_{L^2_v H^N_x}^2
%$ 
%and 
$
\| f(\tau )\|_{H^N_x N^{s,\gamma}}^2
$ 
is
integrable in time.

We now explain the proof of positivity. Previous works which obtain the positivity of solutions without cut-off, to our knowledge, are only \cite{MR839310} and \cite{arXiv:0909.1229v1}.  We  use the argument from \cite{arXiv:0909.1229v1}, however their initial data is somewhat smoother than ours, e.g. they effectively work in $f_0\in H^{M}_{x,v}$ for $M\ge 5$, since $F_0 = \mu + \sqrt{\mu} f_0$, and they study moderate angular singularities $0<s<1/2$.  If our initial data is in $H^{M}_{x,v}$, then 
since we have proven the uniqueness, we conclude that
$F = \mu + \sqrt{\mu} f \ge 0$ if initially
$F_0 = \mu + \sqrt{\mu} f_0 \ge 0$.
The argument is finished by using the density of $H^{M}_{x,v}$ in the larger space $L^2_v H^N_x(\mathbb{T}^3_x\times \mathbb{R}^3_v)$, standard approximation arguments, and our uniqueness theorem.  For the high singularities, 
$1/2\le s<1$, the positivity can be established by using  high derivative estimates  $f(t) \in H^{M}_{x,v}$ from  \cite{sgNonCut2},  and following the same procedure \cite{arXiv:0909.1229v1} as in the low singularity case. 
\end{proof}

\subsection{Coercivity estimates for solutions to the Non-Linear equation}
The following is a by now well known statement of the Linearized H-Theorem \cite{MR2000470}; we prove it for the first time in the regime where there is no angular cut-off, e.g. \eqref{kernelQ}.

\begin{theorem}
\label{positive}  
Given the initial data $f_0 \in  L^2\left(\mathbb{R}^3_v : H^N\left(\mathbb{T}^3_x \right)\right)$ for some $\NgE$,
which satisfies \eqref{conservation} initially and the assumptions of Theorem \ref{local}. 
Consider the corresponding solution, $f(t,x,v)$, to \eqref{Boltz} which continues to satisfy \eqref{conservation}.

There exists a small constant $M_0 >0$ such that if
\begin{gather}
\| f(t) \|^2_{L^2_vH^N_x} \le M_0,
\label{smallL}
\end{gather}
then, further, there are universal constants $\delta>0$ and $C_2>0$ such that
$$
\sum_{|\alpha| \le N} \| \{ {\bf I - P } \} \partial^\alpha f \|_{_{N^{s,\gamma}}}^2(t)
\ge 
\delta
\sum_{|\alpha| \le N} \| { \bf  P  } \partial^\alpha f \|_{_{N^{s,\gamma}}}^2(t) - C_2\frac{d\mathcal{I}(t)}{dt},
$$
where $\mathcal{I}(t)$ is the ``interaction functional'' defined precisely in \eqref{mainINTERACTION} below.
\end{theorem}

We prove this theorem by an analysis of the  macroscopic equations and also the local conservation laws.  The system of macroscopic equations comes from first  expressing
the hydrodynamic part ${\bf P}f$ through the microscopic part $\{{\bf I-P}\}f,$ up to the
higher order term $\Gamma (f,f)$ as
\begin{equation}
\{\partial _t+v\cdot \nabla _x\}{\bf P}f=
-\partial _t \{{\bf I-P}\}f
+l(\{{\bf I-P}\}f)+\Gamma (f,f),  \label{macro}
\end{equation}
where 
\begin{equation}
l(\{{\bf I-P}\}f) \eqdef -\{v\cdot \nabla _x+L\}\{{\bf I-P}\}f.  \label{l}
\end{equation}
Notice that we have isolated the time derivative of the microscopic part.

To derive the macroscopic equations for ${\bf P}f$'s coefficients $a^f(t,x)$, $b^f_i(t,x)$ and 
$c^f(t,x)$, we
 use (\ref{hydro}) to expand entries of left hand side of (\ref{macro}) as 
\begin{gather*}
\sum_{i=1}^3\left\{ v_i\partial_i c|v|^2+\{\partial_t c
+
\partial_i b_i\}v_i^2+
\{\partial_t  b_i+\partial_i a\}v_i\right\} \sqrt{\mu 
}
\\
+\sum_{i=1}^3\sum_{j>i}\{\partial_i b_j+\partial_j b_i\}v_iv_j ~ \sqrt{\mu}
+
\partial_t  a
~\sqrt{\mu}, 
\end{gather*}
where  $\partial_i=\partial _{x_i}$ above. For fixed ($t,x),$ this is an expansion of the left hand side of 
(\ref{macro}) with respect to the following basis, $\{ e_k \}_{k=1}^{13}$, which consists of
\begin{gather}
\left( v_i|v|^2\sqrt{\mu } \right)_{1\le i\le 3}, ~ 
\left(v_i^2\sqrt{\mu } \right)_{1\le i\le 3},  ~
\left(v_iv_j\sqrt{\mu } \right)_{1\le i<j\le 3}, ~
\left(v_i\sqrt{\mu } \right)_{1\le i\le 3}, ~
\sqrt{\mu }.  
\label{base}
\end{gather}
From here one obtains the so-called
macroscopic equations 
\begin{eqnarray}
\nabla _xc&=& -\partial_t r_c+ l_c+\Gamma_c  \label{c} \\
\partial_t c+\partial_i b_i &=& -\partial_t r_i+ l_i+\Gamma_i  \label{bi} \\
\partial_i b_j+\partial_j b_i &=& -\partial_t r_{ij}+ l_{ij}+\Gamma_{ij} \quad (i\neq j)  \label{bij} \\
\partial_t b_i+\partial _i a &=& -\partial_t r_{bi}+ l_{bi}+\Gamma_{bi}
\label{ai} \\
\partial_t a &=& -\partial_t r_a+ l_a+\Gamma_a.  \label{adot}
\end{eqnarray}
For notational convenience we define the index set to be
$$
\mathcal{M} \eqdef \left\{c, ~i,~ \left( ij \right)_{i \ne j}, ~bi,~ a\left| ~ i, j = 1,2,3\right. \right\}.
$$
This set $\mathcal{M}$ is just the collection of all indices in the macroscopic equations.  Then for $\ell \in  \mathcal{M}$ we have that 
each  $l_\ell (t,x)$
are the coefficients of 
$l(\{{\bf I-P}\}f)$ with respect to the elements of \eqref{base}; similarly each 
$\Gamma_\ell(t,x)$ 
and
$r_\ell(t,x)$
are the coefficients of $\Gamma(f,f)$ and $\{{\bf I-P}\}f$ respectively.  Precisely, each element $r_\ell$ can be expressed as 
$$
r_\ell = \sum_{k=1}^{13} C_k^\ell \langle \{{\bf I-P}\}f, e_k \rangle.
$$
All of the constants $C_k^\ell$ above can be computed explicitly
%\cite{MR0033674,MR0135535} 
although we do not give their precise form herein.  Each of the terms $l_\ell$ and $\Gamma_\ell$ can be computed similarly.  

The second set of equations we consider are the local conservation laws satisfied by $(a^f,b^f,c^f)$.  To derive these we 
 multiply \eqref{Boltz} by the collision invariants $\mathcal{N}$ in \eqref{null}
and integrate only in the velocity variables to obtain 
%\begin{eqnarray*}
%% \nonumber to remove numbering (before each equation)
%  \partial_t\int_{\R^3} dv ~ F+\nabla_x\cdot\int_{\R^3} dv ~ v ~ F &=& 0,
%  \\
%  \partial_t\int_{\R^3} dv ~ v ~ F
%  +
%  \nabla_x\cdot\int_{\R^3} dv ~  v\otimes v ~F
%&=&0,
%\\
%  \partial_t\int_{\R^3} dv ~ |v|^2 ~ F 
%  +
%  \nabla_x\cdot\int_{\R^3} dv ~  |v|^2v ~
%  F
%&=&0.
%\end{eqnarray*}
%Plugging $F=\mu+ \sqrt{\mu}{\bf P} f+ \sqrt{\mu}\{{\bf I - P}\}f$ into the
%above equations 
%gives
\begin{eqnarray*}
  \partial_t (a^f+3c^f)+\nabla_x\cdot b^f &=& 0
  \\
  \partial_t b^f+\nabla_x (a^f+5c^f) &=& - \nabla_x\cdot \langle
  v\otimes v\sqrt{\mu},\{{\bf I - P}\}f\rangle
  \\
   \partial_t(3a^f+15c^f)+5\nabla_x\cdot b^f 
   &=& 
   - \nabla_x\cdot  \langle |v|^2v\sqrt{\mu},\{{\bf I - P}\}f\rangle.
\end{eqnarray*}
Above we have used the moment values of the normalized
global Maxwellian $\mu$: 
\begin{eqnarray*}
&&\langle 1, \mu\rangle=1,
\quad
\langle |v_j|^2, \mu\rangle=1,\ \ \langle |v|^2, \mu\rangle=3,
%\\ &&
\langle |v_j|^2|v_i|^2, \mu\rangle=1, \ \ j\neq i,\\
&&\langle |v_j|^4, \mu\rangle=3,\ \ \langle |v|^2|v_j|^2,
\mu\rangle=5,\ \ \langle |v|^4, \mu\rangle=15.
%\\
%&&\langle |v|^4|v_j|^2, \mu\rangle=35, \ \ \langle |v|^6,
%\mu\rangle=105.
\end{eqnarray*}
Comparing the first and third local conservation law results in
\begin{eqnarray}
  \partial_t a^f&=& 
    \frac{1}{2}\nabla_x \cdot \langle
    |v|^2v\sqrt{\mu},\{{\bf I - P}\}f\rangle
  \label{cl.0}
  \\
  \partial_t b^f+\nabla_x (a^f+5c^f) &=& - \nabla_x\cdot \langle
  v\otimes v\sqrt{\mu},\{{\bf I - P}\}f\rangle
  \label{cl.1}
  \\
    \partial_t c^f+\frac{1}{3}\nabla_x\cdot b^f  &=& -
    \frac{1}{6}\nabla_x \cdot \langle
    |v|^2v\sqrt{\mu},\{{\bf I - P}\}f\rangle.
    \label{cl.2}
\end{eqnarray}
These are the local conservation laws that we will study below.
For the rest of this section, we concentrate on a solution $f$ to the Boltzmann equation \eqref{Boltz}.

\begin{lemma}
\label{average}Let $f(t,x,v)$ be the local solution to the Boltzmann equation \eqref{Boltz}
%the solution constructed in
shown to exist in
Theorem \ref{local}
which satisfies 
\eqref{conservation}. Then we have 
\begin{equation}
 \int_{\mathbb{T}^3}~ dx~ a^f(t,x) = \int_{\mathbb{T}^3}~ dx~ b^f(t,x) = \int_{\mathbb{T}^3} ~ dx~ c^f(t,x)  
 =
 0,
 \notag
\end{equation}
where $a^f$, $b^f=[b_1,b_2,b_3],$ $c^f$
are defined in (\ref{hydro}).
\end{lemma}

The proof of this lemma 
  follows directly from the conservation of mass, momentum and energy \eqref{conservation}, using the cancellation that we just used in deriving the 
 conservation laws \eqref{cl.0}, \eqref{cl.1}, and \eqref{cl.2}.  In the following two Lemmas, we establish the required estimates on the linear microscopic piece and then we estimate the non-linear higher order term.

%\begin{proof}  
%
%
%This will follow from the conservation of mass, momentum and energy \eqref{conservation}. 
%Since ${ \bf  P  } $ is the orthogonal projection onto \eqref{null}, it holds that
%\begin{equation*}
%    \int_{\mathbb{R}^3} ~ dv ~ \begin{pmatrix}
%      1   \\      v  \\ |v|^2
%\end{pmatrix}
%\sqrt{\mu} ~ \{{ \bf I- P  } \} f =0,
%\end{equation*}
%Then for fixed $(t,x),$ by the decomposition (\ref{hydro}), we have 
%\begin{eqnarray*}
%\int_{\mathbb{R}^3} ~dv  ~\sqrt{\mu } ~ f &=&a^f(t,x)+3c^f(t,x), 
%\\
%\int_{\mathbb{R}^3} ~dv ~ v\sqrt{\mu } ~ f &=& b(t,x).
%\\
%\int_{\mathbb{R}^3} ~dv ~ |v|^2 ~\sqrt{\mu } ~ f
%&=&
%3a^f(t,x)+15c(t,x),
%\end{eqnarray*}
%%Here we have used
%%\begin{gather*}
%%\int_{\mathbb{R}^3} ~dv ~ \mu(v)
%%= 1,
%%\quad
%%\int_{\mathbb{R}^3} ~dv ~ |v_j|^2 ~\mu(v)
%%=
%%1, ~ (j = 1,2,3)
%%\\
%%\int_{\mathbb{R}^3} ~dv ~ |v|^2 ~\mu(v) = 3, 
%%\quad
%%\int_{\mathbb{R}^3} ~dv ~ |v|^4 ~\mu(v) = 15.
%%\end{gather*}
%Since 
%$
%\begin{pmatrix}
%1 & 3 \\ 
%3 & 15
%\end{pmatrix}
%$ is non-singular, with \eqref{conservation} we conclude by integrating over $x\in \mathbb{T}^3$.
%\end{proof} 

\begin{lemma}
\label{linear}
For any of the microscopic terms, $l_\ell$, from the macroscopic equations
\[
\sum_{\ell \in \mathcal{M}} \|l_\ell\|_{H^{N-1}_x}
\lesssim
\sum_{|\alpha| \le N}\|\{{\bf I-P}\} \partial^\alpha f\|_{L^2_{\gamma+2s}(\mathbb{T}^3_x\times\mathbb{R}^3_v)}.
\]
\end{lemma}

\begin{proof} 
Recall 
$\{ e_k \}_{k=1}^{13}$, the basis in (\ref
{base}).
For fixed $(t,x)$,  it suffices to estimate  the $H^{N-1}_x$ norm of
$
\langle l(\{{\bf I-P}\}f), e_k \rangle.
$
We  use (\ref{l}) to expand out
$$
\langle \partial^\alpha l(\{{\bf I-P}\}f), e_k \rangle
=
-\langle v\cdot \nabla _x (\{{\bf I-P}\}\partial^\alpha f), e_k \rangle
-\langle L (\{{\bf I-P}\}\partial^\alpha f), e_k \rangle.
$$
Now for any $|\alpha |\le N-1$
\begin{gather*}
\| \langle v\cdot \nabla _x (\{{\bf I-P}\}\partial^\alpha f), e_k \rangle\|_{L^2_x}^2 
\lesssim 
\int_{\mathbb{T}^3_x\times \mathbb{R}_{v}^3} ~ dx dv ~
|e_k(v)| ~ 
|v|^2 ~ |\{{\bf I-P}\}\nabla_x\partial ^\alpha f|^2
\\
\lesssim 
\|\{{\bf I-P}\}\nabla _x\partial^\alpha f\|^2_{L^2_{x,v}}.
\end{gather*}
Here we have used the exponential decay of $e_k(v)$.

It remains to estimate the linear operator $L$.  With the expression from \eqref{LinGam} and Lemma \ref{CompactEst}, we have the following
\begin{gather*}
\| \langle L (\{{\bf I-P}\}\partial^\alpha f), e_k \rangle\|_{L^2_x}^2 
\lesssim 
\left\| ~\nsm \{{\bf I-P}\}\partial^\alpha f \nsm_{\eta,\delta}~  \nsm M \nsm_{\delta,\eta} \right\|^2_{L^2_{x}}
\\
\lesssim 
\left\| \{{\bf I-P}\}\partial^\alpha f   \right\|^2_{L^2_{\gamma+2s}(\mathbb{T}^3_x\times\mathbb{R}^3_v)},
\quad 
(\text{taking}
\quad \eta = \delta = 1).
\end{gather*} 
This completes the proof of our estimates for the $l_\ell$.
\end{proof}

We now estimate coefficients of the higher order term $\Gamma(f,f)$:

\begin{lemma}
\label{high}Let (\ref{smallL}) be valid for some $M_0>0.$ Then  for $\NgE$ we have
\[
\sum_{\ell \in \mathcal{M}}
\|\Gamma_\ell\|_{H^{N}_x}
\lesssim
\sqrt{M_0}\sum_{|\alpha |\le N} \| \partial ^\alpha f\|_{L^2_{\gamma+2s}(\mathbb{T}^3_x\times\mathbb{R}^3_v)}.
\]
\end{lemma}

\begin{proof}
As in the proof of Lemma \ref{linear}, 
it will be sufficient to estimate the $H^{N}_x$ norm of $\langle \Gamma(f,f), e_k\rangle$.  We apply Lemma \ref{CompactEst}  to see that
\begin{gather*}
\left\| \langle \Gamma(f,f), e_k\rangle\right\|_{H^{N}_x}
\lesssim
\sum_{|\alpha |\le N}\sum_{\alpha_1 \le \alpha}
\left\|
 \nsm \partial^{\alpha - \alpha_1} f \nsm_{\delta,\eta} \nsm \partial^{\alpha_1} f \nsm_{\eta,\delta} 
\right\|_{L^2_x}.
\end{gather*}
We use \eqref{expand}, and take the supremum over the term with fewer derivatives to obtain
\begin{gather*}
\lesssim
 \| f\|_{L_v^2H^{N}_x}
\sum_{|\alpha |\le N} \| \partial ^\alpha f\|_{L^2_{\gamma+2s}}
\\
\lesssim
\sqrt{M_0}\sum_{|\alpha |\le N} \| \partial ^\alpha f\|_{L^2_{\gamma+2s}}.
\end{gather*}
The last inequalities follow from
 the Sobolev embedding $L^\infty_x \supset H_x^2$.
\end{proof}

We now prove the crucial positivity of $L$ for small solution $f(t,x,v)$ to the Boltzmann equation \eqref{Boltz}. The
conservation laws \eqref{conservation} will play an important
role. \\

\noindent {\it Proof of Theorem \ref{positive}}. 
We first of all notice from 
\eqref{hydro}
that
$$
 \| { \bf  P  } \partial^\alpha f (t) \|_{N^{s,\gamma}}^2
\lesssim
  \|\partial^\alpha a(t)\|^2_{L^2_x}
  +
  \|\partial^\alpha b (t)\|^2_{L^2_x}
  +
\|  \partial^\alpha c(t)\|^2_{L^2_x}.
$$
Thus it will be sufficient to bound each of the terms on the right side above by 
$ \| \{ {\bf I - P } \} \partial^\alpha f (t)\|_{N^{s,\gamma}}^2$ 
plus the time derivative of the interaction functional, which is defined in \eqref{mainINTERACTION}.  Indeed, our  proof is devoted to establishing the following
\begin{gather}
\notag
\|a (t)\|_{H^N_x}^2
+
\|b(t)\|_{H^N_x}^2
+
\|c(t) \|_{H^N_x}^2
\lesssim
\sum_{|\alpha |\le N}
\|\{{\bf I-P}\} \partial ^\alpha  f(t)\|_{L^2_{\gamma + 2s}}^2
\\
+
M_0\sum_{|\alpha |\le N}\|\partial ^\alpha f(t)\|_{L^2_{\gamma + 2s}}^2
+
\frac{d\mathcal{I}(t)}{dt}.
\label{claimH}
\end{gather}
Clearly the second term on the right above can be neglected %for $M_0$ sufficiently small 
because of
\begin{gather*}
\sum_{|\alpha |\le N}\|\partial^\alpha f(t)\|_{L^2_{\gamma + 2s}}^2
\lesssim
\sum_{|\alpha |\le N}
\|{\bf P}\partial ^\alpha f(t)\|_{L^2_{\gamma + 2s}}^2
+
\sum_{|\alpha |\le N}\|\{{\bf I-P}\}\partial
^\alpha f(t)\|_{L^2_{\gamma + 2s}}^2
\\
\lesssim
\left\{\|a(t) \|_{H^N_x}
+
\|b(t)\|_{H^N_x}
+
\|c(t) \|_{H^N_x}\right\}^2
+
\sum_{|\alpha |\le N}
\|\{{\bf I-P}\}\partial ^\alpha f(t)\|_{L^2_{\gamma + 2s}}^2.
\end{gather*}
We have used %the decomposition 
\eqref{hydro}.
Thus \eqref{claimH} will imply Theorem \ref{positive} when $M_0$ is sufficiently
small.

To prove (\ref{claimH}), we estimate each of $a$, $b$, and $c$ individually with spatial derivatives of order $0<|\alpha |\le N$.  Then at the end of the proof we estimate the pure $L^2_x$ norm of $a$, $b$, and $c$ in a uniform way.
  We first  estimate $a(t,x)$. 
Consider any $|\alpha |\le N-1$.
By taking $\partial_i \partial^\alpha$ of (\ref{ai}) and summing over $i$, we get 
\begin{equation}
-\Delta \partial ^\alpha a
=
\frac{d}{dt} \left( \nabla \cdot \partial ^\alpha b \right)
+
\sum_{i=1}^3\left(\partial_t \partial_i\partial ^\alpha r_{bi} -\partial_i\partial ^\alpha \{l_{bi}+\Gamma_{bi}\} \right).  \label{div}
\end{equation}
Multiply with $\partial ^\alpha a$ to (\ref{div}) and integrate over $dx$ to obtain
\begin{gather*}
\|\nabla \partial ^\alpha a\|^2_{L^2_x}
\le 
\frac{d}{dt} \int_{\mathbb{T}^3}  ~ dx ~ \left( \nabla \cdot \partial ^\alpha b \right)\partial^\alpha a(t,x)
+
\frac{d}{dt} \int_{\mathbb{T}^3}  ~ dx ~ \partial_i\partial ^\alpha r_{bi} ~
\partial^\alpha a(t,x)
\\
- \int_{\mathbb{T}^3}  ~ dx ~ \left( \nabla \cdot \partial ^\alpha b \right)
\partial_t \partial^\alpha a(t,x)
-
 \int_{\mathbb{T}^3}  ~ dx ~ \partial_i\partial ^\alpha r_{bi} ~
\partial_t \partial^\alpha a(t,x)
\\
+
\|\partial ^\alpha \{l_{bi}+\Gamma_{bi}\}\|_{L^2_x}
\|\nabla \partial ^\alpha a\|_{L^2_x}.
\end{gather*}
Above we implicitly sum over $i=1,2,3$. We define the interaction functional
$$
\mathcal{I}_a^\alpha(t) \eqdef 
\int_{\mathbb{T}^3}  ~ dx ~ \left( \nabla \cdot \partial ^\alpha b \right)\partial^\alpha a(t,x)
+
\sum_{i=1}^3
\int_{\mathbb{T}^3}  ~ dx ~ \partial_i\partial ^\alpha r_{bi} ~
\partial^\alpha a(t,x).
$$
We also use the local conservation law \eqref{cl.0}, to see that for any $\eta >0$, we have
\begin{gather*}
 \int_{\mathbb{T}^3}  ~ dx ~ \left\{ \left| \left( \nabla \cdot \partial ^\alpha b \right)
\partial_t \partial^\alpha a(t,x) \right|
+ \left| \partial_i\partial ^\alpha r_{bi} ~
\partial_t \partial^\alpha a(t,x) \right| \right\}
\\
\le \eta \| \nabla \cdot \partial ^\alpha b \|^2_{L^2_x} + C_\eta  \| \{ {\bf I - P} \} \nabla \partial^\alpha f \|^2_{L^2_{x,v}}.
\end{gather*}
We combine these last few estimates with Lemma \ref{linear} and \ref
{high} to see that  
\begin{gather}
\notag
\|\nabla \partial ^\alpha a\|^2_{L^2_x}
-
\eta \| \nabla \cdot \partial ^\alpha b \|^2_{L^2_x} 
\lesssim
C_\eta
\sum_{|\alpha| \le N}\|\{{\bf I-P}\} \partial^\alpha f\|^2_{L^2_{\gamma+2s}}
+
\frac{d \mathcal{I}_a^\alpha  }{dt}
\\
+
M_0\sum_{|\alpha |\le N} \| \partial ^\alpha f\|_{L^2_{\gamma+2s}}^2.
\label{mainAest}
\end{gather}
This will be our main estimate for $a(t,x)$ with derivatives.

Next we estimate $c(t,x)$ from  (\ref{c}), with $|\alpha |\le N-1$.   We notice that
\begin{gather*}
\|\nabla \partial^\alpha c\|^2_{L^2_x}
\le 
C\left\{\|\partial ^\alpha l_c\|^2_{L^2_x}+\|\partial^\alpha \Gamma_c\|^2_{L^2_x} \right\} 
- 
\frac{d}{dt}\int_{\mathbb{T}^3} ~ dx ~ \partial ^\alpha r_c(t,x) ~\cdot \nabla_x \partial ^\alpha c(t,x)
\\
+
\int_{\mathbb{T}^3} ~ dx ~ \nabla_x \cdot\partial ^\alpha r_c(t,x) ~  \partial ^\alpha \partial_t c(t,x).
\end{gather*}
We now define another interaction functional as
$$
\mathcal{I}_c^\alpha(t) \eqdef 
- \int_{\mathbb{T}^3} ~ dx ~ \partial ^\alpha r_c(t,x) ~ \cdot \nabla_x \partial ^\alpha c(t,x).
$$
Next we use the conservation law \eqref{cl.2} to obtain the following estimate
\begin{gather*}
\int_{\mathbb{T}^3} ~ dx ~ \left| \nabla_x \partial ^\alpha r_c(t,x) ~ \cdot \partial ^\alpha \partial_t c(t,x) \right|
%\\
\le \eta \| \nabla \cdot \partial ^\alpha b \|^2_{L^2_x} + C_\eta  \| \{ {\bf I - P} \} \nabla \partial^\alpha f \|^2_{L^2_{x,v}},
\end{gather*}
which holds for any $\eta >0$.
Combining these with Lemmas \ref{linear} and \ref
{high}, we see that
\begin{gather}
\notag
\|\nabla \partial ^\alpha c\|^2_{L^2_x}
-
\eta \| \nabla \cdot \partial ^\alpha b \|^2_{L^2_x} 
\lesssim
C_\eta
\sum_{|\alpha| \le N}\|\{{\bf I-P}\} \partial^\alpha f\|^2_{L^2_{\gamma+2s}}
+
\frac{d \mathcal{I}_c^\alpha  }{dt}
\\
+
M_0\sum_{|\alpha |\le N} \| \partial ^\alpha f\|_{L^2_{\gamma+2s}}^2.
\label{mainCest}
\end{gather}
This will be our main estimate for $c(t,x)$ with derivatives.

The last term to estimate with derivatives is $\nabla \partial ^\alpha b$.  
Suppose that 
$|\alpha |\le N-1$, take $\partial_j$ of (\ref{bi}) and (\ref{bij}) and  sum on $j$.  It was shown in a 
nontrivial calculation from \cite{MR2000470}, using 
 the
 elliptic structure of these equations
  use several symmetries, that
\begin{gather*}
\Delta \partial ^\alpha b_i
=
-\partial_i\partial_i\partial ^\alpha b_i
+
2\partial_i\partial ^\alpha l_i
+
2\partial_i\partial ^\alpha
\Gamma_i 
\\
+
\left\{
\sum_{j\neq i}
-\partial_i \partial^\alpha l_j-\partial_i\partial ^\alpha \Gamma_j
+
 \partial_j\partial ^\alpha l_{ij}
 +
 \partial_j\partial^\alpha \Gamma_{ij}
 -\partial_t \partial_{j} \partial^\alpha r_{ij}
\right\}.
\end{gather*}
We then multiply the whole expression by 
$\partial^\alpha b_i$ and integrate by parts to yield 
\begin{gather}
\label{sb}
\|\nabla \partial ^\alpha b_i\|^2
\le C\left\{\sum_{\ell \in \mathcal{M}} 
\|\partial ^\alpha l_\ell\|^2
+
\|\partial ^\alpha \Gamma_\ell\|^2
\right\} 
+
\sum_{j\neq i} \int_{\mathbb{T}^3} ~ dx ~ \partial_{j} \partial^\alpha r_{ij}
\partial_t \partial^\alpha b_i
\\
- 
\frac{d}{dt} \sum_{j\neq i} \int_{\mathbb{T}^3} ~ dx ~ \partial_{j} \partial^\alpha r_{ij}\partial^\alpha b_i.
\notag
\end{gather}
We  define the last component of the interaction functional as
$$
\mathcal{I}_b^\alpha(t) \eqdef 
- 
 \sum_{j\neq i} \int_{\mathbb{T}^3} ~ dx ~ \partial_{j} \partial^\alpha r_{ij}\partial^\alpha b_i.
$$
With the conservation law \eqref{cl.1}, we estimate the term with a time derivative as
\begin{gather*}
\sum_{j\neq i} \int_{\mathbb{T}^3} ~ dx ~  \left|  \partial_{j} \partial^\alpha r_{ij} 
\partial_t \partial^\alpha b_i(t,x) \right| 
%\\
\le \eta \left\{ \| \nabla  \partial ^\alpha a \|^2_{L^2_x} + \| \nabla  \partial ^\alpha c \|^2_{L^2_x}  \right\}
\\
+ C_\eta  \| \{ {\bf I - P} \} \nabla \partial^\alpha f \|^2_{L^2_{x,v}},
\end{gather*}
which once again holds for any $\eta >0$.
Combining these  last few estimates with Lemmas \ref{linear} and \ref
{high}, we obtain
\begin{gather}
\notag
\|\nabla \partial ^\alpha b\|^2_{L^2_x}
-
 \eta \left\{ \| \nabla  \partial ^\alpha a \|^2_{L^2_x} + \| \nabla  \partial ^\alpha c \|^2_{L^2_x}  \right\}
\lesssim
C_\eta
\sum_{|\alpha| \le N}\|\{{\bf I-P}\} \partial^\alpha f\|^2_{L^2_{\gamma+2s}}
+
\frac{d \mathcal{I}_b^\alpha  }{dt}
\\
+
M_0\sum_{|\alpha |\le N} \| \partial ^\alpha f\|_{L^2_{\gamma+2s}}^2.
\label{mainBest}
\end{gather}
This is our main estimate for $b(t,x)$ with derivatives.

Now, with $\mathcal{I}_a^\alpha(t)$, $\mathcal{I}_b^\alpha(t)$ and $\mathcal{I}_c^\alpha(t)$ defined just above, we define the total interaction functional as
\begin{gather}
\mathcal{I}(t) 
\eqdef 
\sum_{|\alpha|\le N-1}
\left\{
\mathcal{I}_a^\alpha(t)
+
\mathcal{I}_b^\alpha(t)
+
\mathcal{I}_c^\alpha(t)\right\}.
\label{mainINTERACTION}
\end{gather}
Choosing for instance $\eta = 1/8$ and collecting 
\eqref{mainAest}, \eqref{mainCest}, \eqref{mainBest}, we have established 
\begin{gather}
\notag
 \| \nabla  \partial ^\alpha a \|^2_{H^{N-1}_x}
 +
\|\nabla \partial ^\alpha b\|^2_{H^{N-1}_x}
+ \| \nabla  \partial ^\alpha c \|^2_{H^{N-1}_x} 
\lesssim
\sum_{|\alpha| \le N}\|\{{\bf I-P}\} \partial^\alpha f\|^2_{L^2_{\gamma+2s}}
+
\frac{d \mathcal{I}  }{dt}
\\
+
M_0\sum_{|\alpha |\le N} \| \partial ^\alpha f\|_{L^2_{\gamma+2s}}^2.
\notag
\end{gather}
To finish \eqref{claimH}, it remains to estimate the terms without derivatives.

With the Poincar\'{e} inequality and Lemma \ref{average}, 
$a$ itself is bounded by 
\begin{gather*}
\|a\|
\lesssim 
\|\nabla a\|+\left|\int_{\mathbb{T}^3}  ~ dx ~ a \right| 
=
\|\nabla a\|.
\end{gather*}
This is also bounded by the right side of  \eqref{claimH} by the last estimate above.
The estimates for $b_i(t,x)$ and $c(t,x)$ without derivatives are exactly the same.  This completes the main estimate \eqref{claimH} and the proof.
\qed \\

We are now ready to prove global in time solutions to  \eqref{Boltz} exist. 

\subsection{Global Existence}  With the coercivity estimate for non-linear solutions from Theorem \ref{positive}, we prove these solutions must be global with a continuity argument. \\

\noindent {\it Proof of Theorem \ref{mainGLOBAL}.}   We first fix $M_0\le 1$ such that both
Theorem \ref{local}  and \ref{positive} are valid.    
For any $C'>0$ we can choose a large constant $C_1>0$  such that
$$
\|f(t)\|^2_{L^2_vH^N_x} 
\le
\left( C_1 + 1 \right)\|f(t)\|^2_{L^2_vH^N_x} - C'\mathcal{I}(t)
\lesssim
 \|f(t)\|^2_{L^2_vH^N_x}. 
$$
Notice $C_1$  only depends upon the structure of the interaction functional and $C'$, but not on $f(t,x,v)$.
We then define the equivalent instant energy functional by
$$
\mathcal{E}(t)  \eqdef \left( C_1 + 1 \right)\|f(t)\|^2_{L^2_vH^N_x} - C' \mathcal{I}(t).
$$
Then 
$
\mathcal{E}(t)
\approx
\|f(t)\|^2_{L^2(\mathbb{R}^3_v; H^N(\mathbb{T}^3_x))}.
$
Now
choose
$
M_1\le \frac{M_0}{2}
$
and
consider initial data 
\[
\mathcal{E}(0)  \le M_1<M_0. 
\]
From Theorem \ref{local}, we may denote $T>0$ so that 
\[
T=\sup\{t\ge 0: \mathcal{E}(t)  \le 2 M_1\}>0. 
\]
We now  take the spatial derivatives of $\partial^{\alpha}$ 
of \eqref{Boltz} to obtain
\begin{gather}
\label{initial}
\frac{1}{2}\frac{d}{dt}\|f(t)\|^2_{L^2_vH^N_x} 
+
\sum_{|\alpha |\le N}\left( L\partial^{\alpha}f,\partial^{\alpha}f\right) 
=
\sum_{|\alpha |\le N}
\left( \partial^{\alpha} \Gamma (f,f), \partial^{\alpha} f \right) . 
\end{gather}
By Lemma \ref{NonLinEst} we have
$$
\sum_{|\alpha |\le N}
\left( \partial^{\alpha} \Gamma (f,f), \partial^{\alpha} f \right) 
\lesssim
\|f(t)\|_{L^2_vH^N_x} \mathcal{D}(t).
$$
Notice that for $0\le t\le T,$ by our choice of $M_1,$
\[
\mathcal{E}(t)  \le 2  M_1\le M_0. 
\]
Thus \eqref{smallL} is valid.  Now with Lemma \ref{lowerN} and then Theorem \ref{positive} we have
\begin{gather*}
\sum_{|\alpha |\le N}\left( L\partial^{\alpha}f,\partial^{\alpha}f\right) 
\ge \delta_0 \sum_{|\alpha |\le N} \| \{ {\bf I - P } \}\partial^{\alpha} h \|_{N^{s,\gamma}}^2
\\
\ge
\frac{\delta_0}{2} \sum_{|\alpha |\le N} \| \{ {\bf I - P } \}\partial^{\alpha} h \|_{N^{s,\gamma}}^2
+
\frac{\delta_0 \delta}{2} \sum_{|\alpha |\le N} \| {\bf P } \partial^{\alpha} h \|_{N^{s,\gamma}}^2
-\frac{\delta_0C_2 }{2} \frac{d\mathcal{I}(t)}{dt}. 
\end{gather*}
With $\tilde{\delta} \eqdef \min\left\{ \frac{\delta_0}{2} , \frac{\delta_0 \delta}{2} \right\}>0$
and
$C' \eqdef \frac{\delta_0C_2 }{2}>0$,
we  conclude that
\begin{gather}
\frac{d}{dt}\left\{\|f(t)\|^2_{L^2_vH^N_x} - C'\mathcal{I}(t)\right\} +\tilde{\delta}  \mathcal{D}(t)  
\lesssim
\|f(t)\|_{L^2_vH^N_x} \mathcal{D}(t).
\nonumber
\end{gather}
We multiply 
\eqref{initial} 
by $C_1$ and add it to this differential inequality to conclude
\begin{gather}
\frac{d\mathcal{E}(t)}{dt} +\tilde{\delta}   \mathcal{D}(t)  
\le C_*
\|f(t)\|_{L^2_vH^N_x} \mathcal{D}(t),
\quad 
C_*>0.
\nonumber
\end{gather}
In the last step we have used the positivity of $L \ge 0$ as shown in Lemma \ref{lowerN}.
Suppose
\[
M_1\eqdef \min \left\{ \frac{\tilde{\delta}^2}{8C_{*}^{2}},\frac{M_0}{2}
\right\}.
\]
We now use the definitions of $M_1$ and $T$
to obtain for $0\le t\le T$ that
\begin{gather}
\frac{d \mathcal{E}(t)}{dt} +\tilde{\delta}   \mathcal{D}(t)  
\le C_*
\|f(t)\|_{L^2_vH^N_x} \mathcal{D}(t)
\le
C_*\sqrt{2  M}\mathcal{D}(t)
\le
\frac{\tilde{\delta} }2 \mathcal{D}(t).  
\nonumber
\end{gather}
Therefore, an integration  over $0\le t\le \tau<T$ 
yields 
\begin{gather}
\mathcal{E}(\tau)
+
\frac{\tilde{\delta} }{2}\int_0^\tau ~ d\tau ~ \mathcal{D}(\tau)
\le 
 \mathcal{E}(0) \nonumber 
\le  M_1<2 M_1.  
\nonumber
\end{gather}
Since $\mathcal{E}(\tau)$ is continuous in $\tau$, 
$
\mathcal{E}(T)\le M_1
$ 
if $T<\infty .$ This is a contradiction to the definition of $T$, thus $T=\infty .$ 
The time decay follows from
$
\mathcal{D}(t)
\gtrsim
\| f(t) \|^2_{L^2_vH^N_x}.
$
 \qed

%\newpage
\section*{Appendix:  Carleman's representation and the dual formulation}
\label{secAPP:HSr}

In this appendix we develop two Carleman \cite{MR1555365} type representations which are used crucially in our main text.  
We consider the general expression
$$
\tilde{\mathcal{C}}(v_*) = \int_{\mathbb{R}^3}dv ~\Phi(|v-v_*|) 
\int_{\mathbb{S}^{2}} d\sigma ~
b\left(\ip{k}{\sigma}\right) ~
H(v,v_*, v^\prime, v^\prime_*),
$$
with $k = \frac{v - v_*}{|v - v_*|}$ and the usual post-collisional velocities  $(v^\prime, v^\prime_*)$ given by \eqref{sigma}.  The functions $b$ and $\Phi$ are generally given by \eqref{kernelP} and \eqref{kernelQ}.  For the purposes of deriving the expression in Proposition \ref{carlemanA} it suffices to suppose that both of these functions are smooth.  The general expressions can then be deduced from these formulas by the usual approximation procedures.
We have the following representation formula

%
%We have the following usual Carleman representation representation formula

%\begin{proposition} 
%\label{carlemanR}
%(Carleman's representation)    Let $H: \mathbb{R}^3 \times \mathbb{R}^3 \times \mathbb{R}^3 \times \mathbb{R}^3 \to \mathbb{R}$ be a smooth, rapidly decaying function at infinity.  Then we have
% \begin{gather*}
%\mathcal{C}(v) 
%= 
%4
%\int_{\mathbb{R}^3}dv_*' ~
%\int_{E^{v}_{v_*'}}d\pi_{v'} ~
%\frac{\Phi(|2v - v' - v'_*|) }{|v - v'_*|}
%\frac{b\left(\ang{\frac{2v - v' - v'_*}{|2v - v' - v'_*|}, \frac{v' - v'_*}{|v' - v'_*|} }\right) }{|2v - v' - v'_*|} ~
%H.
%  \end{gather*}
%  Above 
%  $
%  H = H(v, v'_*+ v' - v, v' , v_*' ),
%  $
%and
%  $
%  E^{v}_{v_*'}
%  $
%  is the hyperplane
%  $$
%  E^{v}_{v_*'} \eqdef \left\{ v'\in \mathbb{R}^3 : \ang{v' - v,v_*' -v} =0 \right\}.
%  $$
%Then $d\pi_{v'} $ denotes the Lebesgue measure on this hyperplane.
%\end{proposition}

\begin{proposition} 
\label{carlemanA}
Let $H: \mathbb{R}^3 \times \mathbb{R}^3 \times \mathbb{R}^3 \times \mathbb{R}^3 \to \mathbb{R}$ be a smooth, rapidly decaying function at infinity.  Then we have
 \begin{gather*}
\tilde{\mathcal{C}}(v_*) 
= 
4
\int_{\mathbb{R}^3}dv '~
\int_{E^{v'}_{v_*}}d\pi_{v} ~
\frac{\Phi(|v-v_*|)  }{|v' - v_*|}
\frac{b\left(\ang{\frac{v-v_*}{|v-v_*|}, \frac{2v' - v - v_*}{|2v' - v - v_*|} } \right)}{|v - v_*|}
H.
  \end{gather*}
  Above 
  $
  H = H(v, v_*, v' , v+v_*- v')
$
and
  $
  E^{v'}_{v_*}
  $
  is the hyperplane
  $$
  E^{v'}_{v_*} \eqdef \left\{ v'\in \mathbb{R}^3 : \ang{v_*-v', v - v'} =0 \right\}.
  $$
Then $d\pi_{v} $ denotes the Lebesgue measure on this hyperplane.
\end{proposition}

We also derive a Carleman representation for 
$$
\mathcal{C}(v) = \int_{\mathbb{R}^3}dv_* ~\Phi(|v-v_*|) 
\int_{\mathbb{S}^{2}} d\sigma ~
b\left(\ip{k}{ \sigma}\right) ~
H(v, v_*, v^\prime, v^\prime_*),
$$
with the same notation and the same comments as in the last case.

%We will also find useful the following additional variant

\begin{proposition} 
\label{carlemanV2}
  Let $H: \mathbb{R}^3 \times \mathbb{R}^3 \times \mathbb{R}^3 \times \mathbb{R}^3 \to \mathbb{R}$ be a smooth, rapidly decaying function at infinity.  Then we have
 \begin{gather*}
\mathcal{C}(v) 
= 
4
\int_{\mathbb{R}^3}dv' ~
\int_{E^{v}_{v'}}d\pi_{v'_*} ~
\frac{\Phi(|2v - v' - v'_*|) }{|v - v'|}
\frac{b\left(\ang{\frac{2v - v' - v'_*}{|2v - v' - v'_*|}, \frac{v' - v'_*}{|v' - v'_*|} }\right) }{|v' - v'_*|} ~
H.
  \end{gather*}
  Above 
  $
  H = H(v, v'_*+ v' - v, v' , v_*' ),
  $
and
  $
  E^{v}_{v'}
  $
  is the hyperplane
  $$
  E^{v}_{v'} \eqdef \left\{ v'_*\in \mathbb{R}^3 : \ang{v' - v, v_*' -v} =0 \right\}.
  $$
Then $d\pi_{v'} $ denotes the Lebesgue measure on this hyperplane.
\end{proposition}

Our expressions above may be at some degree of variance from the usual Carleman representation, however they are of the same form and derived in the same way; a clear proof can be found in \cite{09-GPV}.  
With these expressions we will derive a Dual Representation for the non-linear operator \eqref{gamma0}.

\subsection*{Dual Representation}
We initially suppose that $\int_{\sph} d \sigma ~ |b( \ang{k, \sigma} )| < \infty$ and that the kernel $b$ has mean zero, i.e., $\int_{\sph} d \sigma ~ b(\ang{k,\sigma}) = 0$. 
Then after the pre-post change of variables we can express \eqref{gamma0} as
\begin{gather*}
 \ang{\Gamma(g,h),f}= \int_{\R^3} \! \! dv  \int_{\R^3} \! \! dv_* \int_{\sph} \! \! d \sigma
 ~ \Phi(|v-v_*|) b \left( \ang{k, \sigma} \right) g_* h \left(M_*' f' - M_* f \right) \nonumber 
  \\
  = \int_{\R^3} dv ~ \int_{\R^3} dv_* ~ \int_{\sph}  d \sigma ~ \Phi(|v-v_*|) b \left( \ang{k, \sigma} \right) g_* h M_*' f'. \nonumber 
 \end{gather*}
This follows from the vanishing of $\int_{\sph} b( \ang{k, \sigma} ) d \sigma$.  With Proposition \ref{carlemanA}, this is
 \begin{gather*}
 = 
 4 \int_{\R^3} \!  \!dv_* \! \int_{\R^3} \! \! dv' \!   \int_{E^{v'}_{v_*}}d\pi_{v} ~
\Phi(|v-v_*|)  
\frac{b\left(\ang{\frac{v-v_*}{|v-v_*|}, \frac{2v' - v - v_*}{|2v' - v - v_*|} } \right)}{|v' - v_*|~|v - v_*|} g_* h M_*' f'. 
 \nonumber
\end{gather*}
As usual above $d\pi_{v}$ is Lebesgue measure on the two-dimensional plane $E_{v_*}^{v'}$ passing through $v'$ with normal $v' - v_*$, and of course $v$ is the variable of integration.  In the above formulas, we take $M'_* = M(v+v_* - v')$.  From the identity
\[ \ang{\frac{v-v_*}{|v-v_*|}, \frac{2v' - v - v_*}{|2v' - v - v_*|} } = \frac{|v' - v_*|^2 - |v - v'|^2}{|v - v'|^2 + |v' - v_*|^2}, \]
we observe that
\begin{align*}
\int_{E_{v_*}^{v'}} d\pi_{v} ~ &  b \left( \ang{\frac{v-v_*}{|v-v_*|}, \frac{2v' - v - v_*}{|2v' - v - v_*|} } \right) \frac{|v'-v_*|^2}{|v-v_*|^4} \\
& = \int_{0}^{2 \pi} d \theta \int_0^\infty r \ dr \ b \left( \frac{|v'-v_*|^2 - r^2}{|v' - v_*|^2 + r^2} \right) \frac{|v' - v_*|^2}{(r^2 + |v'-v_*|^2)^2} = 0,
\end{align*}
by a change of variables since $\int_{-1}^1 dt \ b(t) = 0$ and 
\[ \frac{d}{dr} \left[ \frac{|v'-v_*|^2 - r^2}{|v' - v_*|^2 + r^2} \right] = %\frac{-2r (r^2 + |v'-v_*|^2) - 2r (-r^2 + |v' - v_*|^2)}{(r^2 + |v' - v_*|^2)^2} = 
\frac{-4 r |v' - v_*|^2}{(r^2 + |v' - v_*|^2)^2}. \] 
In particular, this implies
\[ 
\int_{E_{v_*}^{v'}} \frac{d \pi_v \ \Phi(|v'-v_*|)}{|v'-v_*| |v - v_*|}  ~    b \left( \ang{\frac{v-v_*}{|v-v_*|}, \frac{2v' - v - v_*}{|2v' - v - v_*|} } \right) \frac{|v'-v_*|^3}{|v-v_*|^3} ~ g_* h' M_* f' = 0. 
\]
We subtract this expression from the Carleman representation just written for $\ang{\Gamma(g,h),f}$, to see that $\ang{\Gamma(g,h),f}$ must also equal
\begin{equation*}
\begin{split}
%
%\ang{\Gamma(g,h),f}
%=
4  \int_{\R^3}   dv_* ~ \int_{\R^3}  dv' ~ \int_{E_{v_*}^{v'}} \! \!  &  \frac{d \pi_v}{|v' - v_*| |v-v_*|}  ~ 
b \left( \ang{\frac{v-v_*}{|v-v_*|}, \frac{2v' - v - v_*}{|2v' - v - v_*|} } \right)
%b(v-v_*,2v' - v -v_*)  
~  g_* f' 
\\
& \times
\left( \Phi(v-v_*) h M_*'    -   \Phi(v'-v_*) \frac{|v' - v_*|^3}{|v-v_*|^3} h' M_* \right). 
\end{split}
% \label{dualrepeq}
\end{equation*}
This will be called the ``dual representation.''

The claim is now that this representation holds even when the mean value of the singular kernel $b(\ang{k, \sigma})$ from \eqref{kernelQ}  is not zero.  To see this claim, suppose $b$ integrable but without mean zero.  Then define
\[ 
b_\epsilon(t) = b(t) - \epsilon^{-1} \ind_{[1-\epsilon,1]}(t) \int_{-1}^1 b(t) dt. \
\]
As a function on $\sph$, $b_\epsilon$ will clearly have a vanishing integral.  However, given arbitrary $f$, $g$ and $h$ which are Schwartz functions, it is not hard to see that
\[ 
\ang{\Gamma(g,h),f} - \ang{\Gamma_\epsilon(g,h),f} \rightarrow 0,
\quad
\epsilon \rightarrow 0. 
\]
Above $\Gamma_\epsilon$ is the non-linear term \eqref{gamma0} formed with $b_\epsilon(t)$ in place of $b(t)$.  This convergence holds %in particular 
because  cancellation guarantees that the integrand vanishes on the set defined by $\ang{k,\sigma} = 1$.  Moreover, an additional cutoff argument shows that the equality also holds provided that $b(t)$ satisfies \eqref{kernelQ}; %%  or \eqref{bCOND2};
%notice that 
the higher-order cancellation is preserved because $\frac{|v'-v_*|}{|v-v_*|}$ possesses radial symmetry in $v-v'$.

The  ``dual representation'' deserves its name
because if one defines
\begin{align*}
T_g f (v) & \eqdef \int_{\R^3} dv_* \int_{\sph} d \sigma B g_* \left(M_*' f' - M_* f \right) 
\\
T_g^* h (v') & \eqdef 4 \int_{\R^3} dv_* \int_{E_{v_*}^{v'}} \frac{d \pi_{v}}{|v-v_*| |v'-v_*|}  b \left( \ang{\frac{v-v_*}{|v-v_*|}, \frac{2v' - v - v_*}{|2v' - v - v_*|} } \right)  g_*
\\
& \hspace{100pt}  \times \left( \Phi(v-v_*) M_*' h   -   {\Phi(v'-v_*)} \frac{|v' - v_*|^3}{|v-v_*|^3} M_* h' \right),
\end{align*} 
then 
\begin{equation}
\label{3dualZ}
\ang{\Gamma(g,h),f}
=
\ang{ T_g f ,h} 
= 
\ang{ f, T_g^* h }. 
\end{equation}
Note that the last inner product above represents an integration over $dv'$ whereas the first two inner products above represent integrations over $dv$.

The advantage of this representation is that $T_g f$ and $T_g^* h$ both depend on $g$ in a fairly elementary way.  
This will allow, for example, the trilinear form $\ang{\Gamma(g,h),f}$ to be understood as a superposition of bilinear forms in $h$ and $f$.

%\subsection*{Acknowledgments} 
%The research of PTG was partially supported by NSF grant DMS-0850791. 
%The research of  RMS during the course of this work has been partly supported by the NSF, in particular by the previous NSF fellowship DMS-0602513 and the current NSF grant DMS-0901463. 

\end{document}